\pdfoutput=1
\documentclass[reqno]{amsart}
\usepackage{hyperref, bbm, amssymb, mathtools, url, enumerate}
\usepackage{amsmath}
\usepackage{tikz}
\usepackage{tikz-cd}
\usepackage[scr]{rsfso}
\usepackage{bm}
\usepackage{amssymb}
\usetikzlibrary{matrix,arrows,decorations.pathmorphing, cd}
\usepackage[capitalise]{cleveref}
\usepackage{rotate}
\usepackage[alphabetic]{amsrefs}
\usepackage{microtype}
\usepackage{aliascnt}

\DeclareMathAlphabet{\mymathbb}{U}{BOONDOX-ds}{m}{n}
\def\twocell[#1]{\arrow[#1, dash, phantom, "\Rightarrow"{scale=1.125, yshift=-.4pt, description, allow upside down, sloped, inner sep=0pt}]}

\tikzset{curve/.style={settings={#1},to path={(\tikztostart)
			.. controls ($(\tikztostart)!\pv{pos}!(\tikztotarget)!\pv{height}!270:(\tikztotarget)$)
			and ($(\tikztostart)!1-\pv{pos}!(\tikztotarget)!\pv{height}!270:(\tikztotarget)$)
			.. (\tikztotarget)\tikztonodes}},
	settings/.code={\tikzset{quiver/.cd,#1}
		\def\pv##1{\pgfkeysvalueof{/tikz/quiver/##1}}},
	quiver/.cd,pos/.initial=0.35,height/.initial=0}

\pgfdeclarelayer{bg}
\pgfsetlayers{bg,main}
\usetikzlibrary{calc}

\newcommand{\mynewtheorem}[2]{\newaliascnt{#1}{theorem}\newtheorem{#1}[#1]{#2}\aliascntresetthe{#1}}
\newtheorem{theorem}{Theorem}[section]
\mynewtheorem{conjecture}{Conjecture}
\mynewtheorem{corollary}{Corollary}
\mynewtheorem{lemma}{Lemma}
\mynewtheorem{proposition}{Proposition}

\theoremstyle{definition}
\newtheorem*{claim*}{Claim}
\mynewtheorem{construction}{Construction}
\mynewtheorem{convention}{Convention}
\mynewtheorem{definition}{Definition}
\newtheorem*{introdefinition}{Definition}
\mynewtheorem{discussion}{Discussion}
\mynewtheorem{example}{Example}
\mynewtheorem{hypothesis}{Hypothesis}
\newtheorem*{notation}{Notation}
\mynewtheorem{observation}{Observation}
\mynewtheorem{question}{Question}
\mynewtheorem{remark}{Remark}
\newtheorem*{remark*}{Remark}
\mynewtheorem{warn}{Warning}

\newtheorem{introthm}{Theorem}
\newaliascnt{introcorollary}{introthm}
\newtheorem{introcorollary}[introcorollary]{Corollary}
\aliascntresetthe{introcorollary}

\newcommand{\qednow}{\pushQED{\qed}\qedhere\popQED}

\newcommand{\Aa}{{\mathcal{A}}}
\newcommand{\Bb}{{\mathcal{B}}}
\newcommand{\Cc}{{\mathcal{C}}}
\newcommand{\Dd}{{\mathcal{D}}}
\newcommand{\Ee}{{\mathcal{E}}}
\newcommand{\Ff}{{\mathcal{F}}}
\newcommand{\Mm}{{\mathcal{M}}}
\newcommand{\myMm}{\,\widehat{\!\smash{\Mm}\vphantom{\textup{t}}}}

\newcommand{\Ii}{{\mathcal{I}}}
\newcommand{\Kk}{{\mathcal{K}}}
\newcommand{\Qq}{{\mathcal Q}}
\newcommand{\Tt}{{\mathcal{T}}}

\newcommand{\Vv}{{\mathcal{V}}}
\newcommand{\Ww}{{\mathcal{W}}}
\newcommand{\Xx}{{\mathcal X}}

\newcommand{\Z} {{\mathbb{Z}}}

\renewcommand{\phi}{\varphi}
\renewcommand{\epsilon}{\varepsilon}

\DeclareMathOperator{\Sp}{Sp}
\DeclareMathOperator{\Spc}{Spc}
\DeclareMathOperator{\Cat}{Cat}
\DeclareMathOperator{\CAT}{CAT}

\newcommand{\PrL}{\textup{Pr}^{\textup{L}}}

\DeclareMathOperator{\Hom}{Hom}

\DeclareMathOperator{\Fun}{Fun}
\DeclareMathOperator{\PSh}{PSh}
\DeclareMathOperator{\Shv}{Shv}
\DeclareMathOperator{\Ar}{Ar}

\DeclareMathOperator{\CMon}{CMon}

\DeclareMathOperator{\AdTrip}{AdTrip}

\DeclareMathOperator{\Span}{Span}
\DeclareMathOperator{\Rep}{Rep}

\newcommand{\catop}{^{\mathrm{op}}}
\newcommand{\op}{{\textup{op}}}

\newcommand{\Seg}{{\textup{Seg}}}
\renewcommand{\smallint}{{\textstyle\int}}

\DeclareMathOperator{\core}{\iota}
\newcommand{\colim}{\textup{colim}}
\newcommand{\const}{\textup{const}}
\newcommand{\id}{\textup{id}}
\newcommand{\pr}{\textup{pr}}

\newcommand{\BC}{\textup{BC}}

\newcommand{\Orb}{{\vphantom{\textup{t}}\smash{\textup{Orb}}}}
\newcommand{\Glo}{\textup{Glo}}
\newcommand{\res}{\textup{res}}
\newcommand{\ind}{\textup{ind}}
\newcommand{\coind}{\textup{coind}}
\newcommand{\Nm}{\mathop{\vphantom{t}\smash{\textup{Nm}}}}
\DeclareMathOperator{\Nmadj}{\widetilde{\Nm}}

\renewcommand{\sqcup}{\amalg}
\renewcommand{\bigsqcup}{\coprod}

\newcommand{\ulhelper}[2]{\underline{\setbox0=\hbox{$#1#2$}\dp0=1pt \box0\relax}}
\newcommand{\ul}[1]{{\mathpalette\ulhelper{#1}}\hbox{\rule[-2pt]{0pt}{0pt}}}
\newcommand{\ulFun}{\ul{\Fun}}
\newcommand{\ulbbU}[1]{\ul{#1}}

\newcommand{\finSets}{\mathbb{F}}
\newcommand{\finTsets}{{\finSets_{T}}}
\newcommand{\finPsets}{\finSets_{T}^{P}}

\newcommand{\ulfinPsets}{\ul{\finSets}_T^{P}}

\newcommand{\Fin}{\textup{Fin}}
\newcommand{\FinGrpd}{\mathscr F}
\newcommand{\FinGrpdfaith}{\FinGrpd_{\raise2pt\hbox{$\scriptstyle\dagger$}}}
\newcommand{\Ab}{\textup{Ab}}
\newcommand{\Mod}{{\textup{Mod}}}

\newcommand{\ulSpan}{\ul{\textup{Span}}}
\newcommand{\coSeg}{\mathrm{coSeg}}

\newcommand{\blank}{{\textup{--}}}

\newcommand{\tcatUn}[1]{\mathop{\hfuzz=10pt\hbox to 0pt{$\textstyle\bm\int$}\kern.3pt\raise.2pt\hbox to 0pt{$\textstyle\bm\int$}\lower.2pt\hbox to 0pt{$\textstyle\bm\int$}\kern.3pt\hbox to 0pt{$\textstyle\bm\int$}\kern-.1pt\raise.1pt\hbox{\color{white}$\textstyle\int$}}}
\newcommand{\GammaS}{{\Gamma\kern-1.5pt\mathscr S}}
\newcommand{\mySp}{{\mathscr S\kern-2ptp}}
\newcommand{\mathscrGr}{{\mathscr G\kern-1.25ptr}}
\newcommand{\Mack}{{\textup{Mack}}}

\newcommand{\pt}{\textup{pt}}

\DeclareMathOperator{\fib}{fib}

\newcommand{\ev}{\mathrm{ev}}
\newcommand{\fw}{\mathrm{fw}}

\newcommand{\loc}{{\textup{loc}}}

\newcommand{\qfgen}{\mathrm{qfgen}}

\newcommand{\QFin}{\mathrm{QFin}}
\newcommand{\Set}{\mathrm{Set}}

\newcommand{\Qprod}{{\textup{$\Qq$-$\times$}}}

\newcommand{\Qoplus}{{\textup{$\Qq$-$\oplus$}}}
\newcommand{\Qcoprod}{{\textup{$\Qq$-$\amalg$}}}

\newcommand{\CatBQProd}{\Cat(\Bb)^\Qprod}

\newcommand{\iso}{\xrightarrow{\;\smash{\raisebox{-0.5ex}{\ensuremath{\scriptstyle\sim}}}\;}}

\newcommand{\isoname}[1]{\xrightarrow[\;\smash{\raisebox{0.5ex}{\ensuremath{\scriptstyle\sim}}}\;]{#1}}

\tikzcdset{pullback/.style = {phantom, "\lrcorner"{very near start}}}

\newcommand\noloc{%
	\nobreak
	\mspace{6mu plus 1mu}
	{:}
	\nonscript\mkern-\thinmuskip
	\mathpunct{}
	\mspace{2mu}
}

\newcommand{\buildarrowfromtikz}[1]{\mathrel{\begin{tikzcd}[ampersand replacement=\&, cramped, cells={nodes={inner sep=0pt}}, column sep=small] \null\arrow[r, #1]\&\null
\end{tikzcd}}}
\renewcommand{\rightarrowtail}{\buildarrowfromtikz{tail}}
\renewcommand{\twoheadrightarrow}{\buildarrowfromtikz{->>}}

\title[Parametrized (higher) semiadditivity and the universality of spans]{Parametrized (higher) semiadditivity\\ and the universality of spans}

\author{Bastiaan Cnossen}
\address{B.C.: Fakultät für Mathematik, Universität Regensburg, 93040 Regensburg, Germany}
\author{Tobias Lenz}
\address{T.L.:{\hskip0pt minus 1pt} Mathematical{\hskip0pt minus 1pt} Institute,{\hskip0pt minus 1pt} University{\hskip0pt minus 1pt} of{\hskip0pt minus 1pt} Utrecht,{\hskip0pt minus 1pt} Budapestlaan{\hskip0pt minus 1pt} 6,{\hskip0pt minus 1pt} 3584{\hskip0pt minus 1pt} CD{\hskip0pt minus 1pt} Utrecht, The Netherlands \& Mathematisches Institut, Rheinische{\hskip0pt minus 1pt} Friedrich-Wilhelms-Universit\"at{\hskip0pt minus 1pt} Bonn,{\hskip0pt minus 1pt} Endenicher{\hskip0pt minus 1pt} Allee{\hskip0pt minus 1pt} 60,{\hskip0pt minus 1pt} 53115{\hskip0pt minus 1pt} Bonn,{\hskip0pt minus 1pt} Germany (\textit{current address})}
\author{Sil Linskens}
\address{S.L.:{\hskip0pt minus 1pt} Mathematisches{\hskip0pt minus 1pt} Institut,{\hskip0pt minus 1pt} Rheinische{\hskip0pt minus 1pt} Friedrich-Wilhelms-Universit\"at{\hskip0pt minus 1pt} Bonn,{\hskip0pt minus 1pt} Endenicher{\hskip0pt minus 1pt} Allee{\hskip0pt minus 1pt} 60,{\hskip0pt minus 1pt} 53115{\hskip0pt minus 1pt} Bonn,{\hskip0pt minus 1pt} Germany}
\begin{document}
\begingroup\parskip=0pt
	\begin{abstract}
		Semiadditivity of an $\infty$-category, i.e.\ the existence of biproducts, provides it with useful algebraic structure in the form of a canonical enrichment in commutative monoids. This ultimately comes from the fact that the $\infty$-category of commutative monoids is the universal semiadditive $\infty$-category equipped with a finite-product-preserving functor to spaces, or equivalently that the $(2,1)$-category of spans of finite sets is the universal semiadditive $\infty$-category. In this article, we prove a vast generalization of these facts in the context of parametrized semiadditivity, a notion we define using Hopkins--Lurie's framework of ambidexterity. This simultaneously generalizes a result of Harpaz for higher semiadditivity and a result of Nardin for equivariant semiadditivity. We deduce that every parametrized semiadditive $\infty$-category is canonically enriched in Mackey functors/sheaves with transfers.

		As an application, we reprove the Mackey functor description of global spectra first obtained by the second-named author and generalize it to $G$-global spectra. Moreover, we obtain universal characterizations of the $\infty$-categories of $\Z$-valued $G$-Mackey profunctors and of quasi-finitely genuine $G$-spectra as studied by Kaledin and Krause--McCandless--Nikolaus, respectively.
	\end{abstract}

	\maketitle
	\tableofcontents
\endgroup

	\section{Introduction}
	Various interesting categories\footnote{Throughout this article, we will say `category' for `$\infty$-category.'} arising throughout mathematics, like any abelian category, the stable homotopy category, or the derived category of a scheme, are \emph{semiadditive}, meaning that they admit both finite coproducts and finite products and that these two constructions agree. This simple \emph{property} provides the category with useful algebraic \emph{structure}. For example, every semiadditive category admits an enrichment in the category of commutative monoids: given two morphisms $f,g\colon X \to Y$, their sum $f+g$ is given as the composite
	\begin{equation}\label{eq:addition}
	f + g \colon X \xrightarrow{\Delta} X \times X \xrightarrow{f \times g} Y \times Y \cong Y\amalg Y \xrightarrow{\nabla} Y,
	\end{equation}
	where $\Delta = (\id_X,\id_X)$ denotes the diagonal and $\nabla = (\id_Y,\id_Y)$ denotes the fold map (`codiagonal'). Getting this structure for free is particularly helpful in higher categorical contexts, where specifying a commutative monoids requires an infinite tower of coherence data in addition to the map $(\ref{eq:addition})$.

	One convenient way to encode the higher coherences in a commutative monoid is via \textit{span categories}. Recall that for a category $\Aa$ with pullbacks there is a span category $\Span(\Aa)$: it has the same objects as $\Aa$, and a morphism $A$ to $B$ is given by a \textit{span} $A \xleftarrow{} C \xrightarrow{} B$; composition in $\Span(\Aa)$ is given via pullback. For a category $\Cc$ with finite products, a commutative monoid in $\Cc$ can then be defined as a product-preserving functor from the span category of finite sets to $\Cc$:
	\[
	\CMon(\Cc) \coloneqq \Fun^{\times}(\Span(\Fin),\Cc).
	\]
	Since products in $\Span(\Fin)$ are computed by taking disjoint unions of finite sets, such a commutative monoid $M\colon \Span(\Fin) \to \Cc$ must be given on objects by sending a finite set $S$ to the $S$-fold product $M(1)^S$ of the value of $M$ at the one-point set $1 \in \Fin$. For a finite set $S$, the span $S \xleftarrow{=} S \to 1$ determines an `addition map' $M(1)^S \to M(1), (x_s)_{s\in S} \mapsto \sum_{s \in S} x_s$, and functoriality of $M$ encodes the higher coherences that these addition maps are supposed to satisfy, including unitality, associativity, and commutativity.

	As evidenced by the enrichment \eqref{eq:addition}, the notion of semiadditivity of categories is intimately tied to the theory of commutative monoids. At the heart of this connection lies the fact that the category $\CMon(\Cc)$ is the \textit{semiadditive completion} of $\Cc$, in the sense that the forgetful functor $\CMon(\Cc) \to \Cc$ given by evaluation at the one-point set is universal among product-preserving functors from a semiadditive category to $\Cc$. This can in turn be traced back to a universal property of the span category $\Span(\Fin)$ itself: it is the \textit{free semiadditive category} generated by the point.

	The main contribution of this article is to formulate and prove analogues of these statements in a general context of \emph{parametrized semiadditivity}. Before stating our results, let us illustrate how these parametrized variations of semiadditivity naturally come up in examples.

	(1) \textit{Equivariant semiadditivity} \cite{exposeI, nardin2016exposeIV, balmerAmbrogio_Mackey, QuigleyShay2021Tate, CLL_Global}:

	Given a finite group $G$, one obtains an equivariant analogue of commutative monoids by replacing $\Fin$ by the category $\Fin_G$ of finite $G$-sets: for a category $\Cc$ with finite products, we define the category of \textit{$\Cc$-valued $G$-Mackey functors\footnote{Note that we do not assume a $G$-Mackey functor to take values in commutative \textit{groups} in $\Cc$.}} as
	\[
		\Mack_G(\Cc) \coloneqq \Fun^{\times}(\Span(\Fin_G), \Cc).
	\]
	In \cite{exposeI}, an accompanying definition of \textit{$G$-semiadditivity} was suggested: instead of asking $S$-indexed coproducts to agree with $S$-indexed products merely for all finite sets $S$, like in ordinary semiadditivity, one would demand this more generally whenever $S$ is a \textit{finite $G$-set}. Formulating this requires the language of parametrized category theory: a \emph{$G$-category} is a product-preserving functor $\Cc\colon\Fin_G^\op\to\Cat$, or equivalently a contravariant functor on the $1$-category of transitive $G$-sets. More informally, this amounts to specifying a category $\Cc_H\coloneqq\Cc(G/H)$ for every subgroup $H\leqslant G$, together with a restriction map $\textup{res}^H_K\colon\Cc_H\to\Cc_K$ for all $K\leqslant H\leqslant G$.\footnote{There is some further structure related to conjugation by elements of $G$, which we are ignoring here for the sake of exposition.} For such $G$-categories there are natural notions of `$G$-(co)limits,' and having colimits for the orbit $G/H$ amounts to the existence of a left adjoint $\ind_G^H$ to the restriction $\textup{res}^G_H$ enjoying certain properties, while the existence of limits over $G/H$ corresponds to the existence of a well-behaved right adjoint $\coind^G_H$. Accordingly, $G$-semiadditivy in particular amounts to an equivalence between induction and coinduction. Just like semiadditivity, this property is ubiquitous in equivariant mathematics:

	\begin{itemize}
		\item For finite groups $H \leqslant G$ and an $H$-representation $V$, there is an \textit{induced $G$-representation} $\ind^G_H(V) = \bigoplus_{[g] \in G/H} V$. The assignment $V \mapsto \ind^G_H(V)$ defines a functor $\ind^G_H\colon \Rep_H \to \Rep_G$ which is both left and right adjoint to the restriction functor $\res^G_H\colon \Rep_G \to \Rep_H$.
		\item In stable equivariant homotopy theory, the restriction functor $\res^G_H\colon \Sp_G \to \Sp_H$ from genuine $G$-spectra to genuine $H$-spectra admits left and right adjoints $\ind^G_H$ and $\coind^G_H$, respectively. There is a canonical natural equivalence
		\[
		\ind^G_H \iso \coind^G_H,
		\]
		due to Wirthmüller \cite{wirthmuller1974equivariant}.
		\item Similar natural equivalences exist for equivariant KK-theory \cite{MeyerNest2006, Meyer2008, BEL2023Kasparov, CLL_Adams}, proper equivariant homotopy theory \cite{DHLPS2019Proper}, and global homotopy theory \cite{g-global,CLL_Global}.
	\end{itemize}
	The connection between equivariant semiadditivity and Mackey functors was investigated by Nardin \cite{nardin2016exposeIV}, who showed that the $G$-category of spans of finite $G$-sets is the \textit{free $G$-semiadditive $G$-category}.

	(2) \textit{Higher semiadditivity} \cite{hopkins2013ambidexterity, harpaz2020ambidexterity,CSY20}:

	For a finite group $G$ and a $G$-representation $V$ over a field $k$ such that the group order $\lvert G \rvert$ is invertible in $k$, there is a canonical isomorphism
	\[
	\Nm\nolimits_G\colon V/G \iso V^G, \qquad [v] \mapsto \sum\nolimits_{g} gv
	\]
	between the vector space $V/G$ of $G$-coinvariants and the space $V^G$ of $G$-invariants. We may think of this isomorphism as a `higher' version of semiadditivity, where limits and colimits indexed by finite sets are replaced by limits and colimits indeed by \textit{finite groupoids}. This phenomenon occurs more generally in chromatic homotopy theory: it was shown by \cite{GreenleesSadofsky, HoveySadofsky} that for a finite group $G$ and a $K(n)$-local spectrum $X$ with $G$-action, a certain \emph{norm map} $\Nm\nolimits_G\colon X_{hG} \to X^{hG}$ is an equivalence. In fact, Hopkins and Lurie \cite{hopkins2013ambidexterity} observed that this phenomenon occurs for a much larger class of indexing spaces:

	Recall that an \emph{$m$-finite space} ($m\ge-2$) is an $m$-truncated space with finitely many path components, each of which has finite homotopy groups; for example, the $0$-finite spaces are precisely the finite sets, while the $1$-finite spaces are precisely finite groupoids. Hopkins and Lurie then proved that the category of $K(n)$-local spectra is \emph{$m$-semiadditive} for all $m$ in the sense that there is a preferred equivalence
	\[
	\Nm\nolimits_A\colon \colim_A X \iso \lim\nolimits_A X
	\]
	for every $m$-finite space $A$ and every $A$-indexed $K(n)$-local spectrum $X\colon A \to \Sp_{K(n)}$.

	The accompanying notion of higher commutative monoids was investigated by Harpaz \cite{harpaz2020ambidexterity}: for a category $\Cc$ admitting limits over $m$-finite spaces, one defines the category $\CMon_m(\Cc)$ of \textit{$m$-commutative monoids} as
	\[
		\CMon_m(\Cc) \coloneqq \Fun^{m\text{-fin}}(\Span(\Spc_m), \Cc),
	\]
	where $\Spc_m$ denotes the category of $m$-finite spaces. Harpaz showed that category $\Span(\Spc_m)$ is the free $m$-semiadditive category on a single generator, and deduced that the forgetful functor $\CMon_m(\Cc) \to \Cc$, given by evaluation at the one-point space, exhibits $\CMon_m(\Cc)$ as the \textit{$m$-semiadditive completion} of $\Cc$, in the sense that it is terminal among functors that preserve $m$-finite limits from an $m$-semiadditive category to $\Cc$.

	(3) \textit{Ambidexterity in geometry}:

	Variations of semiadditivity also come up in geometric contexts. Consider for example a finite covering map $f\colon X \to Y$ of locally compact Hausdorff spaces, giving rise to a pullback functor $f^*\colon \Shv(Y;\Sp) \to \Shv(X;\Sp)$ at the level of sheaves of spectra. Then the functor $f^*$ admits both a left adjoint $f_!$ and a right adjoint $f_*$, and there exists a preferred natural equivalence $f_! \iso f_*$, see \cite{Volpe}*{Proposition~6.18}. Similarly, for a finite étale map $f\colon X \to Y$ of schemes, the pullback functor $f^*\colon \mathrm{SH}(Y) \to \mathrm{SH}(X)$ on Voevodsky's categories of motivic spectra again admits both adjoints which are canonically isomorphic \cite{AyoubI,Hoyois2017Motivic}. In \cite{Scholze2023SixFunctors} it is shown that this is a general phenomenon for six-functor formalisms in the sense of \cite{Mann2022padic}: there exists a class of `cohomologically proper étale maps' $f\colon X \to Y$ for which $f^*$ admits both adjoints which canonically agree.

	As will be discussed below, the appropriate corresponding notion of `commutative monoid' in this context is that of a \textit{sheaf with transfers}.

	\subsection*{Parametrized semiadditivity}

	A common feature among all variations of semiadditivity discussed above is that they provide equivalences $f_! \iso f_*$ between the left and right adjoint of a given functor $f^*$. This phenomenon, called \textit{ambidexterity}, was systematically investigated by Hopkins and Lurie \cite{hopkins2013ambidexterity} for applications in chromatic homotopy theory, and forms the basis of the general notion of parametrized semiadditivity introduced in this paper.

	We work in the setting of parametrized categories $\Cc\colon \Aa\catop \to \Cat$ indexed by some category $\Aa$. The notions of semiadditivity we introduce depend on a chosen wide subcategory $\Qq \subseteq \Aa$ which is closed under base change and which is \textit{inductible}, meaning that it is closed under diagonals and only contains truncated morphisms.

	\begin{introdefinition}[$\Qq$-semiadditivity]
		Given an inductible subcategory $\Qq \subseteq \Aa$, a functor $\Cc\colon \Aa\catop \to \Cat$ is called \textit{$\Qq$-semiadditive} if the following conditions are satisfied:
		\begin{enumerate}[(1)]
			\item It admits \textit{$\Qq$-colimits} (Definition \ref{def:Q_Colimits}): The functors $q^* \coloneqq \Cc(q) \colon \Cc(B) \to \Cc(A)$ for $q\colon A \to B$ in $\Qq$ admit left adjoints $q_!\colon \Cc(A) \to \Cc(B)$ satisfying a base change condition;
			\item It satisfies \textit{ambidexterity} for $\Qq$ (Construction \ref{cons:Ambidexterity}): For every morphism $q\colon A \to B$ in $\Qq$ an inductively defined transformation $\Nmadj_q\colon q^*q_! \to \id_{\Cc(A)}$ exhibits the left adjoint $q_!$ additionally as a \textit{right} adjoint to $q^*$.
		\end{enumerate}
	\end{introdefinition}

	Equivalently we may ask $\Cc$ to admit both $\Qq$-colimits and $\Qq$-limits (the dual notion) and demand that $\Qq$-limits commute with $\Qq$-colimits; see \Cref{prop:Characterization_Q_Semiadditivity} for a precise statement.

	This definition captures all aforementioned variations of semiadditivity:

	\begin{itemize}
		\item We recover ordinary semiadditivity of a category $\Cc$ by letting $\Aa = \Qq$ be the category $\Fin$ of finite sets and taking $\Cc(A) := \Fun(A,\Cc)$ for $A \in \Fin$: condition (1) demands the existence of finite coproducts, while condition (2) asks that finite coproducts are also exhibited as products (via some preferred map);
		\item Given a finite group $G$, we obtain \textit{$G$-semiadditivity} \cite{nardin2016exposeIV} for a $G$-category $\Cc$ by letting $\Aa = \Qq$ be the category of finite $G$-sets;
		\item We obtain \textit{equivariant semiadditivitity} \cite{CLL_Global} by letting $\Aa$ be the category of finite groupoids and letting $\Qq$ be the collection of faithful functors of groupoids.

		\item We obtain \textit{$m$-semiadditivity} of a category $\Cc$ by letting $\Aa = \Qq$ be the category $\Spc_{m}$ of $m$-finite spaces and taking $\Cc(A) \coloneqq \Fun(A,\Cc)$ for $A \in \Spc_{m}$.
		\item  We obtain ambidexterity for six-functor formalisms by letting $\Aa$ be the indexing category of the six-functor formalism and taking $\Qq$ to be an appropriate collection of finite étale maps in $\Aa$.
		\item Various variations are possible, including $p$-typical $m$-semiadditivity, global semiadditivity and very $G$-semiadditivity for an arbitrary group $G$; see \Cref{subsec:Examples}.
	\end{itemize}

	It is often convenient to restrict attention to functors $\Aa\catop \to \Cat$ that satisfy an appropriate sheaf condition. For example, for ordinary semiadditivity we need to restrict to product-preserving functors $\Fin\catop \to \Cat$ to guarantee they are of the form $A \mapsto \Fun(A,\Cc)$ for a category $\Cc$, so that the left and right adjoints of restriction indeed correspond to colimits and limits in $\Cc$. For this reason, we will fix a topos $\Bb$ and work in the context of \textit{$\Bb$-categories}, defined as limit-preserving functors $\Bb\catop \to \Cat$; note that for the topos $\Bb = \Shv_{\tau}(\Aa)$ of sheaves on $\Aa$ with respect to a given Grothendieck topology $\tau$, a $\Bb$-category contains the same data as a $\tau$-sheaf of categories $\Cc\colon \Aa\catop \to \Cat$. To make sure that $\Qq$-semiadditivity is compatible with descent, we will work with what we call \textit{locally inductible} subcategories $\Qq \subseteq \Bb$, see \Cref{def:Locally_Inductible} for a precise definition.

	\subsection*{Main results}

	We will now provide an overview of the main results established in this article. As explained above, we fix a topos $\Bb$ together with a locally inductible subcategory $\Qq \subseteq \Bb$, and work in the context of $\Bb$-categories. We may then consider the \emph{parametrized span category} $\ul{\Span}(\Qq)$ of $\Qq$, given by
	\[
		\ul{\Span}(\Qq) \colon \Bb\catop \to \Cat, \qquad A \mapsto \Span(\Qq_{/A}),
	\]
	where the functoriality in $A$ is induced by the pullback functors $f^*\colon \Qq_{/B} \to \Qq_{/A}$ for $f\colon A \to B$. This $\Bb$-category turns out to be $\Qq$-semiadditive: the functor $\Span(f^*)$ admits a left adjoint given by applying $\Span(\blank)$ to the postcomposition functor $f_!\colon \Qq_{/A} \to \Qq_{/B}$. That this is also a \textit{right} adjoint of $f^*$ follows from the self-duality of span categories; a slightly more elaborate argument shows that this adjunction is indeed exhibited by the map $\Nmadj$. Our first main result shows that $\ul{\Span}(\Qq)$ is in fact universal with this property:

	\begin{introthm}[\Cref{thm:universal_prop_par_spans}]
		\label{introthm:universal_prop_par_spans}
		The $\Bb$-category $\ul{\Span}(\Qq)$ is the free $\Qq$-semiadditive $\Bb$-category on a single generator.
	\end{introthm}
	More precisely, this means that for every $\Qq$-semiadditive $\Bb$-category $\Dd$, evaluation at the identity maps $\id_A \in \Span(\Qq_{/A})$ induces an equivalence of parametrized categories
	\[
	\ul{\Fun}^{\Qq\text-\times}(\ul{\Span}(\Qq),\Dd) \iso \Dd.
	\]
	Here the left-hand side denotes the full subcategory of the parametrized functor category spanned by the functors which are \emph{$\Qq$-continuous}, meaning that they commute with the rights adjoints of restrictions along maps in $\Qq$.

	\Cref{introthm:universal_prop_par_spans} in particular generalizes the universal property of the span category $\Span(\Spc_m)$ established by Harpaz, and in fact our proof follows a similar strategy: the inverse to the evaluation functor is produced via an iterative process of (parametrized) left and right Kan extensions.

	While the language of parametrized categories is necessary to even state the universal property of $\ul\Span(\Qq)$ for general $\Bb$, this result nevertheless has various applications outside parametrized category theory. In particular, since this paper first appeared on the arXiv, Ben-Moshe \cite{ben-moshe-transchromatic} has employed \Cref{introthm:universal_prop_par_spans} to establish naturality of the so-called \emph{transchromatic character map}. Moreover, we further build on \Cref{introthm:universal_prop_par_spans} in \cite{CLL_Span2} to prove a universal property for $(\infty,2)$-categories of iterated spans, which is relevant to the construction of six-functor formalisms.

	\subsection*{$\bm\Qq$-commutative monoids}
	Every choice of $\Qq$ gives rise to a corresponding notion of commutative monoids: If $\Cc$ is a $\Bb$-category admitting $\Qq$-limits, we define the $\Bb$-category of \textit{$\Qq$-commutative monoids in $\Cc$} as
	\[
		\ul{\CMon}^{\Qq}(\Cc) \coloneqq \ul{\Fun}^{\Qprod}(\ul{\Span}(\Qq), \Cc).
	\]

	Building on \Cref{introthm:universal_prop_par_spans}, we prove:

	\begin{introthm}[\Cref{thm:universal_prop_par_Qcom_monoids}, \Cref{thm:universal_prop_par_Qcom_monoids_presentable}]
		\label{introthm:universal_prop_par_Qcom_monoids}
		For a $\Bb$-category $\Cc$ admitting $\Qq$-limits, the forgetful functor $\ul{\CMon}^{\Qq}(\Cc) \to \Cc$ is terminal among $\Qq$-limit-preserving functors $\Dd \to \Cc$ from a $\Qq$-semiadditive $\Bb$-category $\Dd$:
		\[
			\ul\Fun^{\Qprod}(\Dd,\ul{\CMon}^{\Qq}(\Cc)) \iso \ul\Fun^{\Qprod}(\Dd,\Cc).
		\]
		We will say that $\ul{\CMon}^{\Qq}(\Cc)$ is the \textit{$\Qq$-semiadditive completion} of $\Cc$.

		Moreover, if $\Cc$ is presentable (and $\Qq$ satisfies a mild smallness condition), then also $\ul{\CMon}^{\Qq}(\Cc)$ is presentable, and the forgetful functor admits a left adjoint $\Cc \to \ul{\CMon}^{\Qq}(\Cc)$ which is initial among left adjoint functors $\Cc \to \Dd$ into a presentable $\Qq$-semiadditive $\Bb$-category $\Dd$:
		\[
			\ul\Fun^{\mathrm{L}}(\ul{\CMon}^{\Qq}(\Cc),\Dd) \iso \ul\Fun^{\mathrm{L}}(\Cc,\Dd).
		\]
	\end{introthm}

	\subsection*{Sheaves with transfers} Even though $\Qq$-commutative monoids in $\Cc$ are defined as certain \textit{parametrized} functors, we will show as our third main result that they admit a concrete non-parametrized description in terms of \textit{sheaves with transfers} for suitable choices of $\Cc$. To define these, we denote by $\Span(\Bb,\Bb,\Qq)$ the wide subcategory of the span category $\Span(\Bb)$ whose morphisms are spans $A \leftarrow C \rightarrow B$ where the map $C \to B$ lies in $\Qq$.

	\begin{introdefinition}
		For a presentable category $\Ee$ we define the category $\Shv^{\Qq}(\Bb;\Ee)$ of \textit{$\Ee$-valued $\Bb$-sheaves with $\Qq$-transfers} to be the full subcategory
		\[
			\Shv^{\Qq}(\Bb;\Ee) \subseteq \Fun(\Span(\Bb,\Bb,\Qq), \Ee)
		\]
		spanned by those functors $F\colon \Span(\Bb,\Bb,\Qq) \to \Ee$ whose restriction $F\vert_{\Bb\catop}\colon \Bb\catop \to \Ee$ preserves limits. We may assemble the categories of sheaves with transfers on the slice categories $\Bb_{/B}$ into a $\Bb$-category $\ul{\Shv}^{\Qq}(\Bb;\Ee)$.
	\end{introdefinition}

	If $\Qq$ is the groupoid core of $\Bb$, this simplifies to the $\Bb$-category $\ul\Shv(\Bb;\Ee)\colon B\mapsto \Fun^{\textup{R}}(\Bb_{/B}^\op,\Ee)$ of \emph{$\Ee$-valued $\Bb$-sheaves}, which can be equivalently described as the sheafification of the constant functor $\Bb\catop \to \Cat$ with value $\Ee$. We then show:

	\begin{introthm}[\Cref{thm:span-BBQ}]\label{introthm:mackey-sheaves}
		For every presentable category $\Ee$, there exists a (necessarily unique) natural equivalence
		\[
			\ul{\CMon}^{\Qq}(\ul\Shv(\Bb;\Ee)) \simeq \ul{\Shv}^{\Qq}(\Bb;\Ee)
		\]
		of $\Bb$-categories over $\ul\Shv(\Bb;\Ee)$.
	\end{introthm}

	Applying this result for $\Ee = \Spc$ and combining it with \Cref{introthm:universal_prop_par_Qcom_monoids} we deduce:

	\begin{introcorollary}\label{introcor:transfer-enrichment}
		Every $\Qq$-semiadditive $\Bb$-category $\Cc$ is canonically enriched in sheaves with $\Qq$-transfers: the parametrized Hom-functor $\ul{\Hom}\colon \Cc\catop \times \Cc \to \ul{\Spc}_{\Bb}$ uniquely lifts to $\ul{\Shv}^{\Qq}(\Bb;\Spc)$.
	\end{introcorollary}

	Similarly, via the universal property of the $\Bb$-category $\ul\Shv(\Bb;\Spc)$ \cite{martiniwolf2021limits} this implies:

	\begin{introcorollary}\label{introcor:free-presentable}
		The $\Bb$-category $\ul\Shv^\Qq(\Bb;\Spc)$ of space-valued $\Bb$-sheaves with $\Qq$-transfers is the free $\Qq$-semiadditive presentable $\Bb$-category: for any other such $\Dd$, evaluation at a certain `free sheaf with transfers' $\mathbb P(1)$ induces an equivalence
		\[
			\ul\Fun^\textup{L}(\ul\Shv^\Qq(\Bb;\Spc),\Dd)\iso\Dd.
		\]
	\end{introcorollary}

	When $\Bb$ is the sifted cocompletion $\PSh_{\Sigma}(\Aa) = \Fun^{\times}(\Aa\catop,\Spc)$ for an extensive category $\Aa$, and $\Qq$ is given by the maps with fibers in an inductible category $Q\subseteq\Aa$ closed under coproducts,	we show in \Cref{subsec:Mackey_Functors} that the category $\Shv^{\Qq}(\Bb;\Ee)$ simplifies to the category
	\[
		\Mack^{Q}(\Aa;\Ee) \coloneqq \Fun^{\times}(\Span(\Aa,\Aa,Q),\Ee)
	\]
	of \textit{$\Ee$-valued $Q$-Mackey functors on $\Aa$}. In this case, Corollary~\ref{introcor:transfer-enrichment} therefore provides canonical enrichments in Mackey functors: for example, every $G$-semiadditive $G$-category is enriched in $G$-Mackey functors. Similarly, Corollary~\ref{introcor:free-presentable} specializes to a universal property of $G$-Mackey functors.

	\subsection*{Applications} Theorem~\ref{introthm:mackey-sheaves} can be applied to prove, in a unified way, Mackey functor descriptions of various categories considered in (global) equivariant homotopy theory. In particular, we use the universal property of global spectra established in \cite{CLL_Global} to deduce the following result which has also been concurrently proven by Pützstück \cite{puetzstueck}:

	\begin{introthm}[\Cref{cor:global-spectra}]
		Let $G$ be any finite group. Then the category of \emph{$G$-global spectra} from \cite{g-global} is naturally equivalent to
		\begin{equation*}
			\Fun^\times(\Span(\FinGrpd_{/BG},\FinGrpd_{/BG},\FinGrpd_{/BG}[\FinGrpdfaith]),\Sp)
		\end{equation*}
		where $\FinGrpd_{/BG}$ is the $(2,1)$-category of finite groupoids over $BG$, and $\FinGrpd_{/BG}[\FinGrpdfaith]$ denotes the wide subcategory of faithful functors.
	\end{introthm}

	We further use our results from \cite{CLL_Clefts} to reprove the Mackey functor description of $G$-equivariant stable homotopy theory from \cite{cmnn}, see Corollary~\ref{cor:equiv-spectra}.

	On the other hand, we may also use Theorem~\ref{introthm:mackey-sheaves} to provide parametrized interpretations of categories that already come with Mackey functor descriptions, like the category of \textit{Mackey profunctors} defined by Kaledin \cite{Kaledin2022Mackey} and the closely related category of \emph{quasi-finitely genuine $G$-spectra} as defined by Krause--McCandless--Nikolaus \cite{KMN2023Polygonic}. For this, let $\QFin_G \subseteq \Set_G$ denote the full subcategory spanned by the \textit{quasi-finite}\footnote{We adopt the terminology from \cite{KMN2023Polygonic}; Kaledin used the term `admissible.'} $G$-sets, i.e.\ those $G$-sets $S$ for which all orbits are finite and for which the fixed point sets $S^H$ are finite for all cofinite subgroups $H \subseteq G$. Following \cite{KMN2023Polygonic}, we define a \emph{quasi-finitely genuine $G$-spectrum} to be a functor $M \colon \Span(\QFin_G) \to \Sp$ such that for every quasi-finite $G$-set $S$ the canonical map
	\[
		M(S) \to \prod_{\overline{s} \in S/G} M(\pi^{-1}(\overline{s}))
	\]
	is an equivalence, where $\pi\colon S \to S/G$ denotes the quotient map. The category of all such functors is denoted $\Sp^{\qfgen}_G$.

	Letting $\smash{\widehat{\Orb}}_G \subseteq \QFin_G$ denote the full subcategory spanned by the finite orbits $G/H$, we will refer to $\widehat{\Orb}_G$-categories as \emph{$G$-procategories}. Applying our framework to a certain collection of `quasi-finite covering maps,' we define what it means for a $G$-procategory to be \emph{very $G$-semiadditive}, and we prove:

	\begin{introthm}[Theorem~\ref{thm:qfin}]
		\label{introthm:univ_prop_Kaledins_profunctors}
		The category $\Sp^{\qfgen}_G$ of quasi-finitely genuine $G$-spectra is the underlying category of the free presentable very $G$-semiadditive stable $G$-procategory.
	\end{introthm}

	We further provide a similar universal interpretation of Kaledin's category ${\myMm}(G,\mathbb Z)$ of \emph{$\mathbb Z$-valued Mackey profunctors}, see Theorem~\ref{thm:Mackey-pro-fun}.

	\subsection*{Acknowledgments}
	We thank Lior Yanovski for sharing with us their alternative approach to the main theorem of \cite{harpaz2020ambidexterity}, which heavily inspired our proof of \Cref{introthm:universal_prop_par_spans}. We further thank Sebastian Wolf for providing us with the key idea for the proof of \Cref{lem:Colimit_Of_Subcategories}, and Shay Ben-Moshe for useful comments on a draft version of this article.

	The first and third author are associate members of the SFB 1085 Higher Invariants. The second author is an associate members of the Hausdorff Center for Mathematics (DFG GZ 2047/1, project ID 390685813), and the third was. The third author was supported by the DFG Schwerpunktprogramm 1786 `Homotopy Theory and Algebraic Geometry' (project ID SCHW 860/1-1).

	\section{Recollections on parametrized category theory}

	The variants of higher semiadditivity we are interested in are most conveniently phrased using the language of \textit{parametrized category theory}, which we will briefly recall now.

	\begin{definition}[Parametrized categories]
		Throughout this article, we will use the following two notions of parametrized categories:
		\begin{enumerate}
			\item For a small category $T$, we define a \textit{$T$-category} to be a functor $\Cc\colon T\catop \to \Cat$. We write $\Cat_T := \Fun(T\catop,\Cat)$ for the (very large) category of $T$-categories.
			\item For an ($\infty$-)topos $\Bb$, we define a \textit{$\Bb$-category} to be a limit-preserving functor $\Cc\colon \Bb\catop \to \Cat$. We write $\Cat(\Bb)\coloneqq\Fun^{\mathrm{R}}(\Bb^\op,\Cat)$ for the (very large) category of $\Bb$-categories.
		\end{enumerate}
	For a morphism $f\colon A \to B$ in $T$ or $\Bb$, we refer to the functor $f^* := \Cc(f)\colon \Cc(B) \to \Cc(A)$ as the \textit{restriction functor} of $f$.
	\end{definition}

	\begin{remark}
		\label{rmk:Limit_Extension}
		The formalism of $\Bb$-categories is more general than that of $T$-categories: restriction along the Yoneda embedding $T \hookrightarrow \PSh(T)$ defines an equivalence $\Cat(\PSh(T)) \iso \Cat_T$ between $\PSh(T)$-categories and $T$-categories, with inverse given by limit-extension. Given a $T$-category $\Cc$, we will generally abuse notation and denote its limit-extension $\PSh(T)\catop \to \Cat$ again by $\Cc$.

		While the general formalism of parametrized semiadditivity will be developed for $\Bb$-categories, most of our examples will come from $T$-categories for suitable $T$. It will occasionally be convenient to state definitions that apply both to $\Bb$-categories as well as to $T$-categories; in these cases we work with functors $\Aa\catop \to \Cat$ for some (either small or large) category $\Aa$.
	\end{remark}

	\begin{example}[$\Bb$-groupoid]
		Every object $B$ of $\Bb$ defines a $\Bb$-category $\ul{B}$ via the Yoneda embedding:
		\[
			\ul{B} \coloneqq \hom_{\Bb}(-,B) \colon \Bb\catop \to \Spc \hookrightarrow \Cat.
		\]
 		The $\Bb$-categories of this form are called \textit{$\Bb$-groupoids}.
	\end{example}

	\begin{example}[The $\Bb$-category of $\Bb$-groupoids]
		Since $\Bb$ is a topos, the functor $\Bb\catop \to \Cat$ given by $B \mapsto \Bb_{/B}$ (i.e.~the cartesian unstraightening of the target map $\Ar(\Bb)\to\Bb$) preserves limits and thus defines a $\Bb$-category that we denote by $\ul{\Spc}_{\Bb}$ and refer to as the \textit{$\Bb$-category of $\Bb$-groupoids}.
	\end{example}

	\begin{definition}[Underlying category]
		Every $\Bb$-category $\Cc$ has an \textit{underlying category} $\Gamma \Cc := \Cc(1)$, where $1 \in \Bb$ is the terminal object.
	\end{definition}

	\begin{definition}[Parametrized functor category]
		The category $\Cat(\Bb)$ is cartesian closed by \cite{martini2021yoneda}*{Proposition~3.2.11}. We denote the internal hom by $\ul\Fun_\Bb(\Cc,\Dd)$ (or by $\ul\Fun(\Cc,\Dd)$ if $\Bb$ is clear from the context), and denote its underlying category by $\Fun_\Bb(\Cc,\Dd)\coloneqq \Gamma\ulFun(\Cc,\Dd)$.
	\end{definition}

	\begin{proposition}[Categorical Yoneda Lemma, cf.\ \cite{CLL_Global}*{Lemma~2.2.7, Corollary~2.2.9}]
		\label{prop:Yoneda}
		For an object $B \in \Bb$, evaluation at $\id_B\in\ul B(B)$ defines a natural equivalence
		\begin{equation*}
			\Fun_\Bb(\ul B,\Cc)\iso\Cc(B).
		\end{equation*}
		As a consequence, there are natural equivalences
		\[
			\ulFun_{\Bb}(\ul{B}, \Cc) \simeq \Cc(B \times -)
		\]
		and
		\[
		\ul\Fun_\Bb(\Cc,\Dd)(B)\simeq \Fun_\Bb(\Cc\times\ul B,\Dd)\simeq \Fun_\Bb(\Cc,\ul\Fun_\Bb(\ul B,\Dd)).
		\]
	\end{proposition}
	\begin{proof}
		In the special case $\Bb=\PSh(T)$ (hence $\Cat(\Bb)\simeq\Cat_T$), this is the content of \cite{CLL_Global}*{Lemma~2.2.7 and Corollary~2.2.9}. For the general case, we claim that the embedding $\Cat(\Bb) \hookrightarrow \Fun(\Bb\catop,\CAT)$ preserves internal homs, where we jump universes to ensure smallness of $\Bb$. For this, let $\Cc,\Dd\in\Cat(\Bb)$ arbitrary. By the above special case, the internal hom in the category on the right satisfies
		\begin{align*}
			\iota\big(\ul{\Fun}(\Cc,\Dd)(A)^{[n]}\big) \simeq \iota\ulFun(\Cc\times[n],\Dd)(A) &\simeq \iota\Fun(\Cc\times[n],\Dd(A \times \blank))\\&=\hskip3.14159pt\hom(\Cc\times[n],\Dd(A \times \blank)),
		\end{align*}
		for every $A\in\Bb$ and $n\ge 0$; in particular, the complete Segal space associated to $\ul\Fun(\Cc,\Dd)(A)$ is contained in the original universe, so $\ul\Fun(\Cc,\Dd)$ is contained in $\Fun(\Bb\catop, \Cat)$. It remains to show that it is even contained in $\Fun^{\mathrm{R}}(\Bb\catop,\Cat)$. As $\Fun(\Cc,\blank)$ is a right adjoint, it suffices that $\Bb^\op\to\Fun(\Bb^\op,\CAT),A\mapsto\Dd(A\times\blank)$ preserves small limits. Since limits in functor categories are pointwise, this amounts to saying that $\Dd(\blank\times B):\Bb^\op\to\CAT$ preserves small limits for every $B$, which directly follows from cartesian closure of $\Bb$ and the sheaf property of $\Dd$.
	\end{proof}

	\begin{remark}
		In what follows, we will freely cite results from \cite{CLL_Global} for internal homs of $T$-categories even when working with general $\Bb$-categories; in each case the reduction step used in the proof of \Cref{prop:Yoneda} applies.
	\end{remark}

	\begin{remark}
		\label{rmk:fun-slice}
		Given an object $B \in \Bb$, every $\Bb$-category $\Cc$ canonically gives rise to a $\Bb_{/B}$-category $\pi_B^*\Cc$ by precomposing $\Cc$ with the (colimit-preserving) forgetful functor $\pi_B\colon \Bb_{/B} \to \Bb$. The resulting functor $\pi_B^*\colon \Cat(\Bb) \to \Cat(\Bb_{/B})$ preserves internal homs by \cite{CLL_Global}*{Corollary~2.2.11}, and as a result there is for all $\Cc, \Dd \in \Cat(\Bb)$ and every $B \in \Bb$ a natural equivalence
		\[
		\ul\Fun_\Bb(\Cc,\Dd)(B)\iso\Fun_{\Bb_{/B}}(\pi_B^*\Cc,\pi_B^*\Dd).
		\]
		Under this equivalence, restriction along $f\colon A\to B$ corresponds to restriction along $\Bb_{/f}\colon\Bb_{/A}\to\Bb_{/B}$ and conjugating by the evident equivalence, see \cite{CLL_Global}*{Lemma~2.2.12}.
	\end{remark}

	\subsection{Parametrized colimits}

	In parametrized category theory, there is a notion of `groupoid-indexed colimit' that we will now recall. To this end, recall that a class of morphisms $\Qq$ in a category $\Aa$ is said to be \textit{closed under base change} if base changes ($=$ pullbacks) of morphisms in $\Qq$ along morphisms in $\Aa$ exist and are again in $\Qq$.

	\begin{definition}[$\Qq$-colimits]
		\label{def:Q_Colimits}
		Let $\Aa$ be a category and let $\Qq$ be a class of morphisms in $\Aa$ closed under base change. Given a functor $\Cc\colon \Aa\catop \to \Cat$, we say that $\Cc$ \textit{admits $\Qq$-colimits} or \textit{is $\Qq$-cocomplete} if the following conditions are satisfied:
		\begin{enumerate}[(1)]
			\item For every morphism $q\colon A \to B$ in $\Qq$, the functor $q^* \colon \Cc(B) \to \Cc(A)$ admits a left adjoint $q_!\colon \Cc(A) \to \Cc(B)$.
			\item For every pullback square
			\[
			\begin{tikzcd}
				A' \dar[swap]{q'} \rar{g} \drar[pullback] & A \dar{q} \\
				B' \rar{f} & B
			\end{tikzcd}
			\]
			in $\Aa$ with $q$ in $\Qq$, the Beck--Chevalley transformation $\BC_!\colon q'_!g^* \to f^*q_!$ of functors $\Cc(A) \to \Cc(B')$ is an equivalence.
		\end{enumerate}

		Dually, we define what it means for $\Cc$ to \emph{admit $\Qq$-limits}.
	\end{definition}

	\begin{remark}
		By \cite{martiniwolf2021limits}*{Corollary~3.2.11}, the above amounts to saying that $q^*\colon\pi_B^*\Cc\to\ul\Fun(\ul A,\pi_B^*\Cc)\simeq\Cc(A\times_B\blank)$ has a \emph{parametrized left adjoint} $q_!$, i.e.~a left adjoint in the homotopy $2$-category of ${\Fun}\big((\Aa_{/B})^\op,{\Cat}\big)$.
	\end{remark}

	We will mostly use this notion in the case of $\Bb$-categories for a topos $\Bb$. In this case, we will further assume that the class of morphisms $\Qq$ in $\Bb$ is \textit{local}, meaning that a morphism $q\colon A \to B$ is in $\Qq$ whenever there exists an effective epimorphism $\bigsqcup_{i\in I} B_i \twoheadrightarrow B$ in $\Bb$ such that each of the base change maps $A \times_B B_i \to B_i$ is in $\Qq$.

	\begin{remark}\label{rk:q-loc-colim}
		Let $T$ be small and let $Q\subseteq T$ be closed under base change. We define $\Qq\coloneqq Q_\loc$ as the collection of all maps $q\colon X\to Y$ in $\PSh(T)$ such that for every map $A\to Y$ from a representable the base change $A\times_YX\to A$ belongs to $Q$. Given an effective epimorphism $\coprod_i Y_i\to Y$, the Yoneda lemma shows that any map $A\to Y$ from a representable factors through one of the $Y_i$. Thus, $Q_\loc$ is a local class in $\PSh(T)$, and we will refer to it as the \emph{local class generated by $Q$}. By \cite{CLL_Global}*{Remark~2.3.15}, a $T$-category is then $Q$-cocomplete if and only if its limit extension is $Q_\loc$-cocomplete.
	\end{remark}

	\begin{definition}
		Let $\Cc,\Dd\colon\Aa^\op\to\Cat$ be $\Qq$-cocomplete. A natural transformation $F\colon \Cc \to \Dd$ is said to \textit{preserve $\Qq$-colimits} if for every morphism $q\colon A \to B$ in $\Qq$ the Beck--Chevalley map $q_!F_A \to F_B q_!$ is an equivalence; alternatively we say that \textit{$F$ is $\Qq$-cocontinuous}.

		If $\Aa=\Bb$ is a topos, we denote by $\Cat(\Bb)^{\Qcoprod} \subseteq \Cat(\Bb)$ the (non-full) subcategory spanned by those $\Bb$-categories admitting $\Qq$-colimits and those $\Bb$-functors preserving $\Qq$-colimits. Dually, we define the non-full subcategory $\CatBQProd\subseteq\Cat(\Bb)$.
	\end{definition}

	\begin{remark}
		If $\Aa=T$ is small and $Q\subseteq T$ is closed under base change, \cite{CLL_Global}*{Lemma~2.3.16} shows that a functor $F\colon\Cc\to\Dd$ in $\Cat_T$ preserves $Q$-colimits if and only if it preserves $Q_\loc$-colimits when viewed as a map in $\Cat(\PSh(T))$.
	\end{remark}

	\begin{example}\label{ex:non-param-colim}
		In the case $\Bb=\Spc$, the condition of being $\Qq$-cocomplete reduces to a non-parametrized cocompleteness condition. Recall that taking global sections defines an equivalence $\Cat(\Spc)\iso\Cat$, with inverse given by sending a category $\Cc$ to $\Fun(\blank,\Cc)$. For local $\Qq \subseteq \Spc$, a category $\Cc$ then has $\Qq$-colimits if and only if $q^*\colon\Fun(B,\Cc)\to\Fun(A,\Cc)$ has a left adjoint (satisfying base change) for every $q\colon A\to B$ in $\Qq$. Specializing to $B=1$, we see that $\Cc$ has $A$-indexed colimits for all $A\in\Qq_{/1}\subseteq\Spc\subseteq\Cat$; conversely, if $\Cc$ admits such colimits, then Kan's pointwise formula and the closure of $\Qq$ under base change show that all the required adjoints exist and satisfy base change, i.e.~$\Cc$ is $\Qq$-cocomplete as a $\Bb$-category.

		In the same way, we see that a $\Spc$-functor is $\Qq$-cocontinuous if and only if it preserves $\Qq_{/1}$-colimits as a functor of non-parametrized categories.
	\end{example}

	\begin{construction}\label{const:parametrized_Q}
		For a local class of morphisms $\Qq$ in $\Bb$ and an object $B \in \Bb$, we denote by
		\[
		\ulbbU{\Qq}(B) \subseteq \Bb_{/B}
		\]
		the full subcategory spanned by those morphisms $q\colon A \to B$ which are contained in $\Qq$. Since $\Qq$ is closed under base change, pullback along a morphism $f\colon A \to B$ restricts to a functor $f^*\colon \ulbbU{\Qq}(B) \to \ulbbU{\Qq}(A)$, and since $\Qq$ is local
		we obtain a $\Bb$-subcategory $\ulbbU{\Qq} \subseteq \ul{\Spc}_{\Bb}$.
	\end{construction}

	\begin{remark}
		The $\Bb$-category $\ulbbU{\Qq}$ is a \textit{class of $\Bb$-groupoids} in the terminology of \cite{martiniwolf2021limits}, and thus determines a notion of \textit{$\ulbbU{\Qq}$-colimits} in a $\Bb$-category. By Proposition~5.4.2 of \emph{op.\ cit.} this precisely recovers the above definitions of $\Qq$-colimits and $\Qq$-cocontinuity.
	\end{remark}

	\begin{proposition}[{\cite{martiniwolf2021limits}*{Proposition~5.2.7}}] \label{prop:fun-colimits}
		Let $\Cc$ and $\Dd$ be $\Bb$-categories, and assume that $\Dd$ has $\Qq$-colimits. Then:
		\begin{enumerate}
			\item The $\Bb$-category $\ul\Fun(\Cc,\Dd)$ again has all $\Qq$-colimits.
			\item For any $\Cc\to\Cc'$ the restriction $\ul\Fun(\Cc',\Dd)\to\ul\Fun(\Cc,\Dd)$ is $\Qq$-cocontinuous.
			\item For any $\Qq$-cocontinuous functor $\Dd\to\Dd'$ the induced functor $\ul\Fun(\Cc,\Dd)\to\ul\Fun(\Cc,\Dd')$ is again $\Qq$-cocontinuous. \qed
		\end{enumerate}
	\end{proposition}

	\begin{construction}\label{const:coprod_pres_func_cat}
		Let $\Cc,\Dd$ be $\Qq$-cocomplete $\Bb$-categories. We define a full $\Bb$-subcategory $\ul\Fun_\Bb^{\Qcoprod}(\Cc,\Dd)\subseteq\ul\Fun_\Bb(\Cc,\Dd)$ spanned in degree $B\in\Bb$ by the
		objects corresponding to $\pi_B^{-1}\Qq$-cocontinuous functors $\pi_B^*\Cc\to\pi_B^*\Dd$ under the equivalence from \Cref{rmk:fun-slice}; see \cite{martiniwolf2021limits}*{Remark~5.2.4} for a proof that this is indeed a $\Bb$-category.

		By \cite{CLL_Global}*{Remark~2.3.27}, $\ul\Fun_\Bb^{\Qcoprod}(\Cc,\Dd)$ can equivalently be described as the full subcategory spanned in degree $B\in\Bb$ by those objects that correspond to $\Qq$-cocontinuous functors $\Cc\to\ul\Fun(\ul B,\Dd)$ under the final equivalence of Proposition~\ref{prop:Yoneda}.
	\end{construction}

        Assume now that the morphisms in $\Qq$ are closed under composition and contain all equivalences, so that $\Qq \subseteq \Bb$ defines a wide subcategory. In this case, the $\Bb$-category $\ulbbU{\Qq}$ admits $\Qq$-colimits. In fact, it is universal with this property:

        \begin{proposition}[{\cite{martiniwolf2021limits}*{Theorem~7.1.13}}]\label{prop:uni_prop_U_Q}
                The $\Bb$-category $\ulbbU{\Qq}$ is the free $\Qq$-cocomplete $\Bb$-category with $\Qq$-colimits: for every $\Qq$-cocomplete $\Bb$-category $\Dd$, evaluation at the point $\pt \colon \ul{1} \to \ulbbU{\Qq}$ defines an equivalence of $\Bb$-categories
                \[
                        \ulFun^{\Qcoprod}(\ulbbU{\Qq}, \Dd) \iso \Dd,
                \]
                whose inverse is given by left Kan extension along $\pt \colon \ul{1} \to \ulbbU{\Qq}$. \qed
        \end{proposition}

	While most of our paper only refers to the above `groupoid indexed colimits,' we will on some rare occasions need the complementary notion of \emph{fiberwise colimits}:

	\begin{definition}\label{defi:fw-colim}
		Let $K$ be a (non-parametrized) category. We say that a $\Bb$-category $\Cc$ has \emph{fiberwise $K$-shaped colimits} if the category $\Cc(A)$ has $K$-shaped colimits for every $A \in \Bb$ and the restriction functor $f^*\colon\Cc(B)\to\Cc(A)$ preserves $K$-shaped colimits for each $f\colon A \to B$ in $\Bb$.

		Given a functor $F\colon\Cc\to\Dd$ of $\Bb$-categories with fiberwise $K$-shaped colimits, we say that $F$ \emph{preserves fiberwise $K$-shaped colimits} if each $F_A\colon\Cc(A)\to\Dd(A)$ preserves $K$-shaped colimits.
	\end{definition}

	\begin{remark}\label{rk:fun-colimits-full}
		\cite{martiniwolf2021limits}*{Proposition~5.7.2} also proves the analogue of Proposition~\ref{prop:fun-colimits} for fiberwise colimits in $\Bb$-categories and functors preserving them, making precise that {all} colimits in functor categories are pointwise.
	\end{remark}

	Combining the above two notions of colimits we define:

	\begin{definition}\label{defi:cocomplete}
		A $\Bb$-category is called \emph{cocomplete} if it is $\Bb$-cocomplete in the sense of Definition~\ref{def:Q_Colimits} and moreover \emph{fiberwise cocomplete}, i.e.~has all small fiberwise colimits in the sense of Definition~\ref{defi:fw-colim}.

		A $\Bb$-functor $F\colon\Cc\to\Dd$ of cocomplete $\Bb$-categories is called \emph{cocontinuous} if it is $\Bb$-cocontinuous and preserves all small fiberwise colimits.
	\end{definition}

	\begin{warn}
		If $\Bb=\PSh(T)$, we referred to the above notion as \emph{$T$-cocompleteness} in \cite{CLL_Global, CLL_Clefts}, which clashes with the terminology in Definition~\ref{def:Q_Colimits} above.
	\end{warn}

	\begin{remark}
		If $\Cc$ is cocomplete, then the inclusion of constant diagrams $\pi_A^*\Cc\to\ul\Fun(\Kk,\pi_A^*\Cc)$ has a left adjoint for every $A\in\Bb$ and every small $\Bb_{/A}$-category $\Kk$ by \cite{martiniwolf2021limits}*{Corollary~5.4.7}. In particular, it makes sense to talk about $\Kk$-shaped colimits in $\pi_A^*\Cc$ for any such $\Kk$.
	\end{remark}

	\section{Parametrized semiadditivity}

	In this section, we introduce a wide range of generalized notions of semiadditivity for parametrized categories, using the framework of \textit{ambidexterity} by Hopkins and Lurie \cite{hopkins2013ambidexterity}. We recall this framework in \Cref{subsec:Ambidexterity}, and specialize it in  \Cref{subsec:Parametrized_Semiadditivity} to $\Bb$-categories for a topos $\Bb$. In \Cref{subsec:Presheaf_Topoi} we discuss the case of presheaf topoi, where we present a more flexible `user interface' for parametrized semiadditivity. A wide range of examples is provided in \Cref{subsec:Examples}, and several alternative characterizations of parametrized semiadditivity are given in \Cref{subsec:Characterizations_Q_Semiadditivity}.

	\subsection{Ambidexterity}
	\label{subsec:Ambidexterity}

	We start with a recollection on ambidexterity.

	\begin{definition}[Inductible subcategory]
		\label{def:Inductible_Subcategory}
		Let $\Aa$ be a category and let $\Qq$ be a wide subcategory of $\Aa$ closed under base change. We say that $\Qq$ is \textit{inductible} if the following conditions are satisfied:
		\begin{enumerate}
			\item $\Qq$ is \textit{closed under diagonals}: for every morphism $q\colon A \to B$ in $\Qq$, the diagonal map $\Delta_q \colon A \to A \times_B A$ is again in $\Qq$;
			\item $\Qq$ is \textit{truncated}: every morphism $q\colon A \to B$ in $\Qq$ is truncated (i.e.~$n_q$-truncated for some natural number $n_q$).
		\end{enumerate}
	\end{definition}

	The assumptions on $\Qq$ allow us to make inductive definitions for morphisms in $\Qq$ by iteratively passing to diagonals, explaining our terminology. The condition that $\Qq$ is closed under diagonals in $\Aa$ admits various alternative characterizations:

	\begin{lemma}
		\label{lem:CharacterizationClosedUnderDiagonals}
		For a wide subcategory $\Qq\subseteq\Aa$ closed under base change, the following conditions are equivalent:
		\begin{enumerate}[(1)]
			\item $\Qq$ is closed under diagonals;
			\item $\Qq$ is \emph{left-cancelable}: for morphisms $p\colon A \to B$ and $q\colon B \to C$ in $\Aa$, if both $q$ and $qp$ are in $\Qq$ then also $p$ is in $\Qq$;
			\item $\Qq$ admits pullbacks and the inclusion $\Qq \hookrightarrow \Aa$ preserves pullbacks.
		\end{enumerate}
	\end{lemma}
	\begin{proof}
		For (1) $\implies$ (2), observe that with $p$ and $q$ as in (2) we may factor $p$ as the composite of $(1,p)\colon A \to A \times_C B$ and $\pr_B\colon A \times_C B \to B$. The first map is a base change of $\Delta_q\colon B \to B \times_C B$ and the second map is a base change of $qp$, hence by assumption both lie in $\Qq$ and thus so does $p$. For (2) $\implies$ (3), consider morphisms $A \to B$ and $A' \to B$ in $\Qq$. It follows from (2) that a map $C \to A \times_B A'$ is in $\Qq$ if and only if the two components $C \to A$ and $C \to A'$ are, from which (3) is an immediate consequence. The implication (3) $\implies$ (1) is clear.
	\end{proof}

	Consider an inductible subcategory $\Qq$ of a category $\Aa$, and let $\Cc\colon \Aa\catop \to \Cat$ be a functor which is $\Qq$-cocomplete in the sense of \Cref{def:Q_Colimits}. The restriction of $\Cc$ to $\Qq\catop$ admits a cartesian unstraightening $\smallint (\Cc\vert_{\Qq^\op}) \to \Qq$, which due to $\Qq$-cocompleteness of $\Cc$ is a \textit{Beck--Chevalley fibration} in the sense of \cite{hopkins2013ambidexterity}*{Definition~4.1.3} and thus gives rise to a notion of \textit{$\Cc$-ambidexterity}:

	\begin{construction}[Ambidexterity, \cite{hopkins2013ambidexterity}*{Construction~4.1.8}]
		\label{cons:Ambidexterity}
		Let $\Qq$ be an inductible subcategory of a category $\Aa$ and let $\Cc\colon \Aa\catop \to \Cat$ be a functor which is $\Qq$-cocomplete in the sense of \Cref{def:Q_Colimits}. We will inductively define what it means for an $n$-truncated morphism $q\colon A \to B$ in $\Qq$ to be \textit{$\Cc$-ambidextrous}, in which case we will construct a transformation $\mu_q^{(n)}\colon \id_{\Cc(B)} \to q_!q^*$ exhibiting $q_!$ as a right adjoint to $q^*$.

		The induction starts at $n = -2$, in which case any $(-2)$-truncated morphism $q$ is declared to be $\Cc$-ambidextrous. Since $q$ is an equivalence, the counit map $q_!q^* \to \id_{\Cc(B)}$ is an equivalence, and we define $\mu_q^{(-2)}\colon \id_{\Cc(B)} \to q_!q^*$ as its inverse.

		Assume now that we have defined the $n$-truncated $\Cc$-ambidextrous morphisms for some $n \geq -2$ and have assigned to them the required transformations $\mu_q^{(n)}$. We say that an $(n+1)$-truncated morphism $q\colon A \to B$ in $\Qq$ is \textit{weakly $\Cc$-ambidextrous} if its diagonal $\Delta_q\colon A \to A \times_B A$ is $\Cc$-ambidextrous (which is well-defined since $\Delta_q$ is $n$-truncated). Consider the following commutative diagram:
		\[
		\begin{tikzcd}
			A \ar[bend left=15]{drr}{\id_A} \ar[bend right=15, swap]{ddr}{\id_A} \drar{\Delta} \\
			& A \times_B A \dar{\pr_2} \rar{\pr_1} \drar[pullback] & A \dar{q} \\
			& A \rar{q} & B\rlap.
		\end{tikzcd}
		\]
		We define the \textit{adjoint norm map} $\Nmadj_q\colon q^*q_! \to \id$ as the following composite:
		\[
			\Nmadj_q\colon q^*q_!\xrightarrow{\BC_!^{-1}} \pr_{1!}\pr_2^* \xrightarrow{\mu_{\Delta}^{(n)}} \pr_{1!}\Delta_!\Delta^*\pr_{2}^*\simeq \id.
		\]
		An $(n+1)$-truncated morphism $q\colon A \to B$ is called \textit{$\Cc$-ambidextrous} if every base change $q'$ of $q$ is weakly $\Cc$-ambidextrous and the adjoint norm map $\Nmadj_{q'}\colon q^{\prime*}q'_* \to \id_{\Cc(A')}$ exhibits $q'_!$ as a right adjoint of $q^{\prime*}$. In this case, we let $\mu_q^{(n+1)}\colon \id_{\Cc(B)} \to q_!q^*$ denote the corresponding unit for the resulting adjunction $q^* \dashv q_!$.
	\end{construction}

	\begin{remark}\label{rk:projections-swapped}
		The norm map is independent of the choice of pullback. In particular, taking the same object $A\times_BA$ but with the two projection maps swapped, we see that we can equivalently define the adjoint norm map as the composite
		\begin{equation*}
			q^*q_!\simeq\pr_{2!}\pr_1^*\xrightarrow{\;\mu\;}\pr_{2!}\Delta_!\Delta^*\pr_1^*\simeq\id.
		\end{equation*}
	\end{remark}

	\begin{remark}\label{rk:ambidexterity-restriction}
		Let $f\colon\Aa'\to\Aa$ be a functor and let $\Qq'\subseteq\Aa'$ be inductible such that $f(\Qq')\subseteq\Qq$ and $f$ preserves pullbacks along maps in $\Qq'$. Given any $\Cc\colon\Aa^\op\to\Cat$, we define $f^*\Cc\coloneqq \Cc\circ f\colon \Aa^{\prime\op}\to\Cat$. It then follows straight from the definition that $f^*\Cc$ is $\Qq'$-cocomplete if $\Cc$ is $\Qq$-cocomplete, and that $q\in\Qq'$ is (weakly) $f^*\Cc$-ambidextrous if $f(q)$ is (weakly) $\Cc$-ambidextrous. Moreover, the adjoint norm map for $q$ agrees with the adjoint norm map for $f(q)$ in $\Cc$.
	\end{remark}

	\begin{remark}[Norm map]
		In the situation of \Cref{cons:Ambidexterity}, consider a weakly $\Cc$-ambidextrous morphism $q\colon A \to B$. If the functor $q^*\colon \Cc(B) \to \Cc(A)$ admits a right adjoint $q_*\colon \Cc(A) \to \Cc(B)$, then the adjoint norm map $\Nmadj_q\colon q^*q_! \to \id$ corresponds to a transformation $\Nm_q\colon q_! \to q_*$ that we call the \textit{norm map} associated to $q$. It follows that $q$ is $\Cc$-ambidextrous if and only if for each base change $q'$ the restriction functor $q^{\prime*}$ admits a right adjoint $q'_*$ and the norm map $\Nm_{q'}\colon q'_! \to q'_*$ is an equivalence.
	\end{remark}

	The above construction interacts with natural transformations as one would expect:

	\begin{proposition}[cf.\ {\cite{CSY20}*{Theorem~3.2.3}}]
		\label{prop:adj-norm-vs-B-functor}
		Let $F\colon \Cc \to \Dd$ be a natural transformation of $\Qq$-cocomplete functors $\Aa^\op\to\Cat$. Assume that for every $(n-1)$-truncated map $p$ in $\Qq$ at least one of the Beck--Chevalley maps $\BC_!\colon p_!F\to Fp_!$ and $\BC_*\colon Fp_*\to p_*F$ is invertible.
		\begin{enumerate}
			\item  Let $q$ be an $n$-truncated map that is both weakly $\Cc$-ambidextrous and weakly $\Dd$-ambidextrous. Then the following diagram commutes:
			\[
			\begin{tikzcd}
				q^*q_!F \dar[swap]{\Nmadj_qF} \rar{\BC_!} & q^*Fq_! \rar{\sim} & Fq^*q_! \dar{F\Nmadj_q} \\
				F \ar[equal]{rr} && F.
			\end{tikzcd}
			\]
			\item Assume in addition that $q^*\colon\Cc(B)\to\Cc(A)$ and $q^*\colon\Dd(B)\to\Dd(A)$ admit right adjoints $q_*$. Then the following diagram commutes:
			\begin{equation*}
				\begin{tikzcd}
					q_!F\arrow[r, "\Nm_q"]\arrow[d, "\BC_!"'] &[1em] q_*F\\
					Fq_!\arrow[r, "F(\Nm_q)"'] & Fq_*\arrow[u, "\BC_*"']\rlap.
				\end{tikzcd}
			\end{equation*}
			\item Assume that $q$ is $\Cc$-ambidextrous and $\Dd$-ambidextrous and that at least one of the Beck--Chevalley maps $q_!F\to Fq_!$ and $Fq_*\to q_*F$ is invertible. Then also the following diagram commutes:
			\begin{equation*}
				\begin{tikzcd}
					F\arrow[d, "\mu_q"']\arrow[r, "F(\mu_q)"] & Fq_!q^*\\
					q_!q^*F\arrow[r,equal] & q_!Fq^*\arrow[u, "\BC_!"']\llap.
				\end{tikzcd}
			\end{equation*}
		\end{enumerate}
	\end{proposition}

	\begin{proof}
		First fix $n$ and $q$ and observe that $(2)$ follows from $(1)$ via adjoining over, also see \cite{CSY20}*{Lemma 2.2.11}. We will now show that this in turn implies $(3)$: indeed, in the diagram
		\begin{equation*}
			\begin{tikzcd}
				F\arrow[r, "F(\eta)"]\arrow[d,"\eta"'] & Fq_*q^*\arrow[d, "\BC_*"] \arrow[r, "F(\Nm^{-1})"] &[2em] Fq_!q^*\\
				q_*q^*F\arrow[r,equals] & q_*Fq^*\arrow[r, "\Nm^{-1}"'] & q_!Fq^*\arrow[u, "\BC_!"']
			\end{tikzcd}
		\end{equation*}
		the right-hand square commutes by $(2)$ and the assumption that at least one of the two Beck--Chevalley maps is invertible, while the left-hand square commutes by direct inspection.

		Using this, we will now prove $(1)$ by induction on $n$. For $n=-2$, $\Nmadj_q$ is simply the inverse of the unit $\id\to q^*q_!$, and the statement follows by a standard mate argument, also see \cite{CSY20}*{Lemma~2.2.3(3)}. If we already know the statement for $n-1$, then we consider the diagram
		\[\begin{tikzcd}
			{q^*q_!F} & { (\pr_1)_!\pr_2^* F} && { (\pr_1)_!\Delta_! \Delta^* \pr_2^* F} & F \\
			& {(\pr_1)_!F\pr_2^*} && { (\pr_1)_!\Delta_! \Delta^* F \pr_2^*} \\
			{q^*F q_!} & {{}} && { (\pr_1)_!\Delta_! F \Delta^* \pr_2^*} & F \\
			& {(\pr_1)_!F\pr_2^*} && {(\pr_1)_! F \Delta_! \Delta^* \pr_2^*} \\
			{ F q^*q_!} & {F(\pr_1)_!(\pr_2)^* } && {F (\pr_1)_!\Delta_! \Delta^* \pr_2^*} & F
			\arrow["\sim", from=1-4, to=1-5]
			\arrow["{\BC_!^{-1}}", from=5-1, to=5-2]
			\arrow["{\BC_!}"', from=1-1, to=3-1]
			\arrow["\sim"', from=3-1, to=5-1]
			\arrow[equals, from=1-4, to=2-4]
			\arrow[equals, from=1-2, to=2-2]
			\arrow[equals, from=3-4, to=2-4]
			\arrow["{\BC_!}"', from=4-4, to=5-4]
			\arrow["{\BC_!}"', from=3-4, to=4-4]
			\arrow["{\BC_!^{-1}}", from=1-1, to=1-2]
			\arrow["{\BC_!}"', from=4-2, to=5-2]
			\arrow[Rightarrow, no head, from=2-2, to=4-2]
			\arrow[Rightarrow, no head, from=1-5, to=3-5]
			\arrow[""{name=0, anchor=center, inner sep=0}, Rightarrow, no head, from=3-5, to=5-5]
			\arrow["{\textstyle(*)}"{description}, draw=none, from=1-1, to=5-2]
			\arrow["{\mu_{\Delta}}", from=1-2, to=1-4]
			\arrow["{\mu_{\Delta}}", from=2-2, to=2-4]
			\arrow["{\mu_{\Delta}}", from=4-2, to=4-4]
			\arrow["{\mu_{\Delta}}", from=5-2, to=5-4]
			\arrow["{\textstyle(\dagger)}"{description,xshift=-6pt}, draw=none, from=2-2, to=4-4]
			\arrow["\sim", from=3-4, to=3-5]
			\arrow["\sim", from=5-4, to=5-5]
			\arrow["\textstyle{(*)}"{description,xshift=-2pt}, draw=none, from=4-4, to=0]
		\end{tikzcd}\]
		whose top and bottom row spell out $\Nmadj_q$ and $F(\Nmadj_q)$, respectively; here and in what follows, we will simply denote the naturality constraints of an $\Aa^\op$-natural transformation by equality signs to streamline notation.

		The two subdiagrams marked $(*)$ commute by basic mate arguments, cf.~\cite{CSY20}*{Lemma~2.2.4(1)}, while the subdiagram $(\dagger)$ commutes by the induction hypothesis and the above implication $(1)\Rightarrow(3)$. As all the remaining subdiagrams commute simply by naturality, this completes the inductive step.
	\end{proof}

	As an immediate consequence, we can now describe the interaction of the norm with base change, cf.~\cite{hopkins2013ambidexterity}*{Proposition~4.2.1 and Remark~4.2.3}:

	\begin{corollary}\label{cor:norm-vs-restrictions}
		Let
		\begin{equation*}
			\begin{tikzcd}
				A'\arrow[d, "q'"']\arrow[r, "g"]\drar[pullback]& A\arrow[d, "q"]\\
				B'\arrow[r, "f"'] & B
			\end{tikzcd}
		\end{equation*}
		be a pullback in $\Aa$ such that $q$ is a map in $\Qq$ (hence so is $q'$).
		\begin{enumerate}
			\item If $\Cc$ is $\Qq$-cocomplete and $q$ is weakly $\Cc$-ambidextrous, then we have a commutative diagram
			\begin{equation*}
				\begin{tikzcd}
					q^{\prime*}f^*q_!\arrow[r,equals] & g^*q^*q_!\arrow[d, "g^*\Nmadj_q"]\\
					\arrow[u, "\BC_!"]q^{\prime*}q'_!g^*\arrow[r, "\Nmadj_{q'}"'] & g^*\rlap.
				\end{tikzcd}
			\end{equation*}
			\item Assume in addition that $q^*$ and $q^{\prime*}$ admit right adjoints. Then also
			\begin{equation*}
				\begin{tikzcd}
					f^*q_!\arrow[r, "f^*\Nm_q"] &[1em] f^*q_*\arrow[d, "\BC_*"]\\
					q'_!g^*\arrow[u, "\BC_!"]\arrow[r,"\Nm_{q'}"'] & q'_*g^*
				\end{tikzcd}
			\end{equation*}
			commutes.
		\end{enumerate}
		\begin{proof}
			For the first statement, let $\pi_B\colon\Aa_{/B}\to\Aa$ denote the projection. It then suffices to apply Proposition~\ref{prop:adj-norm-vs-B-functor}(1) to the $\Aa_{/B}$-natural transformation $f^*\colon\pi_B^*\Cc\to\Cc(A\times_B\blank)$, using Remark~\ref{rk:ambidexterity-restriction} to identify the adjoint norms on both sides.

			The second statement follows in the same way from Proposition~\ref{prop:adj-norm-vs-B-functor}(2).
		\end{proof}
	\end{corollary}

	\subsection{Parametrized semiadditivity}
	\label{subsec:Parametrized_Semiadditivity}
	The notion of ambidexterity leads to a variety of notions of parametrized semiadditivity for $\Bb$-categories. These varieties are most naturally indexed on \emph{locally inductible subcategories}, which we introduce now.

		\begin{definition}\label{def:Locally_Inductible}
		A wide local subcategory $\Qq$ of a topos $\Bb$ is \emph{locally inductible} if
		\begin{enumerate}
			\item every morphism $q\colon A \to B$ in $\Qq$ \textit{locally truncated}: there exists a covering $(B_i \to B)_{i\in I}$ (i.e.~the induced map $\bigsqcup_{i \in I} B_i \to B$ is an effective epimorphism) such that each base change $q_i\colon B_i \times_B A \to B_i$ is truncated, \textit{and}
			\item $\Qq$ is closed under diagonals.
		\end{enumerate}
		Note that the collection of truncated morphisms in $\Qq$ is an inductible subcategory of $\Bb$, so that the framework of ambidexterity applies.
		\end{definition}

	\begin{definition}[$\Qq$-semiadditivity]
	\label{def:Q_Semiadditivity}
	Let $\Bb$ be a topos equipped with a local inductible subcategory $\Qq$. We say that a $\Bb$-category $\Cc$ is \textit{$\Qq$-semiadditive} if it admits $\Qq$-colimits and if every truncated map $q\colon A\rightarrow B$ in $\Qq$ is $\Cc$-ambidextrous.
	\end{definition}

	\begin{remark}\label{rk:norm-res-pb}
		Let $f\colon\Bb'\to\Bb$ be a left adjoint functor that preserves pullbacks, and let $\Qq'\subseteq\Bb',\Qq\subseteq\Bb$ be locally inductible with $f(\Qq')\subseteq\Qq$. Specializing Remark~\ref{rk:ambidexterity-restriction}, we see that for any $\Qq$-semiadditive $\Bb$-category $\Cc$ the restriction $f^*\Cc$ is a $\Qq'$-semiadditive $\Bb'$-category, with the evident (adjoint) norms for truncated maps.

		In particular, if $A\in\Bb$ is arbitrary, we can apply this to the forgetful functor $\pi_A\colon\Bb_{/A}\to\Bb$ and the locally inductible subcategory $\Qq'=\Bb_{/A}[\Qq]\coloneqq\pi_A^{-1}(\Qq)$. This will in various proofs allow us to restrict to slices, simplifying notation.
	\end{remark}

	\begin{remark}\label{rem:functor-cat-Q-semi}
		Suppose that $\Cc$ is $\Qq$-semiadditive. Because parametrized (co)limits in functor categories are computed pointwise, one easily checks by induction that $\ulFun(\Ii,\Cc)$ is again $\Qq$-semiadditive for every small $\Bb$-category $\Ii$, with (adjoint) norm maps given pointwise by the norms in $\Cc$.
	\end{remark}

	\begin{remark}
		Note that the definition of $\Qq$-semiadditivity for a locally inductible class $\Qq$ only requires that truncated maps in $\Qq$ are $\Cc$-ambidextrous, because only in this case does the inductive procedure of \Cref{cons:Ambidexterity} terminate. Nevertheless, we will show in Theorem~\ref{thm:norm-non-truncated} that there are natural units and counits witnessing an adjunction $q^*\dashv q_!$ for any map $q\in \Qq$.
	\end{remark}

	Conversely, it suffices to check $\Qq$-cocompleteness on the classes $\Qq_{\le n}$ of $n$-truncated maps for every finite $n$:

	\begin{lemma}\label{lemma:colim-enough-truncated}
		Let $\Qq\subseteq\Bb$ be a locally truncated local class. Then a $\Bb$-category $\Cc$ is $\Qq$-cocomplete if and only if it is $\Qq_{\le n}$-cocomplete for every $n\ge-2$.
		\begin{proof}
			The `only if' part is clear. For the other direction, fix $q\colon A\to B$ and consider the full subcategory $\Sigma\subseteq\Bb_{/B}$ of all $f\colon B'\to B$ such that the pullback $q'\coloneqq f^*(q)\colon A\times_BB'\to B'$ is truncated. This is a sieve as truncated maps are stable under pullback, and it is covering by the assumption that $q$ be locally truncated. Moreover, $\Qq_{\le n}$-cocompleteness for all $n\ge-2$ shows that $q^{\prime*}$ admits a left adjoint $q_!'$ satisfying base change along maps in $\Sigma$. Letting $q$ vary, the lemma is therefore an instance of Corollary~\ref{cor:colimits-local}.
		\end{proof}
	\end{lemma}

	In the same way Lemma~\ref{lemma:cocont-local} specializes to:

	\begin{lemma}\label{lemma:cocont-enough-truncated}
		Let $\Qq\subseteq\Bb$ be local and locally truncated. Then a functor $F\colon\Cc\to\Dd$ of $\Qq$-cocomplete $\Bb$-categories is $\Qq$-cocontinuous if and only if it is $\Qq_{\le n}$-cocontinuous for all $n\ge-2$.\qed
	\end{lemma}

	While the definition of $\Qq$-semiadditivity only refers to $\Qq$-colimits, we in fact also have all $\Qq$-limits:

	\begin{corollary}
		\label{cor:Semiadditive_Cats_Have_QLimits}
		Every $\Qq$-semiadditive $\Bb$-category $\Cc$ admits $\Qq$-limits.
	\end{corollary}
	\begin{proof}
		By Lemma~\ref{lemma:colim-enough-truncated}${}^\op$ it is enough to show that it has $\Qq_{\le n}$-limits for all $n\ge-2$.

		Let $q\colon A\to B$ be a map in $\Qq_{\le n}$. By $\Qq$-semiadditivity, we know that $q^*$ has a right adjoint $q_*$, so it only remains to verify the Beck--Chevalley condition, i.e.~that for every pullback
		\begin{equation*}
			\begin{tikzcd}
				A'\arrow[d, "q'"'] \arrow[r, "f'"] \drar[pullback] & A\arrow[d,"q"]\\
				B'\arrow[r, "f"'] & B
			\end{tikzcd}
		\end{equation*}
		the Beck--Chevalley map $\BC_*\colon f^*q_*\to q'_*f^{\prime*}$ is an equivalence. However, this follows immediately from Corollary~\ref{cor:norm-vs-restrictions} by $2$-out-of-$3$.
	\end{proof}

	In the same way one shows (using Lemma~\ref{lemma:cocont-enough-truncated} and its dual):

	\begin{corollary}
		\label{cor:Functor_Preserves_Limits_Iff_Preserves_Colimits}
	A functor between $\Qq$-semiadditive $\Bb$-categories preserves $\Qq$-limits if and only if it preserves $\Qq$-colimits.\qed
	\end{corollary}

	\begin{definition}
		A functor $F\colon\Cc\to\Dd$ of $\Qq$-semiadditive $\Bb$-categories is called \emph{$\Qq$-semiadditive} if it preserves $\Qq$-colimits or, equivalently, $\Qq$-limits. We write $\Cat(\Bb)^{\Qq\text-\oplus}$ for the category of $\Qq$-semiadditive $\Bb$-categories and $\Qq$-semiadditive functors. Given $\Cc,\Dd\in\Cat(\Bb)^{\Qq\text-\oplus}$, we write $\ul\Fun^{\Qq\text-\oplus}(\Cc,\Dd)\coloneqq\ul\Fun^{\Qq\text-\times}(\Cc,\Dd)=\ul\Fun^{\Qq\text-\amalg}(\Cc,\Dd)$.
	\end{definition}

	\subsection{Presheaf topoi}
	\label{subsec:Presheaf_Topoi}
	For the applications we have in mind, we are mainly interested in the case where the topos $\Bb$ is a presheaf topos $\PSh(T)$ on some small category $T$, so that $\Bb$-categories correspond to $T$-categories $T\catop \to \Cat$ by \Cref{rmk:Limit_Extension}. In this case, the local classes $\Qq$ that appear in practice are usually generated by a much smaller collection of morphisms, and the condition of $\Qq$-semiadditivity of a $T$-category simplifies accordingly. We suggestively refer to these smaller classes as `pre-inductible':

	\begin{definition}[Pre-inductible subcategory]
		Let $T$ be a small category and let $Q \subseteq \PSh(T)$ be a replete subcategory containing all representable presheaves. We say that $Q$ is \textit{pre-inductible} if the following conditions are satisfied:
		\begin{enumerate}
			\item (Locality) Consider a morphism $q\colon A \to B$ in $\PSh(T)$ with $B \in Q$. Then $q$ lies in $Q$ if and only if for every pullback square
			\[
			\begin{tikzcd}
				A' \dar[swap]{q'} \rar{f} \drar[pullback] & A \dar{q} \\
				B' \rar{g} & B
			\end{tikzcd}\]
			in $\PSh(T)$ with $B' \in T$ the base change $q'$ lies in $Q$.
			\item (Diagonals) For every morphism $q$ in $Q$, also its diagonal $\Delta_q$ lies in $Q$.
			\item (Truncation) Every morphism in $Q$ with target in $T$ is truncated.
		\end{enumerate}
	\end{definition}

	\begin{remark}
		\label{rmk:Preinductible_Closed_Under_Base_Change}
		The first axiom together with the pasting law implies that $Q$ is closed under base change along maps $f\colon A\to B$ such that $A,B\in Q$ (but $f$ need not be a map in $Q$).
	\end{remark}

	There are two `extreme' cases of pre-inductible subcategories:

	\begin{example}
		\label{ex:Inductible_Is_Preinductible}
		Every inductible subcategory $Q \subseteq T$ is pre-inductible when regarded as a subcategory of $\PSh(T)$: condition (3) holds by assumption and conditions (1) and (2) are a consequence of the fact that the Yoneda embedding preserves pullbacks.
	\end{example}

	\begin{example}
		Every locally inductible subcategory $\Qq \subseteq \PSh(T)$ is in particular pre-inductible: conditions (1) and (2) hold by assumption and condition (3) follows from the fact that every locally truncated map with representable target $A \in T$ is already truncated: any cover $(A_i\to A)_{i\in I}$ of $A$ has to contain a map $A_i\to A$ hitting the component of $\id_A$, so that already $A_i\to A$ itself is an effective epimorphism, implying the claim.
	\end{example}

	\begin{definition}[$Q$-semiadditivity]
		\label{def:Q-semiadditivity_preinductible}
		Let $Q\subseteq \PSh(T)$ be a pre-inductible subcategory, and let $\Aa \subseteq \PSh(T)$ be the full subcategory spanned by the objects of $Q$. For a $T$-category $\Cc$, we denote by $\Cc\vert_{\Aa\catop}\colon \Aa\catop \to \Cat$ its right Kan extension along $T\catop \hookrightarrow \Aa\catop$, or equivalently the restriction to $\Aa\catop$ of $\Cc\colon \PSh(T)\catop \to \Cat$.
		\begin{enumerate}
			\item We say that $\Cc$ is \textit{$Q$-cocomplete} if $\Cc\vert_{\Aa\catop}$ is $Q$-cocomplete in the sense of \Cref{def:Q_Colimits};
			\item We say that $\Cc$ is \textit{$Q$-semiadditive} if in addition every truncated morphism $q$ in $Q$ is $\Cc\vert_{\Aa\catop}$-ambidextrous.
		\end{enumerate}
	\end{definition}

	\begin{remark}
		If $Q = \Qq \subseteq \PSh(T)$ is in fact locally inductible then $\Aa = \PSh(T)$ and one observes that \Cref{def:Q-semiadditivity_preinductible} specializes to \Cref{def:Q_Semiadditivity}.
	\end{remark}

	\begin{remark}
		If $T = \Bb$ happens to be a topos and $Q = \Qq\subseteq\Bb$ is an inductible local class, then after passing to a larger universe we may regard $\Qq$ as a pre-inductible subcategory of $\PSh(\Bb)$ by \Cref{ex:Inductible_Is_Preinductible}. In this case we get $\Aa = \Bb$, and we see that a $\Bb$-category $\Cc$ is $\Qq$-semiadditive in the sense of Definition~\ref{def:Q_Semiadditivity} if and only if its underlying functor $\Bb\catop \to \Cat$ is $\Qq$-semiadditive in the sense of \Cref{def:Q-semiadditivity_preinductible}.
	\end{remark}

	 The main reason for introducing $Q$-semiadditivity for pre-inductible $Q$ is the flexibility of this setup: essentially all examples of parametrized semiadditivity provided in \Cref{subsec:Examples} below will be of this form. We will now show that this setup is indeed a special case of our general formalism of $\Qq$-semiadditivity for locally inductible $\Qq$.

	\begin{construction}
		Let $Q \subseteq \PSh(T)$ be a pre-inductible subcategory. A morphism $q\colon A \to B$ in $\PSh(T)$ is said to be \textit{locally in $Q$} if for every morphism $B' \to B$ in $\PSh(T)$ with $B' \in Q$ we have that the base change map $A \times_B B' \to B'$ lies in $Q$. Since such morphisms are clearly closed under composition and contain all equivalences, they determine a wide subcategory $Q_{\loc}$ of $\PSh(T)$. We refer to $Q_{\loc}$ as the \textit{locally inductible subcategory generated by $Q$}.
	\end{construction}

	\begin{remark}
	If $Q = \Qq$ is already locally inductible, then we have $\Qq_{\loc} = \Qq$.
	\end{remark}

	\begin{lemma}
		For every pre-inductible subcategory $Q \subseteq \PSh(T)$, the wide subcategory $Q_{\loc} \subseteq \PSh(T)$ is locally inductible.
	\end{lemma}
	\begin{proof}
		It is easy to check that $Q_{\loc}$ is closed under base change and composition, and it is local by the same argument as in Remark~\ref{rk:q-loc-colim}.

		For a morphism $q\colon A \to B$ in $Q_{\loc}$ we may cover $B$ by representable objects so that assumption (3) immediately implies that $q$ is locally truncated. It remains to show that $Q_{\loc}$ is closed under diagonals. By \Cref{lem:CharacterizationClosedUnderDiagonals}, we may equivalently show that $Q_{\loc}$ is left-cancelable: if $p\colon A \to B$ and $q\colon B \to C$ are morphisms of presheaves on $T$ such that $q$ and $qp$ are in $Q_{\loc}$, then also $p$ must be in $Q_{\loc}$. In other words, given a morphism $b\colon B' \to B$ in $\PSh(T)$ with $B' \in Q$, we have to show that the base change $p'\colon A \times_B B' \to B'$ is in $Q$. To this end, consider the following commutative pullback diagram:
		\[
		\begin{tikzcd}
			A \times_B B' \dar[swap]{p'} \rar{1 \times_q 1} \drar[pullback] & A \times_C B' \dar[swap]{p''} \rar{\pr_1} \drar[pullback] & A \dar{p} \\
			B' \rar{(b,1)} \drar["="'] & B \times_C B' \dar[swap]{q'} \rar{\pr_1} \drar[pullback,xshift=-.8pt] & B \dar{q} \\
			& B' \rar{qb} & C.
		\end{tikzcd}
		\]
		Since $q$ and $qp$ are locally in $Q$, the morphisms $q'$ and $q'p''$ are in $Q$. As $Q$ is closed under diagonals, \Cref{lem:CharacterizationClosedUnderDiagonals} implies that also $p''$ is in $Q$, and hence $p'$ is in $Q$ by \Cref{rmk:Preinductible_Closed_Under_Base_Change}. This finishes the proof.
	\end{proof}

	The following is the main result of this subsection:

	\begin{proposition}
		\label{prop:Check_Semiadditivity_Locally}
		Let $Q \subseteq \PSh(T)$ be a pre-inductible subcategory.
		\begin{enumerate}
			\item A $T$-category $\Cc$ is $Q$-cocomplete in the sense of \Cref{def:Q-semiadditivity_preinductible} if and only if its limit-extension $\Cc\colon \PSh(T)\catop \to \Cat$ is $Q_{\loc}$-cocomplete in the sense of \Cref{def:Q_Colimits}.
			\item A $T$-category $\Cc$ is $Q$-semiadditive in the sense of \Cref{def:Q-semiadditivity_preinductible} if and only if its limit-extension $\Cc\colon \PSh(T)\catop \to \Cat$ is $Q_{\loc}$-semiadditive in the sense of \Cref{def:Q_Semiadditivity}.
		\end{enumerate}
	\end{proposition}

	\begin{proof}
		For the first statement, the `if'-part is clear; for the other direction, we note that it even suffices to check the existence of adjoints and the Beck--Chevalley conditions after restricting to maps in $Q_\loc$ with representable target by \cite{CLL_Global}*{Remark~2.3.15}.

		For the second statement, the `if'-direction is again clear. For the `only if'-direction, we will argue by induction that $\Cc$ is $(Q_\loc)_{\le n}$-semiadditive for all $n\ge-2$.

		For $n=-2$ there is nothing to show. Now assume that we already know that $\Cc$ is $(Q_\loc)_{\le n-1}$-semiadditive. By assumption, the restriction along any $n$-truncated $q\in Q$ has a right adjoint $q_*$, and $\Nmadj\colon q^*q_!\to\id$ adjoins to an equivalence $q_!\to q_*$. Arguing as in Corollary~\ref{cor:Semiadditive_Cats_Have_QLimits}, we deduce from Corollary~\ref{cor:norm-vs-restrictions} that $q_*$ satisfies base change along maps in $\Aa$, so \cite{CLL_Global}*{Remark~2.3.15${}^\op$} shows that $\Cc$ is $(Q_\loc)_{\le n}$-complete. Given now a general map $q\colon A\to B$ in $\Qq$, Corollary~\ref{cor:norm-vs-restrictions} shows that $f^*\Nm_q$ agrees up to equivalence with the norm along $q'\coloneqq f^*(q)$; in particular, $f^*\Nm_q$ is invertible whenever $B'$ is representable (so that $q'\in Q$). Covering $B$ by representables, we see that $\Nm_q$ itself is invertible, as desired.
	\end{proof}

	\subsection{Examples}
	\label{subsec:Examples}
	We will now provide various examples of pre-inductible subcategories and discuss their associated notion of semiadditivity. Let us start with the non-parametrized examples:

	\begin{example}[Ordinary semiadditivity]\label{ex:first-Q-example}
		\label{ex:Finite_Sets}
		The subcategory $\Fin \subseteq \Spc$ of finite sets is pre-inductible. A category $\Cc$ is $\Fin$-semiadditive if and only if it is semiadditive in the classical sense.
	\end{example}

	\begin{example}[$m$-semiadditivity]
	\label{ex:m_finite_Spaces}
		Given an integer $-2 \leq m < \infty$, recall that a space is called \textit{$m$-finite} for $-2 \leq m < \infty$ if it is $m$-truncated, has finitely many path components, and all its homotopy groups are finite. The subcategory $\Spc_m \subseteq \Spc$ by the $m$-finite spaces is pre-inductible, and a category $\Cc$ is $\Spc_m$-semiadditive if and only if $\Cc$ is $m$-semiadditive in the sense of \cite{hopkins2013ambidexterity}*{Definition~4.4.2}.
	\end{example}

	\begin{example}[$\infty$-semiadditivity]
		Recall that a space is called \textit{$\pi$-finite} if it is $m$-finite for some integer $m$. The subcategory $\Spc_{\pi} \subseteq \Spc$ of $\pi$-finite spaces is pre-inductible, and the associated notion of semiadditivity is that of \textit{$\infty$-semiadditivity} \cite{CSY20}*{Definition~3.1.10}: a category $\Cc$ is $\infty$-semiadditive if and only if it is $m$-semiadditive for all $m \geq -2$.
	\end{example}

	\begin{example}[$p$-typical $m$-semiadditivity]
		\label{ex:p_typical}
		As a variation on the previous two examples, let $p$ be a prime and let $\Spc^{(p)}_m \subseteq \Spc_m$ be the full subcategory consisting of the $m$-finite $p$-spaces, i.e.\ those $m$-finite spaces all of whose homotopy groups are $p$-groups. Then $\Spc^{(p)}_m$ is pre-inductible, and the corresponding notion of semiadditivity is that of \textit{$p$-typical $m$-semiadditivity} \cite{CSY2021AmbiHeight}*{Definition~3.1.1}. Working with $\Spc^{(p)}_{\pi}$, the category of $\pi$-finite $p$-spaces, similarly gives the notion of \textit{$p$-typical $\infty$-semiadditivity}.
	\end{example}

	\begin{example}
		\label{ex:Q-semiadd-spc}
		As the common generalization of the previous examples, let $Q \subseteq \Spc$ be a full subcategory of truncated spaces which is closed under base change and extensions and which satisfies $1 \in Q$. Then $Q$ is pre-inductible, and a category is $Q$-semiadditive if and only if it admits $A$-shaped limits and colimits for every $A\in Q$ and the norm map $\Nm_A\colon\colim_A\to\lim_A$ is an equivalence for each such $A$.

		 In fact, this is the most general form of semiadditivity our formalism provides in the non-parametrized setting: given an arbitrary locally inductible subcategory $\Qq \subseteq \Spc$, the full subcategory $Q := \Qq_{/1} \subseteq \Spc_{/1} = \Spc$ satisfies the above assumptions, and since $\Qq = Q_{\loc}$ we see that $\Qq$-semiadditivity agrees with $Q$-semiadditivity by \Cref{prop:Check_Semiadditivity_Locally}.
	\end{example}

	It turns out that the individual categories of a parametrized semiadditive category inherit some degree of non-parametrized semiadditivity.

	\begin{lemma}[Fiberwise semiadditivity]\label{lem:fiberwise-semiadd}
		Let $\Qq$ be a locally inductible subcategory of a topos $\Bb$ and consider the full subcategory $Q_{\fib} \subseteq \Spc$ consisting of those spaces $A$ which are truncated and for which the map $\colim_{A} 1 \to 1$ in $\Bb$ is contained in $\Qq$.
		\begin{enumerate}[(1)]
			\item The subcategory $Q_{\fib} \subseteq \Spc$ is pre-inductible;
			\item Every $\Qq$-semiadditive $\Bb$-category $\Cc\colon \Bb\catop \to \Cat$ is fiberwise $Q_{\fib}$-semiadditive, i.e.\ factors through the (non-full) subcategory $\Cat^{Q_{\fib}\text{-}\oplus}$ of $Q_{\fib}$-semiadditive categories.
		\end{enumerate}
	\end{lemma}
	\begin{proof}
		Denote by $L\colon\Spc\to\Bb$ the unique left exact left adjoint, given on objects by sending a space $A$ to $\colim_A 1 \in \Bb$. Note that $Q_{\fib}$ consists precisely of those truncated spaces $A$ such that the canonical map $L(A) \to L(\pt) = 1$ is in $\Qq$. Since $L$ preserves colimits and finite limits, it follows that $Q_{\fib}$ contains the point and is closed under base change and extensions, hence it is pre-inductible by \Cref{ex:Q-semiadd-spc}.

		Given now any object $X\in\Bb$, the functor $X\times L(\blank)$ is again a pullback-preserving left adjoint. Since $\Qq$ is closed under base change, we see that $X \times L(A) \to X$ is in $\Qq$ for all $A \in Q_{\fib}$, and thus by left-cancelability and locality of $\Qq$ we deduce that $X \times L(\blank)$ maps all morphisms of $(Q_{\fib})_{\loc} \subseteq \Spc$ to morphisms in $\Qq$. By Remark~\ref{rk:norm-res-pb} we conclude that the category $\Cc(X)$ is $Q_{\fib}$-semiadditive. Moreover, if $f\colon X\to Y$ is any map in $\Bb$, then the base change condition for $\Cc$ shows that $f^*\colon\Cc(Y)\to\Cc(X)$ preserves $A$-indexed (co)limits for $A \in Q_{\fib}$. It follows that $\Cc$ factors through $\Cat^{Q_{\fib}\text{-}\oplus}$, finishing the proof.
	\end{proof}

	\begin{example}
		If $\Qq \subseteq \Bb$ is a locally inductible subcategory containing the map $1\amalg 1\to 1$ (hence all fold maps $X\amalg X\to X$), then each $\Cc(X)$ is semiadditive in the usual sense, and each $f^*\colon\Cc(Y)\to\Cc(X)$ is a semiadditive functor.
	\end{example}

	We now come to the examples of semiadditivity that are truly parametrized.

	\begin{example}[$G$-semiadditivity]
		\label{ex:Finite_G_Sets}
		For a finite group $G$ the subcategory $\Fin_G \subseteq \Spc_G = \PSh(\Orb_G)$ of finite $G$-sets is pre-inductible. A \emph{$G$-category} $\Cc \colon \Orb_G\catop \to \Cat$ is $\Fin_G$-semiadditive if and only if $\Cc$ is $G$-semiadditive in the sense of \cite{nardin2016exposeIV, QuigleyShay2021Tate}.
	\end{example}

	\begin{example}[Equivariant semiadditivity]\label{ex:last-semiadd-ex}\label{ex:equiv_sadd}
		Consider the subcategory $\Glo \subseteq \FinGrpd$ of the $(2,1)$-category $\FinGrpd$ of finite groupoids spanned by the \textit{connected} finite groupoids, i.e.\ groupoids of the form $BG$ for a finite group $G$. Consider also the wide subcategory $\FinGrpdfaith$ of $\FinGrpd$ spanned by the faithful functors between finite groupoids. Identifying $\FinGrpd$ with the full subcategory of $\PSh(\Glo)$ spanned by the finite disjoint unions of representable objects, the resulting subcategory $\FinGrpdfaith$ of $\PSh(\Glo)$ is pre-inductible. A \emph{global category} $\Cc\colon \Glo\catop \to \Cat$ is $\FinGrpdfaith$-semiadditive if and only if it is \textit{equivariantly semiadditive} in the sense of \cite{CLL_Global}*{Example~4.5.2}.
	\end{example}

	\begin{example}[Global semiadditivity]\label{ex:global-semiadditivity}
		In fact, also the full subcategory $\FinGrpd \subseteq \PSh(\Glo)$ is pre-inductible. We will refer to the associated notion of semiadditivity as \textit{global semiadditivity}. Informally, the difference to the notion from the previous example is that we now require that for \emph{any} homomorphism $\alpha\colon H\to G$ of finite groups the restriction functor $\alpha^*$ admits both adjoints and that they agree, instead of just requiring this for subgroup inclusions.

		The notion of global semiadditivity may be seen as a generalization of the notion of 1-semiadditivity from \Cref{ex:m_finite_Spaces} as follows: Given a (non-parametrized) category $\Cc$ we may form its \textit{Borelification} $\Cc^\flat$, i.e.\ the global category defined via $\Cc^\flat(BG) \coloneqq \Fun(BG,\Cc)$; here we use the canonical embedding $\Glo \subseteq \FinGrpd \hookrightarrow \Spc$. Since the essential image of the inclusion functor $\FinGrpd \hookrightarrow \Spc$ is precisely the subcategory of 1-finite spaces, one observes that a category $\Cc$ is $1$-semiadditive if and only if its Borelification is globally semiadditive, also cf.~\cite{CLL_Adams}*{Lemma~5.9}. In this sense, global semiadditivity generalizes $1$-semiadditivity (see also Remark~\ref{rk:1-semiadd-embed}).
	\end{example}

	\begin{example}[$P$-semiadditivity]
		\label{ex:Finite_P_Sets}
		As a common generalization of Examples~\ref{ex:Finite_Sets}, \ref{ex:Finite_G_Sets}, and \ref{ex:equiv_sadd} (but \emph{not} of the previous example), let $T$ be a small category and let $P \subseteq T$ be an atomic orbital subcategory, in the sense of \cite{CLL_Global}*{Definition~4.3.1}. Let $\finTsets \subseteq \PSh(T)$ be the full subcategory of $\PSh(T)$ spanned by finite disjoint unions of representable presheaves, and let $\finPsets \subseteq \finTsets$ be the wide subcategory consisting of finite disjoint unions of morphisms of the form $\bigsqcup_{i=1}^n p_i\colon \bigsqcup_{i=1}^n A_i \to B$, where each morphism $p_i\colon A_i \to B$ lies in $P$. Then the subcategory $\finPsets \subseteq \PSh(T)$ is pre-inductible. A $T$-category $\Cc$ is $\finPsets$-semiadditive if and only if it is $P$-semiadditive in the sense of \cite{CLL_Global}*{Definition~4.5.1}.
	\end{example}

	\begin{example}[Very $G$-semiadditive $G$-procategories]\label{ex:last-Q-example}
		\label{ex:Quasi-finite_G_Sets}
	Let $G$ be an arbitrary group. We denote by $\smash{\widehat{\Orb}}_G \subseteq \Orb_G$ the full subcategory spanned by the orbits of the form $G/H$ where $H$ is a finite-index subgroup of $G$.\footnote{This subcategory is denoted $\Orb_G$ by \cite{KMN2023Polygonic}.} We refer to functors $\Cc\colon \smash{\widehat{\Orb}}_G\catop \to \Cat$ as \textit{$G$-procategories}.

	In \cite{Kaledin2022Mackey}*{Definition~3.1}, Kaledin considers $G$-sets $S$ satisfying the following two conditions:
	\begin{enumerate}
		\item For every $s \in S$ the stabilizer subgroup $G_s \subseteq G$ is cofinite.
		\item Every cofinite subgroup $H \subseteq G$ the fixed point set $S^H$ is finite.
	\end{enumerate}
	Following \cite{KMN2023Polygonic}, we will call such $G$-sets \textit{quasi-finite}, and write $\QFin_G$ for the full subcategory of $\Set_G$ spanned by them. Assigning to $S$ the presheaf $G/H \mapsto S^H$ determines a fully faithful functor $\QFin_G \hookrightarrow \PSh(\smash{\widehat{\Orb}}_G)$ which exhibits $\QFin_G$ as a pre-inductible subcategory. We say a profinite $G$-category $\Cc$ is \textit{very $G$-semiadditive} if it is $\QFin_G$-semiadditive.
	\end{example}

	\begin{example}[Tempered ambidexterity]
	Let $\Tt$ be a full subcategory of $\Glo$ containing the final object $1$, and consider the category $\PSh(\Tt)$.  We write $R\colon \Spc\rightarrow \PSh(\Tt)$ for the fully faithful right adjoint of $\ev_1\colon \PSh(\Tt)\rightarrow \Spc$. We observe that $R(\Spc_{\pi})$ is a pre-inductible subcategory of $\PSh(\Tt)$. To see this, note that the category $\Tt$ is a full subcategory of $\Spc$, and so by Yoneda's lemma $R(BG)$ is equivalent to the representable object associated to $B G\in \Tt$. The remaining properties of a pre-inductible subcategory are inherited from $\pi$-finite spaces, using that $R$ preserves limits. In \cite{Ell3}, Lurie considers the case where $\Tt$ is the full subcategory of $\Glo$ spanned by the groupoids with abelian isotropy. The main result of \cite{Ell3} shows that the $\Tt$-category of tempered local systems associated to an oriented $\mathbf{P}$-divisible group is $R(\Spc_{\pi})$-semiadditive.
	\end{example}

	\begin{example}
	In \cite{Scholze2023SixFunctors}*{Lecture~6}, Scholze defined for every six-functor formalism $\Dd$ notions of \textit{cohomologically proper} and \textit{cohomologically \'etale} morphisms $f$: roughly speaking, this condition demands that the functor $f_!$ given by the covariant functoriality of the six-functor formalism is \textit{right} (resp.\ \textit{left}) adjoint to the morphism $f^*$ coming from the contravariant functoriality in some preferred way.

	Only remembering the contravariant functoriality, every six-functor formalism $\Dd$ forgets to a category parametrized by some category $T$. As we will show in future work, the class $Q$ of maps in $T$ that are both cohomologically \'etale and cohomology proper form an inductible subcategory of $T$, and $\Dd$ is $Q$-semiadditive in the sense of \Cref{def:Q-semiadditivity_preinductible}.
	\end{example}

	\subsection{Alternative characterizations of \texorpdfstring{$\bm\Qq$}{Q}-semiadditivity}
	\label{subsec:Characterizations_Q_Semiadditivity}
	Recall that ordinary semiadditivity of a category $\Cc$ may be characterized by asking $\Cc$ to admit finite products and finite coproducts, and requiring these to commute with each other. The following result provides an analogous characterization for $\Qq$-semiadditivity:

	\begin{proposition}
		\label{prop:Characterization_Q_Semiadditivity}
		Let $\Cc$ be a $\Qq$-complete and $\Qq$-cocomplete $\Bb$-category. The following are equivalent:
		\begin{enumerate}
			\item The category $\Cc$ is $\Qq$-semiadditive.
			\item The category $\Cc^\op$ is $\Qq$-semiadditive.
			\item For every truncated $q\colon A\rightarrow B$ in $\Qq$, the functor $q_!\colon \ulFun(\ul{A},\pi^*_B\Cc)\rightarrow \pi^*_B\Cc$ preserves $\Qq$-limits.
			\item For every truncated $q\colon A\rightarrow B$ in $\Qq$, the functor $q_*\colon \ulFun(\ul{A},\pi^*_B\Cc)\rightarrow \pi^*_B\Cc$ preserves $\Qq$-colimits.
			\item For every pullback square
			\[
			\begin{tikzcd}
				A' \dar[swap]{q'} \rar{p'} \drar[pullback] & A \dar{q} \\
				B' \rar{p} & B
			\end{tikzcd}
			\]
			consisting of truncated maps in $\Qq$ the double Beck--Chevalley transformation
			$\BC_{!,*}\colon p_!q'_*\rightarrow q_*p'_!$ is an equivalence.
			\item For every truncated map $q\colon A\rightarrow B$ in $\Qq$ the double Beck--Chevalley transformation $\BC_{!,*}\colon q_!{\pr_1}_* \rightarrow q_*{\pr_2}_!$ associated to the pullback square
			\[
			\begin{tikzcd}
				A \times_B A \dar[swap]{\pr_1} \rar{\pr_2} \drar[pullback] & A \dar{q} \\
				A \rar{q} & B.
			\end{tikzcd}
			\] in $\Qq$ is an equivalence.
		\end{enumerate}
	\end{proposition}

	We may think of conditions (3)--(5) as expressing that `$\Qq$-limits commute with $\Qq$-colimits.'

	\begin{proof}
		We will prove that $(1)\Rightarrow(3)\Rightarrow(5)\Rightarrow(6)\Rightarrow(1)$. Dually, we then have $(2)\Rightarrow(4)\Rightarrow(5)\Rightarrow(6)\Rightarrow(2)$, so that all statements are indeed equivalent.

		For $(1)\Rightarrow(3)$, note that $q_!$ is a left adjoint, so it preserves $\Qq$-colimits by \cite{martiniwolf2021limits}*{Proposition~5.2.5}, hence $\Qq$-limits by Corollary~\ref{cor:Functor_Preserves_Limits_Iff_Preserves_Colimits}. The implication $(3)\Rightarrow(5)$ simply amounts to spelling out the definition of preserving $\Qq$-limits, while $(5)\Rightarrow(6)$ is immediate. Finally, for $(6)\Rightarrow(1)$, we will be done by induction if we show that for any truncated $q$ the norm map $\Nm_q\colon q_!\to q_*$ can be factored as the composite
		\[
		q_! \simeq q_!{\pr_1}_*\Delta_* \xrightarrow{\Nm_{\Delta}^{-1}} q_!{\pr_1}_*\Delta_! \xrightarrow{\BC_{!,*}} q_*{\pr_2}_!\Delta_! \simeq q_*.
		\]
		Up to replacing $\Cc$ by $\Cc\catop$, the proof of this claim is identical to that of \cite{CLL_Global}*{Lemma~4.4.2} and will hence be omitted.
	\end{proof}

	\begin{remark}
		We will later prove that one can equivalently drop all the truncatedness assumptions, see Corollary~\ref{cor:Characterization_Q_Semiadditivity_nontruncated}.
	\end{remark}

	\begin{remark}
		Assume that $\Bb=\PSh(T)$ and $\Qq=Q_{\loc}$ for some pre-inductible $Q\subset\PSh(T)$. Recall that if $q$ is any map in $Q$, then so is its diagonal $\Delta$ by left cancelability. Thus, combining the inductive argument from the implication $(6)\Rightarrow(1)$ above with Proposition~\ref{prop:Check_Semiadditivity_Locally}, we see that it suffices to check invertibility of the double Beck--Chevalley transformation $\BC_{!,*}\colon q_!{\pr_1}_* \rightarrow q_*{\pr_2}_!$ for $q$ contained in $Q$ (as opposed to all of $Q_\loc$). It follows directly that we can also restrict to maps contained in $Q$ in conditions (3), (4), and (5).
	\end{remark}

	\section{Parametrized span categories}
	In this section, we will introduce parametrized span categories and prove their semiadditivity properties.

	We begin by briefly recalling Barwick's \emph{span categories} \cite{barwick2017spectral}:

	\begin{definition}
		An \textit{adequate triple} is a triple $(\Cc,\Cc_L,\Cc_R)$ where $\Cc$ is a category and $\Cc_L,\Cc_R \subseteq \Cc$ are wide subcategories such that every (solid) cospan
		\[
			\begin{tikzcd}
				A\arrow[r, "r'",dashed]\arrow[d,"\ell'"',dashed] & B\arrow[d,"\ell"]\\
				C\arrow[r,"r"'] & D
			\end{tikzcd}
		\]
		with $\ell\in\Cc_L$ and $r\in\Cc_R$ can be completed to a pullback square, and such that for any such pullback we have $r'\in\Cc_R$ and $\ell'\in\Cc_L$. We define a \emph{morphism of adequate triples} $(\Cc,\Cc_L,\Cc_R)\to(\Dd,\Dd_L,\Dd_R)$ to be a functor $F\colon\Cc\to\Dd$ with $F(\Cc_L)\subseteq\Dd_L$, $F(\Cc_R)\subseteq\Dd_R$, and preserving pullbacks of maps in $\Cc_L$ along maps in $\Cc_R$. The category of adequate triples will be denoted by $\AdTrip$.
	\end{definition}

	\begin{construction}
		We may associate to any adequate triple $(\Cc,\Cc_L,\Cc_R)$ a \textit{span category} $\Span(\Cc,\Cc_L,\Cc_R)$ which may informally be described as follows:
		\begin{itemize}
			\item The objects of $\Span(\Cc,\Cc_L,\Cc_R)$ are the objects of $\Cc$;
			\item A morphism from $X$ to $Y$ is a \textit{span} $X \xleftarrow{l} U \xrightarrow{r} Y$ where $l \in \Cc_L$ and $r \in \Cc_R$;
			\item The composite of a span $X \xleftarrow{l} U \xrightarrow{r} Y$ with a span $Y \xleftarrow{l'} V \xrightarrow{r'} Z$ is given by the pullback span $X \leftarrow U \times_Y V \to Z$.
		\end{itemize}
		For a formal definition of the functor
		\[
			\Span \colon \AdTrip \to \Cat
		\]
		we refer to \cite{HHLNa}*{Definition~2.12}. This functor has the following properties:
		\begin{enumerate}
			\item There is a natural equivalence $\Span(\Cc,\Cc_L,\Cc_R)\catop \simeq \Span(\Cc,\Cc_R,\Cc_L)$ \cite[Lemma~2.14]{HHLNa};
			\item The functor $\Span(-)$ admits a left adjoint by \cite[Theorem~2.18]{HHLNa}, hence preserves all limits (computed pointwise in $\AdTrip$);
			\item There are natural inclusions of wide subcategories
			\[
				\Cc_L\catop \hookrightarrow \Span(\Cc,\Cc_L,\Cc_R) \qquad \text{ and } \qquad \Cc_R \hookrightarrow \Span(\Cc,\Cc_L,\Cc_R).
			\]
		\end{enumerate}
	\end{construction}

	\begin{definition}
		An \textit{adequate triple of $\Bb$-categories} is a $\Bb$-category $\Cc$ equipped with two wide subcategories $\Cc_L$ and $\Cc_R$ such that $(\Cc(A),\Cc_L(A), \Cc_R(A))$ is an adequate triple for every $A \in \Bb$ and such that for every $f\colon A'\to A$ in $\Bb$ the restriction functor $f^*$ is a map of adequate triples; equivalently, this is a limit-preserving functor $\Bb\catop \to \AdTrip$. We define the \textit{parametrized span category} $\ul{\Span}(\Cc,\Cc_L,\Cc_R)$ of this triple as the following composite:
		\[
			\Bb\catop \xrightarrow{(\Cc,\Cc_L,\Cc_R)} \AdTrip \xrightarrow{\Span} \Cat.
		\]
		Since $\Span(-)$ preserves limits this is indeed a $\Bb$-category.
	\end{definition}

	We will next concern ourselves with parametrized limits and colimits in parametrized span categories. In the non-parametrized setting, recall that for an extensive category $\Cc$ (one that admits finite coproducts which suitably interact with pullbacks, to be recalled in Definition~\ref{defi:extensive}), its span category $\Span(\Cc)$ admits products and coproducts which are both computed as coproducts in $\Cc$; in particular $\Span(\Cc)$ is semiadditive. We will now prove the analogous statements for parametrized span categories.

	\begin{definition}[Parametrized extensiveness]
		Let $(\Cc,\Cc_L,\Cc_R)$ be an adequate triple of $\Bb$-categories, and let $\Qq$ be a local class of morphisms $\Qq$ in $\Bb$. We will say that the triple $(\Cc,\Cc_L,\Cc_R)$ is \textit{left $\Qq$-extensive} if the following conditions are satisfied:
		\begin{enumerate}
			\item The category $\Cc$ is $\Qq$-cocomplete;
			\item The following hold for every morphism $q\colon A \to B$ in $\Qq$:
			\begin{enumerate}
				\item The left adjoint $q_!\colon \Cc(A) \to \Cc(B)$ of $q^*$ is again a morphism of adequate triples.
				\item The unit $\eta\colon \id \to q^*q_!$ and counit $\epsilon\colon q_!q^*\to \id$ of the adjunction $q_!\colon \Cc(A) \rightleftarrows \Cc(B) \noloc q^*$ are contained in $\Cc_L$;
				\item For morphisms $r\colon X \to Y$ and $r'\colon X' \to Y'$ in $\Cc_R(A)$ and $\Cc_R(B)$, respectively, the naturality squares
				\[
				\begin{tikzcd}
					X \rar{\eta_X} \dar[swap]{r} & q^*q_!X \dar{q^*q_!r} \\
					Y \rar{\eta_Y} & q^*q_!Y
				\end{tikzcd}
				\qquad \text{ and } \qquad
				\begin{tikzcd}
					q_!q^*X' \dar[swap]{q_!q^*q'} \rar{\epsilon_{X'}} & X' \dar{r'} \\
					q_!q^*Y' \rar{\epsilon_{Y'}} & Y'
				\end{tikzcd}
				\]
				are pullback squares.
			\end{enumerate}
		\end{enumerate}
		We say the triple $(\Cc,\Cc_L,\Cc_R)$ is \textit{right $\Qq$-extensive} if $(\Cc,\Cc_R,\Cc_L)$ is left extensive, and we say it is \textit{$\Qq$-extensive} if it is both left and right $\Qq$-extensive.
	\end{definition}

	\begin{proposition}
		\label{prop:Colimits_In_Span}
		Let $\Qq$ be a local class of morphisms in $\Bb$, and let $(\Cc,\Cc_L,\Cc_R)$ be an adequate triple of $\Bb$-categories. If the triple is left $\Qq$-extensive, then the parametrized span category $\ul{\Span}(\Cc,\Cc_L,\Cc_R)$ admits $\Qq$-colimits, and the inclusion
		\[
		\Cc_R \hookrightarrow \ul{\Span}(\Cc,\Cc_L,\Cc_R)
		\]
		preserves them. Dually, if $(\Cc,\Cc_L,\Cc_R)$ is right $\Qq$-extensive, then $\ul{\Span}(\Cc,\Cc_L,\Cc_R)$ admits $\Qq$-limits, and the inclusion
		\[
		\Cc_L\catop \hookrightarrow \ul{\Span}(\Cc,\Cc_L,\Cc_R)
		\]
		preserves them.
	\end{proposition}
	\begin{proof}
		Observe that the second claim follows immediately from the first one by swapping the roles of $\Cc_L$ and $\Cc_R$ and using the equivalence $\ul{\Span}(\Cc,\Cc_L,\Cc_R)\catop \simeq \ul{\Span}(\Cc,\Cc_R,\Cc_L)$. For the first claim, fix a morphism $q\colon A \to B$ in $\Qq$, and consider the adjunction
		\[
		q_!\colon \Cc(A) \rightleftarrows \Cc(B) \noloc q^*.
		\]
		The compatibility assumptions precisely tell us that the conditions of \cite{BachmannHoyois2021Norms}*{Corollary C.21} are satisfied, so that applying $\Span(-)$ to both of these functors results in another adjunction
		\[
		\Span(q_!)\colon \Span(\Cc(A),\Cc_L(A),\Cc_R(A)) \leftrightarrows \Span(\Cc(B),\Cc_L(B),\Cc_R(B)) \noloc \Span(q^*),
		\]
		with unit and counit given by the spans
		\[
			X \xleftarrow{\id} X \xrightarrow{\eta} q^*q_!X \qquad \text{ and } \qquad q_!q^*Y \xleftarrow{\id} q_!q^*Y \xrightarrow{\epsilon} Y
		\]
		in $\Cc(A)$ and $\Cc(B)$, respectively. The Beck--Chevalley conditions for the adjunctions $q_! \dashv q^*$ then immediately imply the Beck--Chevalley conditions for the adjunctions $\Span(q_!) \dashv \Span(q^*)$, showing that $\ul{\Span}(\Cc,\Cc_L,\Cc_R)$ admits $\Qq$-colimits. Finally, it is clear from the construction that the commutative squares
		\[
		\begin{tikzcd}
			\Cc_R(B) \rar[hookrightarrow] \dar[swap]{q^*} & \Span(\Cc(B),\Cc_L(B),\Cc_R(B)) \dar{\Span(q^*)} \\
			\Cc_R(A) \rar[hookrightarrow] & \Span(\Cc(A),\Cc_L(A),\Cc_R(A))
		\end{tikzcd}
		\]
		is vertically left adjointable, which precisely amounts to the inclusion $\Cc_R \hookrightarrow \ul{\Span}(\Cc,\Cc_L,\Cc_R)$ preserving $\Qq$-colimits.
	\end{proof}

	\begin{lemma}\label{cor:test_colimit_pres_span}
		Let $(\Cc,\Cc_L,\Cc_R)$ and $\Qq$ be as in \Cref{prop:Colimits_In_Span}, and consider any $\Bb$-functor $F\colon \ul{\Span}(\Cc,\Cc_L,\Cc_R)\rightarrow \Dd$. If the triple is left $\Qq$-extensive, then $F$ preserves $\Qq$-colimits if and only if the composite $\Cc_R \xhookrightarrow{\iota} \ul{\Span}(\Cc,\Cc_L,\Cc_R)\xrightarrow{F} \Dd$ preserves $\Qq$-colimits. If the triple is right $\Qq$-extensive, the dual statement for $\Qq$-limits holds.
	\end{lemma}
	\begin{proof}
		Since the inclusion $\iota$ preserves $\Qq$-colimits by \Cref{prop:Colimits_In_Span}, the `only if'-direction is clear. For the converse, assume that $F \circ \iota$ preserves $\Qq$-colimits. We have to show that for any map $q\colon A \to B$ in $\Qq$, the Beck--Chevalley transformation $q_! F_A \Rightarrow F_B q_!$ filling the right square of the following diagram is an equivalence:
		\[\begin{tikzcd}
			{\Cc_R(A)} & {\Span(\Cc(A),\Cc_L(A),\Cc_R(A))} \dlar[Rightarrow, shorten <= 8pt, shorten >= 10pt, "\sim"{description}] & {\Dd(A)} \dlar[Rightarrow, shorten <= 16pt, shorten >= 8pt] \\
			{\Cc_R(B)} & {\Span(\Cc(B),\Cc_L(B),\Cc_R(B))} & {\Dd(B)}.
			\arrow["{q_!}"', from=1-2, to=2-2]
			\arrow["{F_A}", from=1-2, to=1-3]
			\arrow["{F_B}", from=2-2, to=2-3]
			\arrow["{q_!}", from=1-3, to=2-3]
			\arrow["{\iota_B}", from=2-1, to=2-2, hookrightarrow]
			\arrow["{q_!}"', from=1-1, to=2-1]
			\arrow["{\iota_A}", from=1-1, to=1-2, hookrightarrow]
		\end{tikzcd}\]
		Since the inclusion $\iota_A\colon \Cc_R(A)\hookrightarrow \Span(\Cc(A),\Cc_L(A),\Cc_R(A))$ is essentially surjective, it suffices to test this after precomposing by $\iota_A$. But since Beck--Chevalley transformations compose and $\iota\colon \Cc_R \hookrightarrow \ul{\Span}(\Cc,\Cc_L,\Cc_R)$ preserves $\Qq$-colimits, this follows from the assumption that $F \circ \iota$ preserves $\Qq$-colimits. This finishes the proof of the first statement. Once again the second statement is formally dual to the first.
	\end{proof}

	Let now
	\begin{equation}\label{diag:not-a-pb}
		\begin{tikzcd}
			A' \dar[swap]{q'} \rar{p'} & A \dar{q} \\
			B' \rar{p} & B
		\end{tikzcd}
	\end{equation}
	be any commutative diagram in $\Qq$ such that $q,q'\in \Qq_L$ and $p,p' \in \Qq_R$. Then we have by our specific construction for any left $\Qq_L$- and right $\Qq_R$-extensive triple $(\Cc,\Cc_L,\Cc_R)$ an equivalence
	\[
		p_!q'_*=\Span(p_!)\Span(q'_!)\simeq\Span(q_!)\Span(p'_!)=q_*p'_!
	\]
	induced by the total mate of the naturality equivalence $p'^*q^*\simeq q'^*p^*$ in $\Cc$. As the next proposition shows, this equivalence can be described purely abstractly via the double Beck--Chevalley map whenever $(\ref{diag:not-a-pb})$ is actually a pullback:

	\begin{proposition}
		\label{prop:Limits_Commute_With_Colimits_In_Spans}
		Let $(\Cc,\Cc_L,\Cc_R)$ be an adequate triple of $\Bb$-categories. Let $\Qq_L$ and $\Qq_R$ be local classes of morphisms in $\Bb$ and assume that the triple is left $\Qq_R$-extensive and right $\Qq_L$-extensive. Then $\Qq_R$-colimits commute with $\Qq_L$-limits in $\ul{\Span}(\Cc,\Cc_R,\Cc_L)$, i.e.~for every pullback square
		\[
		\begin{tikzcd}
			A' \dar[swap]{q'} \rar{p'} \drar[pullback] & A \dar{q} \\
			B' \rar{p} & B
		\end{tikzcd}
		\]
		with $q,q'\in \Qq_L$ and $p,p' \in \Qq_R$, the double Beck--Chevalley map $\BC_{!,*}\colon p_!q'_*\rightarrow q_*p'_!$ is an equivalence. More precisely, for the concrete choices of units and counits made above, it agrees with the canonical equivalence $\Span(p_!)\Span(q'_!)\simeq\Span(q_!)\Span(p'_!)$.
	\end{proposition}
	\begin{proof}
	The units and counits for the adjunctions
	\begin{align*}
		p_!\colon\ul\Span(\Cc,\Cc_L,\Cc_R)(B')&\rightleftarrows \ul\Span(\Cc,\Cc_L,\Cc_R)(B)\noloc p^*\\
		p'_!\colon\ul\Span(\Cc,\Cc_L,\Cc_R)(A')&\rightleftarrows \ul\Span(\Cc,\Cc_L,\Cc_R)(A)\noloc p'^*
	\end{align*}
	constructed above are simply the forward maps associated to the corresponding units and counits on $\Cc$. Thus, the Beck--Chevalley map $p'_!q'^*\to q^*p_!$ of functors $\ul\Span(\Cc,\Cc_L,\Cc_R)(B')\to\ul\Span(\Cc,\Cc_L,\Cc_R)(A)$ is the forward map associated to the corresponding Beck--Chevalley map in $\Cc$, so its inverse agrees with the corresponding backwards map.

	Similarly, the unit and counit for $q^*\dashv q_*$ and $q'^*\dashv q_*'$ were given as the backwards map of the corresponding counit and unit, respectively in $\Cc$. Plugging this in, we see that the double Beck--Chevalley map in $\ul\Span(\Cc,\Cc_L,\Cc_R)$ is the backwards map associated to the composite
	\[
		q_!p'_!\xrightarrow{\;\eta\;} q_!p'_!q'^*q'_!\xrightarrow{\;\BC_!\;} q_!q^*p_!q'_!\xrightarrow{\;\epsilon\;}p_!q'_!
	\]
	in $\Cc$. The latter is by definition the total mate of the naturality equivalence, in particular itself invertible. Thus, the double Beck--Chevalley map can be equivalently described as the forward map associated to its inverse, as claimed.
	\end{proof}

	\begin{corollary}
		In the above situation, let $\Qq_0\subseteq\Qq_L\cap\Qq_R$ be locally inductible. Then $\ul\Span(\Cc,\Cc_L,\Cc_R)$ is $\Qq_0$-semiadditive.
	\end{corollary}
	\begin{proof}
		This is an immediate consequence of the criterion for $\Qq_0$-semiadditivity from Proposition~\ref{prop:Characterization_Q_Semiadditivity}.
	\end{proof}

	\begin{remark}
		Using the inductive description of the norm maps in terms of the double Beck--Chevalley maps given in the proof of Proposition~\ref{prop:Characterization_Q_Semiadditivity}, we see that with respect to the above choices the norm map $q_!\to q_*$ is simply the identity for every truncated morphism $q$ in $\Qq_0$. We immediately get that the adjoint norm map $\Nmadj\colon q^*q_!\to\id$ and the corresponding map $\mu\colon\id\to q_!q^*$ are just the counit and unit $q^*q_!\to\id$ and $\id\to q_!q^*$ constructed in the proof of Proposition~\ref{prop:Colimits_In_Span}, i.e.~they are given by flipping the spans representing the unit and counit, respectively, of $q_!\dashv q^*$.
	\end{remark}

	We close this section by discussing a key special case of the above construction:

	\begin{construction}[Parametrized span category]\label{constr:parametrized_spans}
		Let $\Qq$ be a left-cancelable local wide subcategory of $\Bb$. Then the functor $\Bb\catop \to \Cat, \, A \mapsto \Bb_{/A}$ restricts to a limit-preserving functor
		\[
			\ulbbU{\Qq}\colon \Bb\catop \to \Cat, \qquad A \mapsto \Qq_{/A},
		\]
		cf.\ \Cref{const:parametrized_Q}. For an object $A \in \Bb$, the full subcategory $\Qq_{/A} \subseteq \Bb_{/A}$ is closed under pullbacks by \Cref{lem:CharacterizationClosedUnderDiagonals}; thus we obtain a limit preserving functor
		\[
			(\ulbbU{\Qq},\ulbbU{\Qq},\ulbbU{\Qq})\colon\Bb^\op\to\AdTrip.
		\]
		We denote the resulting span category by $\ul\Span(\Qq)$.

		More generally, let $\Qq_L,\Qq_R\subseteq\Qq$ be left-cancelable wide local subcategories. For an object $A \in \Bb$, we denote by
		\[
			\Qq_{/A}^L \subseteq \Qq_{/A} \qquad \text{ and } \qquad \Qq_{/A}^R \subseteq \Qq_{/A}
		\]
		the wide subcategories containing all objects, but containing only those morphisms in $\Qq_{/A}$ whose underlying map in $\Qq$ lies in $\Qq_L$ or $\Qq_R$, respectively. This results in wide $\Bb$-subcategories
		\[
			\ulbbU{\Qq}^L \subseteq \ulbbU{\Qq} \qquad \text{ and } \qquad \ulbbU{\Qq}^R \subseteq \ulbbU{\Qq}.
		\]
		We also obtain a wide $\Bb$-subcategory $\ul{\Span}(\Qq,\Qq_L,\Qq_R)$ of $\ul\Span(\Qq)$ given at an object $A \in \Bb$ by
		\[
			\ul{\Span}(\Qq,\Qq_L,\Qq_R)(A) \coloneqq \Span(\Qq_{/A}, \Qq_{/A}^L, \Qq_{/A}^R) \subseteq \Span(\Qq_{/A}).
		\]
		The inclusions of the left- and right-pointing morphisms give rise to canonical inclusions
		\[
			(\ulbbU{\Qq}^L)\catop \hookrightarrow \ul{\Span}(\Qq,\Qq_L,\Qq_R) \hookleftarrow \ulbbU{\Qq}^R.
		\]
	\end{construction}

	\begin{warn}
		We warn the reader that the underlying category of $\ul{\Span}(\Qq)$ is \emph{not} equivalent to $\Span(\Qq)$, because $\Qq_{/1} \neq \Qq$.
	\end{warn}

	\begin{corollary}
		\label{cor:Span_Complete_And_Cocomplete}
		Let $\Qq_L,\Qq_R\subseteq\Qq\subseteq\Bb$ be left-cancelable wide local subcategories.
		\begin{enumerate}
			\item The $\Bb$-category $\ul{\Span}(\Qq,\Qq_L,\Qq_R)$ is both $\Qq_L$-complete and $\Qq_R$-cocomplete.
			\item The canonical inclusion $\ulbbU{\Qq^R} \hookrightarrow \ul{\Span}(\Qq,\Qq_L,\Qq_R)$ preserves $\Qq_R$-colimits.
			\item The canonical inclusion $\ulbbU{\Qq^L}\catop \hookrightarrow \ul{\Span}(\Qq^L)$ preserves $\Qq_L$-limits.
			\item $\Qq_L$-limits and $\Qq_R$-colimits in $\ul\Span(\Qq,\Qq_L,\Qq_R)$ commute in the sense of Proposition~\ref{prop:Limits_Commute_With_Colimits_In_Spans}. In particular, if $\Qq_0\subseteq\Qq_L\cap\Qq_R$ is locally inductible, then $\ul\Span(\Qq,\Qq_L,\Qq_R)$ is $\Qq_0$-semiadditive.
		\end{enumerate}
		\begin{proof}
			By the above discussion, we only have to show that $(\ulbbU{\Qq},\ulbbU{\Qq^L},\ulbbU{\Qq^R})$ is left $\Qq_R$-extensive and right $\Qq_L$-extensive. We will argue for the former, the other argument being analogous.

			For this pick any morphism $q\colon A \to B$ in $\Qq_R$, and consider the adjunction
			\[
			q_!\colon \Qq_{/A} \rightleftarrows \Qq_{/B} \noloc q^*,
			\]
			where $q_!(f\colon A' \to A) = q \circ f$.	It is clear that $q_!$ maps $\Qq^L_{/A}$ to $\Qq^L_{/B}$, $\Qq^R_{/A}$ to $\Qq^R_{/B}$, and that it even preserves all pullbacks. The counit $q_!q^*(f)\to f$ is given by the map $\pr\colon A'\times_AB\to A'$, while the unit is given by $(\id,f)\colon A'\to A'\times_BA$. The former is in $\Qq_R$ as $\Qq_R$ is closed under base change, while the latter is so as it is a base change of the diagonal $A\to A\times_BA$ and since $\Qq_R$ is also assumed to be closed under diagonals.

			Finally, one directly checks from the above descriptions that the naturality squares for the unit and counit are pullback squares.
		\end{proof}
	\end{corollary}

	\section{The universal property of parametrized spans}\label{sec:Univ_Prop_Spans}

	Throughout this section, we fix a topos $\Bb$. If $\Qq\subseteq\Bb$ is locally inductible, then Corollary~\ref{cor:Span_Complete_And_Cocomplete} in particular shows that the parametrized span category $\ul\Span(\Qq)$ is a $\Qq$-semiadditive $\Bb$-category. The goal of this section is to show that it is in fact universal among these. More precisely, if we denote by $\pt\colon \ul{1}\rightarrow \ul{\Span}(\Qq)$ the functor given in degree $B\in \Bb$ by the object $\id_B \in \Qq_{/B} \subseteq \Span(\Qq_{/B})$, we will show:

	\begin{theorem}
		\label{thm:universal_prop_par_spans}
		For every $\Qq$-semiadditive $\Bb$-category $\Dd$, restriction along the functor $\pt\colon \ul{1}\rightarrow \ul{\Span}(\Qq)$ induces an equivalence of $\Bb$-categories.
		\[
		\ul{\Fun}^{\Qq\text-\oplus}(\ul{\Span}(\Qq),\Dd) \iso \Dd.
		\]
	\end{theorem}

	As we will explain at the end of this section, this can be seen as a parametrized generalization of a theorem of Harpaz \cite{harpaz2020ambidexterity} identifying the free non-parametrized $m$-semiadditive category. Our proof of \Cref{thm:universal_prop_par_spans} is inspired by an alternative approach to this non-parametrized result due to Lior Yanovski, and we would like to thank him for sharing his ideas and notes with us.

	The first key technical ingredient needed for the proof of the universal property is an extension result for functors out of parametrized span categories, letting us increase the number of right-pointing arrows on which the functor is defined. We will present a general form of this result in \Cref{subsec:Extending_Cosegal_Functors} and specialize to span categories in \Cref{sec:Good_Inclusions_Of_Span_Categories}.  As the second key ingredient, we prove in \Cref{sec:Cosegal_Vs_Segal} that for functors into a semiadditive category this result can be dualized, allowing us to increase the number of \textit{left}-pointing arrows on which the functor is defined. Combining these ingredients, we then establish the universal property in \Cref{sec:Universal_Property}.

	\subsection{The co{S}egal condition}
	\label{subsec:Extending_Cosegal_Functors}

	This section contains one of the main technical ingredients needed for the proof of the universal property of parametrized spans: the existence of unique extensions for so-called \textit{coSegal functors}, see \Cref{prop:good_equivalence} below.

	Following \cite{shah2022parametrizedII}*{Definition 3.1} we define a \textit{factorization system} on a $\Bb$-category $\Cc$ to be a pair $(E,M)$ of wide $\Bb$-subcategories $E, M \subseteq \Cc$ such that for every $B \in \Bb$ the wide subcategories $E(B)$ and $M(B)$ of $\Cc(B)$ define a factorization system on $\Cc(B)$ in the sense of \cite{HTT}*{Definition~5.2.8.8}. We denote maps in $E$ with the symbol $\twoheadrightarrow$ and maps in $M$ with the symbol $\rightarrowtail$.

	\begin{definition}[Distinguished object]
		Let $(E,M)$ be a factorization system on a $\Bb$-category $\Cc$. We say an object $X \in \Gamma M \subseteq \Gamma \Cc$ is \textit{distinguished} if the corresponding $\Bb$-functor $X\colon \ul{1} \to M$ is fully faithful. We will frequently denote a distinguished object by $\pt$, and denote the corresponding inclusion by $\{\pt\} \hookrightarrow M$. By restriction we also obtain an object $\pt_B \coloneqq B^* \pt \in \Cc(B)$ for every $B \in \Bb$. Given an object $X\in \Cc(B)$, we refer to maps in $M$ of the form $\pt_B \rightarrowtail X$ as \textit{coSegal maps}.
	\end{definition}

	Given the inclusion $\Cc\subseteq \Dd$ of a subcategory and an object $X\in \Gamma\Cc$, we write $\Cc_{/X}$ for the pullback $\Cc\times_{\Dd} \Dd_{/X}$.

	\begin{definition}
	Let $F\colon \Cc'\rightarrow \Dd$ be a functor, and let $\Cc \subseteq  \Cc'$ be a full subcategory. We say $F$ is (pointwise) \emph{left Kan extended from $\Cc$} if for every object $A\in \Bb$ and every object $X\in \Cc'(A)$, the $\Bb_{/A}$-parametrized colimit of the composite $\pi_A^*\Cc_{/X}\rightarrow \pi_A^*\Cc\rightarrow \pi_A^*\Dd$ exists and the canonical map
	\[
	\operatorname{colim}\limits_{\pi_A^*\Cc_{/X}} \pi_A^* F \rightarrow F_A(X)
	\]
	is an equivalence.
	\end{definition}

	\begin{remark}
	It follows from \cite{martiniwolf2021limits}*{Remark 6.3.6} that a pointwise left Kan extension admits the universal property of a left Kan extension.
	\end{remark}

	\begin{definition}[CoSegal functor]
		Let $\Cc$ and $\Dd$ be $\Bb$-categories, let $(E,M)$ be a factorization system on $\Cc$, and let $\pt$ be a distinguished object. We say a $\Bb$-functor $F\colon \Cc \to \Dd$ is \textit{coSegal} if $F|_M$ is left Kan extended from $\pt$. We let $\Fun_{\coSeg}(\Cc,\Dd) \subseteq \Fun_{\Bb}(\Cc,\Dd)$ denote the full subcategory of coSegal $\Bb$-functors.
	\end{definition}

	\begin{example}
	As we will show in the next subsection, $\pt$ is a distinguished object for the standard factorization system on $\ul\Span(\Qq)$, and a functor $\ul\Span(\Qq)\to\Dd$ into a $\Qq$-cocomplete category is coSegal if and only if it is $\Qq$-cocontinuous.
	\end{example}

	Given a general subcategory $\Cc^\circ\subseteq\Cc$, it does not even make sense to ask whether the restriction of a coSegal functor on $\Cc$ to $\Cc^\circ$ is again coSegal. We therefore introduce:

	\begin{definition}
		\label{def:Compatible_With_Factorization_System}
		Let $(E,M)$ be a factorization system on an non-parametrized category $\Cc$. We say that a subcategory $\Cc^\circ\subseteq\Cc$ is \emph{compatible with $(E,M)$} if the following two conditions hold for every object $X \in \Cc^{\circ}$:
		\begin{enumerate}
			\item Every morphism $X \twoheadrightarrow Y$ in $E$ belongs to $\Cc^{\circ}$.
			\item For a morphism $e\colon X \twoheadrightarrow X'$ in $E$, a morphism $f\colon X' \to Y$ belongs to $\Cc^{\circ}$ if and only if the composite $fe\colon X \to Y$ belongs to $\Cc^{\circ}$.
		\end{enumerate}
		Given a parametrized factoriation system $(E,M)$ on some $\Bb$-category $\Cc$, we call a parametrized subcategory $\Cc^\circ\subseteq\Cc$ \emph{compatible with $(E,M)$} if each $\Cc^\circ(A)\subseteq\Cc(A)$ is compatible with $(E(A),M(A))$.
	\end{definition}

	\begin{example}
		A \emph{full} subcategory $\Cc^\circ\subseteq\Cc$ is compatible with $(E,M)$ if and only if it satisfies the following (a priori weaker) version of the first axiom: \emph{given $X\twoheadrightarrow Y$ in $E$, if $X\in \Cc^\circ$ then also $Y\in \Cc^\circ$}.
	\end{example}

	\begin{lemma}
	\label{lemma:Restricted_Factorization_System}
	If $\Cc^{\circ} \subseteq \Cc$ is compatible with $(E,M)$, then $E^\circ\coloneqq E\cap \Cc^\circ$ and $M^\circ\coloneqq M\cap \Cc^\circ$ form a factorization system on $\Cc^\circ$.
	\begin{proof}
		It suffices to prove the non-parametrized version of the statement. First note that if $f$ is a morphism in $\Cc^\circ$ with factorization $f=me$ in $(\Cc,E,M)$, then $e$ belongs to $E^\circ$ by part (1) of \Cref{def:Compatible_With_Factorization_System}, so $m$ belongs to $M^\circ$ by part (2).

		It then only remains to show that $E^\circ$ and $M^\circ$ are orthogonal to each other, which amounts to saying that for every commutative square
		\begin{equation*}
			\begin{tikzcd}
				X\arrow[d, ->>]\arrow[r] & Y\arrow[d, rightarrowtail]\\
				X'\arrow[r] & Y'
			\end{tikzcd}
		\end{equation*}
		in $\Cc^\circ$ the essentially unique lift $X'\to Y$ in $\Cc$ already belongs to $\Cc^\circ$. This is again immediate from part (2) of the definition.
	\end{proof}
	\end{lemma}

	\begin{definition}[Good subcategory]
		Let $\Cc$ be equipped with a factorization system $(E,M)$ and a distinguished object $\pt$. We say that a subcategory $\Cc^{\circ}$ is \textit{good} if it is compatible with $(E,M)$ and contains the coSegal maps $\pt_B \rightarrowtail X$ for all $X\in \Cc^{\circ}(B)$.
	\end{definition}

	Note that for a good subcategory $\Cc^{\circ}$ we may now again talk about coSegal functors: Lemma~\ref{lemma:Restricted_Factorization_System} shows that $(E^\circ,M^\circ)$ is a parametrized factorization system on $\Cc^\circ$, and the fact that $\Cc^{\circ}$ contains all coSegal maps ensures that $\Cc^{\circ}$ inherits a distinguished object $\pt^{\circ} = \pt$. The goal for the rest of this section is to show that every coSegal functor on $\Cc^{\circ}$ uniquely extends to $\Cc$ and that all coSegal functors on $\Cc$ are of this form. We start with the following preliminary result:

	\begin{lemma}
		\label{lem:Restriction_Along_Good_Inclusion_Conservative}
		Restriction along a good inclusion $\Cc^{\circ} \hookrightarrow \Cc$ takes coSegal functors to coSegal functors, and the resulting functor $\Fun_{\coSeg}(\Cc,\Dd) \to \Fun_{\coSeg}(\Cc^{\circ},\Dd)$ is conservative.
	\end{lemma}
	\begin{proof}
		For the first claim, we must show that for a coSegal functor $F\colon \Cc \to \Dd$ the restriction $F\vert_{M^{\circ}}$ is pointwise left Kan extended from $\{\pt\}$. By definition, this requires the comparison of $F_AX$ with a certain colimit indexed by a category of morphisms $\pt_B \rightarrowtail X$ in $M^{\circ}$ for $X \in C^{\circ}(B)$. But by assumption this indexing category agrees with the category of \textit{all} morphisms $\pt_B \rightarrowtail X$ in $M$, hence the claim is immediate from the coSegal property of $F$.

		For the second claim, it suffices to observe that further composition with evaluation at $\pt \in \Cc$ is conservative: it equals the composition
		\[
		\Fun_{\coSeg}(\Cc,\Dd) \to \Fun_{\coSeg}(M,\Dd) \iso \Fun(\{\pt\},\Dd) \simeq \Dd,
		\]
		where the first functor is conservative since $M \subseteq \Cc$ is wide, and the second functor is an equivalence by the definition of being coSegal.
	\end{proof}

	Next, we show that the coSegal condition behaves well with respect to left Kan extension along the
	inclusion $\Cc^{\circ} \hookrightarrow \Cc$. If $\Cc^\circ\subseteq\Cc$ is a subcategory compatible with the factorization system, e.g.\ if it is good, we use the notations
	\[
	\Cc^{\circ}_{/X} := \Cc^{\circ} \times_{\Cc} \Cc_{/X} \qquad \text{ and } \qquad M^{\circ}_{/X} := M^{\circ} \times_{M} M_{/X}
	\]
	for an object $X \in \Cc$.

	\begin{lemma}
		For every global section $X \in \Gamma\Cc$, the inclusion $M_{/X} \hookrightarrow \Cc_{/X}$ admits a parametrized left adjoint, and the resulting adjunction $\Cc_{/X}\rightleftarrows M_{/X}$ restricts to an adjunction $\Cc_{/X}^\circ\rightleftarrows M_{/X}^\circ$.
		\begin{proof}
		Let $A\in\Bb$. Then the inclusion $M_{/X}(A)\hookrightarrow\Cc_{/X}(A)$ can be identified with the inclusion $M(A)_{/X_A}\to \Cc(A)_{/X_A}$. By \cite{CLL_Clefts}*{proof of Proposition~3.33} the latter is fully faithful and it admits a left adjoint $\lambda_A$ such that the unit consists of maps in $\Cc(A)_{/X_A}\times_{\Cc(A)}E(A)$. To see that the $\lambda_A$'s assemble into a parametrized functor, it will be enough to check the Beck--Chevalley condition. By full faithfulness of the inclusions this just amounts to saying that for every $f\colon A\to A'$ there is \emph{some} dashed arrow filling
			\begin{equation*}
				\begin{tikzcd}
					\Cc(A')_{/X_{A'}}\arrow[r, "f^*"]\arrow[d,"\lambda_{A'}"'] & \Cc(A)_{/X_{A}}\arrow[d, "\lambda_A"]\\
					M(A')_{/X_{A'}}\arrow[r, dashed] & M(A)_{/X_A}\rlap.
				\end{tikzcd}
			\end{equation*}
			This follows at once from the fact that the vertical arrows are localizations at $\Cc(A')_{/X_{A'}}\times_{\Cc(A')} E(A')$ and $\Cc(A)_{/X_{A}}\times_{\Cc(A)} E(A)$, respectively, by \emph{loc.~cit.}

			Next we show the second statement, which entails showing that the left adjoint $\lambda$ as well as the unit and counit restrict accordingly. The fact that $\lambda_A$ sends objects of $\Cc^\circ_{/X}$ to objects of $M^\circ_{/X}$ and that the unit lies pointwise in $\Cc^\circ_{/X}$ is immediate from part (1) of \Cref{def:Compatible_With_Factorization_System}, while part (2) guarantees that it maps morphisms of $\Cc^\circ_{/X}$ to morphisms in $M^\circ_{/X}$. Finally, the counit is an equivalence, hence lies pointwise in $M^\circ_{/X}$ as claimed.
		\end{proof}
	\end{lemma}

	\begin{corollary}
	The inclusion $M_{/X}^{\circ}\hookrightarrow \Cc^\circ_{/X}$ is final in the sense of \cite{martiniwolf2021limits}*{Proposition 4.6.1}, i.e.~$\Cc^\circ_{/X}$-shaped colimits can be computed after restricting to $M_{/X}^\circ$.
	\end{corollary}

	\begin{proof}
	We will prove more generally that any parametrized right adjoint $R\colon \Cc\rightarrow \Dd$ is final. By Quillen's Theorem A for parametrized categories, see \cite{martini2021yoneda}*{Corollary 4.4.8}, it suffices to show that the comma category $\pi_A^*\Cc_{X/}$ is weakly contractible for all $X\in \Dd(A)$. However by \cite{martiniwolf2021limits}*{Corollary 3.3.5} this category even admits an initial object, and so is clearly weakly contractible.
	\end{proof}

	\begin{proposition}\label{prop:restriction-final}
	Let $\iota\colon \Cc^{\circ}\hookrightarrow \Cc$ be a good inclusion such that $\Cc$ is small, and suppose that $\Dd$ admits all colimits. Then the Beck--Chevalley transformation filling the square
	\[\begin{tikzcd}
		{\Fun(\Cc^\circ,\Dd)} & {\Fun(\Cc,\Dd)} \\
		{\Fun(M^\circ,\Dd)} & {\Fun(M,\Dd)}
		\arrow["{-\vert_{M^\circ}}"', from=1-1, to=2-1]
		\arrow["{\iota_!}", from=1-1, to=1-2]
		\arrow["{\iota_!}", from=2-1, to=2-2]
		\arrow["{-\vert_M}", from=1-2, to=2-2]
	\end{tikzcd}\] is an equivalence.
	\end{proposition}

	Note that in the above situation the two left Kan extension functors indeed exist by \cite{martiniwolf2021limits}*{Corollary~6.3.7}.

	\begin{proof}
	Consider a functor $F\colon \Cc^\circ\rightarrow \Dd$ and let $X\in M(A)$. Then a quick computation shows that the composite
	\[
		\colim_{\pi_A^*M^\circ_{/X}}F\simeq\iota_!(F\vert_{M^{\circ}})(X) \to(\iota_!F)\vert_M(X)\simeq\colim_{\pi^*_A\Cc^{\circ}_{/X}}F
	\]
	is induced on colimits by the map $\pi^*_AM^\circ_{/X}\rightarrow \pi^*_A \Cc^\circ_{/X}$ of $\Bb_{/A}$-categories. However we note that $\pi^*_A \Cc^\circ \hookrightarrow \pi^*_A \Cc$ is again a good inclusion, and so this is an equivalence by the previous corollary.
	\end{proof}

	\begin{definition}
		Let $\Cc$ be a $\Bb$-category equipped with a factorization system and a distinguished object. We say a $\Bb$-category $\Dd$ \textit{admits $\Cc$-coSegal colimits} if for every object $Y \in \Gamma\Dd$ the pointwise left Kan extension of $Y\colon \{\pt\} \to \Dd$ along the inclusion $\{\pt\} \hookrightarrow M$ exists.
	\end{definition}

	\begin{remark}
		Suppose $\Dd$ admits $\Cc$-coSegal colimits and suppose $\Cc^\circ \subseteq \Cc$ is a good inclusion. Then $\Dd$ also admits $\Cc^\circ$-coSegal colimits.
	\end{remark}

	\begin{lemma}
		\label{lem:LKE_Along_Good_Inclusion}
		Let $\Cc^{\circ} \subseteq \Cc$ be a good subcategory and suppose $\Dd$ is a $\Bb$-category with $\Cc$-coSegal colimits. For every coSegal $\Bb$-functor $F^{\circ}\colon \Cc^{\circ} \to \Dd$, there exists a left Kan extension $F\colon \Cc \to \Dd$ along the inclusion $\Cc^{\circ} \hookrightarrow \Cc$. Furthermore, the $\Bb$-functor $F$ is again coSegal, and the canonical map $F^{\circ}(\pt) \to F(\pt)$ is an equivalence.
	\end{lemma}

	\begin{proof}
		By changing universe we may assume that $\Cc$ is small. Let us first assume that $\Dd$ is cocomplete, so that we may apply \Cref{prop:restriction-final} to deduce the restriction $F\vert_M \colon M \to \Dd$ is the left Kan extension along $M^{\circ} \hookrightarrow M$ of the restriction $F^{\circ}\vert_{M^{\circ}}$. Because $F^{\circ}$ is coSegal, $F^{\circ}\vert_{M^{\circ}}$ is itself left Kan extended from $\{\pt\}$, and thus it follows from transitivity of left Kan extension that also $F\vert_M$ is left Kan extended from $\{\pt\}$, i.e.\ that $F$ is coSegal. The final claim follows from the assumption that the inclusions $\{\pt\} \hookrightarrow M^{\circ}$ and $\{\pt\} \hookrightarrow M$ are fully faithful.

		For arbitrary $\Dd$, pick an embedding $\Dd \hookrightarrow \Dd'$ into a cocomplete $\Bb$-category $\Dd'$ which preserves all colimits that exists in $\Dd$ (e.g.\ the coYoneda embedding). By the previous paragraph, the left Kan extension $F$ of $F^\circ$ exists as a functor into $\Dd'$ and has the required properties. Furthermore as we have seen, once $F$ is restricted to $M$ it is pointwise left Kan extended from $\{\pt\}$ and hence lands in $\Dd$ by the assumption that $\Dd$ has coSegal colimits. However $M\subseteq \Cc$ is a wide subcategory, and so $F$ itself lands in $\Dd$.
	\end{proof}

	Putting everything together, we obtain the main result of this subsection:
	\begin{proposition}\label{prop:good_equivalence}
		Consider a good inclusion $\iota\colon \Cc^{\circ} \hookrightarrow \Cc$, and let $\Dd$ be a $\Bb$-category with $\Cc$-coSegal colimits. Then restriction along $\Cc^\circ\hookrightarrow \Cc$ induces an equivalence of categories
		\[
		\Fun_{\coSeg}(\Cc,\Dd) \iso \Fun_{\coSeg}(\Cc^{\circ},\Dd).
		\]
	\end{proposition}

	\begin{proof}
		By \Cref{lem:Restriction_Along_Good_Inclusion_Conservative} the restriction functor $\iota^*\colon \Fun(\Cc,\Dd) \to \Fun(\Cc^{\circ},\Dd)$ restricts to a conservative functor $\Fun_{\coSeg}(\Cc,\Dd) \to \Fun_{\coSeg}(\Cc^{\circ},\Dd)$. By \Cref{lem:LKE_Along_Good_Inclusion}, this restriction admits a left adjoint
		\[
			\iota_! \colon \Fun_{\coSeg}(\Cc^{\circ},\Dd) \to \Fun_{\coSeg}(\Cc,\Dd),
		\]
		given by left Kan extension along $\iota$. Since $\iota^*$ is conservative when restricted to the coSegal functors, it remains to show that the unit $\id \to \iota^*\iota_!$ of this adjunction is an equivalence. This follows again from \Cref{lem:LKE_Along_Good_Inclusion} and the fact that evaluation at the point is conservative.
	\end{proof}

	\subsection{Good inclusions of span categories}
	\label{sec:Good_Inclusions_Of_Span_Categories}
	Our main interest in \Cref{prop:good_equivalence} is in the case where $\Cc$ is a parametrized span category. In this subsection, we show that various inclusions of parametrized span categories are good inclusions.

	\begin{convention}
		\label{conv:QL_QR_Q}
		Throughout this subsection, we fix local wide subcategories $\Qq_L,\Qq_R \subseteq \Qq \subseteq \Bb$ closed under diagonals. We do \emph{not} assume any of these to be locally inductible yet.

		We recall from Construction~\ref{constr:parametrized_spans} the definition of the parametrized span category $\ul{\Span}(\Qq,\Qq_L,\Qq_R)$ and its wide subcategories $(\ulbbU{\Qq}^L)\catop$ and $\ulbbU{\Qq}^R$. By applying \cite{HHLNa}*{Proposition~4.9} levelwise, it follows that these subcategories define a factorization system on $\ul{\Span}(\Qq,\Qq_L,\Qq_R)$: the left class consists of the left-pointing maps and the right class consists of the right-pointing maps:
		\begin{align*}
			A \twoheadrightarrow B &= (A \leftarrow B = B), \\
			A \rightarrowtail B &= (A = A \to B).
		\end{align*}
		We further take the distinguished object of the span category to always be the identity map $\pt := (1 \to 1) \in \core(\Qq_{/1}) \subseteq \ul{\Span}(\Qq,\Qq_L,\Qq_R)(1)$.
	\end{convention}

	As a first step, we show that the $\Cc$-coSegal colimits for $\Cc = \ul{\Span}(\Qq,\Qq_L,\Qq)$ can be expressed in terms of $\Qq$-colimits.

	\begin{proposition}
	\label{prop:Functor_CoSegal_Iff_QColimit_Preserving}
	\label{prop:coSegal_Colimits_Are_QColimits}
	Let $\Dd$ be $\Qq$-cocomplete. Then $\Dd$ has $\ul\Span(\Qq,\Qq_L,\Qq)$-coSegal colimits and a functor $F\colon\ul\Span(\Qq,\Qq_L,\Qq)\rightarrow \Dd$ is coSegal if and only if it preserves $\Qq$-colimits.
	\end{proposition}

	\begin{proof}
	By \Cref{prop:uni_prop_U_Q}, the left Kan extension of any $\ul1\to\Dd$ along $\{\pt\}\hookrightarrow\ulbbU{\Qq}$ exists, and a functor $\ulbbU{\Qq}\to\Dd$ arises this way if and only if it is $\Qq$-cocontinuous. The claim follows as by \Cref{cor:test_colimit_pres_span} a functor $F\colon \ul\Span(\Qq,\Qq_L,\Qq)\rightarrow \Dd$ preserves $\Qq$-colimits if and only if the composite $G\colon \ulbbU{\Qq}\rightarrow \ul\Span(\Qq,\Qq_L,\Qq)\rightarrow \Dd$ preserves $\Qq$-colimits.
	\end{proof}

	We now provide two examples of good inclusions of parametrized span categories.

	\begin{lemma}
	\label{lem:Extending_Covariant_Direction}
	The inclusion
	\[
	\iota\colon \ul{\Span}(\Qq_L) \hookrightarrow \ul{\Span}(\Qq, \Qq_L,\Qq)
	\]
	is good. In particular, for a $\Bb$-category $\Dd$ admitting $\Qq$-colimits, restriction along $\iota$ induces an equivalence of $\infty$-categories
	\[
	\iota^*\colon \Fun_{\coSeg}(\ul{\Span}(\Qq, \Qq_L,\Qq),\Dd) \iso \Fun_{\coSeg}(\ul{\Span}(\Qq_L),\Dd).
	\]
	\end{lemma}
	\begin{proof}
	Note that $\ul{\Span}(\Qq_L)$ is a \textit{full} subcategory of $\ul\Span(\Qq,\Qq_L,\Qq)$ by left-cancel\-ability, and obviously contains $\pt$. It thus remains to show that $\ul{\Span}(\Qq_L)$ is compatible with the factorization system. Condition (1) follows from left-cancelability of $\Qq_L$, while Condition (2) is automatic. The final claim follows immediately from \Cref{prop:good_equivalence} and \Cref{prop:coSegal_Colimits_Are_QColimits}.
	\end{proof}

	\begin{notation}
	Given any collection of maps $\Qq \subseteq \Bb$, we write $\Delta(\Qq)$ for the collection of all maps in $\Bb$ of the form $\Delta_q\colon A\to A\times_BA$ for morphisms $q\colon A \to B$ in $\Qq$.
	\end{notation}

	\begin{lemma}\label{lem:Segal_covariant}
	Assume that $\Delta(\Qq) \subseteq \Qq_R$. Then the inclusion
	\[
	\iota \colon \ul{\Span}(\Qq, \Qq_L,\Qq_R) \hookrightarrow \ul{\Span}(\Qq, \Qq_L,\Qq)
	\]
	is good. In particular, for any $\Bb$-category $\Dd$ admitting $\Qq$-colimits, restriction along $\iota$ induces an equivalence of $\infty$-categories
	\[
	\Fun_{\coSeg}(\ul{\Span}(\Qq, \Qq_L,\Qq),\Dd) \iso \Fun_{\coSeg}(\ul{\Span}(\Qq,\Qq_L,\Qq_R),\Dd).
	\]
	\end{lemma}

	\begin{proof}
	Note that $\ul{\Span}(\Qq, \Qq_L,\Qq_R)$ contains all left-pointing maps in $\ul{\Span}(\Qq, \Qq_L,\Qq)$, and thus condition (1) of \Cref{def:Compatible_With_Factorization_System} is immediate. Condition (2) follows from a simple calculation of the composition in the relevant span categories. This shows that $\ul{\Span}(\Qq, \Qq_L,\Qq_R)$ is compatible with the factorization system on $\ul{\Span}(\Qq, \Qq_L,\Qq)$. To see it is even a good subcategory, it remains to show that it contains the coSegal maps $s\colon \pt_B \rightarrowtail A$ for all $A \in \Qq_{/B}$, which boils down to showing that for every span
	\[\begin{tikzcd}
		B & B & A \\
		& B
		\arrow["s", from=1-2, to=1-3]
		\arrow["f", from=1-3, to=2-2]
		\arrow[Rightarrow, no head, from=1-2, to=2-2]
		\arrow[Rightarrow, no head, from=1-1, to=1-2]
		\arrow[Rightarrow, no head, from=1-1, to=2-2]
	\end{tikzcd}
	\]
	in $\Qq_{/B}$ the morphism $s$ is contained in $\Qq_R$. To this end, consider the following diagram in $\Qq$
	\[\begin{tikzcd}
		B & A & B \\
		A & {A\times_BA} & A \\
		& A & B
		\arrow[from=2-2, to=2-3]
		\arrow["f", from=2-3, to=3-3]
		\arrow["s", from=1-3, to=2-3]
		\arrow[from=1-2, to=2-2]
		\arrow["\Delta", from=2-1, to=2-2]
		\arrow["f"', from=3-2, to=3-3]
		\arrow[from=2-2, to=3-2]
		\arrow["s", from=1-1, to=1-2]
		\arrow["f", from=1-2, to=1-3]
		\arrow[pullback, draw=none, xshift=-3.5pt, from=2-2, to=3-3]
		\arrow[pullback, from=1-2, to=2-3]
		\arrow["s"', from=1-1, to=2-1]
		\arrow[pullback, draw=none, from=1-1, to=2-2]
	\end{tikzcd}\]
	in which all squares are pullbacks. Since the diagonal map $\Delta\colon A \to A \times_B A$ lies in $\Qq_R$ by assumption, also its base change $s\colon B \rightarrow A$ lies in $\Qq_R$ as desired. We conclude that $\ul{\Span}(\Qq, \Qq_L,\Qq_R)$ is a good subcategory. The final claim follows immediately from \Cref{prop:good_equivalence} and \Cref{prop:coSegal_Colimits_Are_QColimits}.
	\end{proof}

	Using this, we can give a more concrete description of coSegal functors out of $\ul\Span(\Qq,\Qq_L,\Qq_R)$ when $\Delta(\Qq)\subseteq\Qq_R$, generalizing Proposition~\ref{prop:Functor_CoSegal_Iff_QColimit_Preserving}.

	\begin{construction}[CoSegal map]
		\label{cons:CoSegal_Map}
			Let $\Dd$ be a $\Bb$-category admitting $\Qq$-colimits, and let $F\colon \ulbbU{\Qq}^R \to \Dd$ be a $\Bb$-functor. Assume moreover that $\Delta(\Qq)\subseteq\Qq_R$. We will construct for every morphism $q\colon A \to B$ in $\Qq$ a \textit{coSegal map}
			\[
			\mathrm{coSegal}\colon q_!F(\id_A) \to F(q).
			\]
			For this, let $F'\colon \ulbbU{\Qq} \to \Dd$ be the left Kan extension of $F(\id_1) \colon \ul{1} \to \Dd$. By \Cref{lem:Segal_covariant} (for $\Qq_L=\iota\Qq$), the restriction of $F'$ to $\ulbbU{\Qq}^R$ is still left Kan extended along $\{\id_1\} \hookrightarrow \ulbbU{\Qq}^R$, and hence there exists a unique natural transformation
			\[
			\mathrm{coSegal}\colon F'\vert_{\ulbbU{\Qq}^R} \to F
			\]
			which results in the identity of $F(\id_1)$ when evaluated on $\id_1$. Since $F'$ preserves $\Qq$-colimits by \Cref{prop:uni_prop_U_Q}, we have $F'(q) \simeq q_!F'(\id_A) = q_!F(\id_A)$ for every morphism $q\colon A \to B$ in $\Qq$, resulting in the desired map
			\[
			q_!F(\id_A) \simeq F'(q) \xrightarrow{\mathrm{coSegal}} F(q).
			\]
		\end{construction}

		\begin{construction}
			We continue to assume that $\Delta(\Qq)\subseteq\Qq_R$. Consider any $\Bb$-functor $F\colon \ul{\Span}(\Qq,\Qq_L,\Qq_R) \to \Dd$. For every morphism $q\colon A \to B$ in $\Qq$, we may apply \Cref{cons:CoSegal_Map} to the restriction of $F$ to $\ulbbU{\Qq}^R$ to obtain a coSegal map
			\[
			\mathrm{coSegal}\colon q_!F(\id_A) \to F(q).
			\]
			Unwinding definitions, we see that $F$ is a coSegal functor if and only if the coSegal map is an equivalence for every $q \in \Qq$.
		\end{construction}

	\subsection{CoSegal functors as Segal functors}
	\label{sec:Cosegal_Vs_Segal}
	While \Cref{lem:Segal_covariant} allows us to extend the \textit{covariant} functoriality of a functor out of a parametrized span category, we will also need an analogous result which lets us extend the \textit{contravariant} functoriality, at least under suitable semiadditivity assumptions on $\Dd$. We will accomplish this by showing that in this case a functor out of a parametrized span category is coSegal if and only if it is \textit{Segal}, defined dually:

	\begin{definition}[Segal functor]\label{defi:segal-functor}
	A functor $F\colon \ul{\Span}(\Qq,\Qq_L,\Qq_R) \to \Dd$ is called \textit{Segal} if the composite
	\[
	\ul{\Span}(\Qq,\Qq_R,\Qq_L) \simeq \ul{\Span}(\Qq,\Qq_L,\Qq_R)\catop \xrightarrow{F\catop} \Dd\catop
	\]
	is coSegal. We denote by
	\[
	\Fun_{\Seg}(\ul{\Span}(\Qq,\Qq_L,\Qq_R),\Dd) \subseteq \Fun_{\Bb}(\ul{\Span}(\Qq,\Qq_L,\Qq_R),\Dd)
	\]
	the full subcategory spanned by the Segal functors.
	\end{definition}

	The following is the main result of this subsection:

	\begin{proposition}\label{prop:Segal_Equals_CoSegal}
	Assume $\Qq$ is locally inductible such that every map in $\Qq$ is truncated, and let $\Dd$ be a $\Qq$-semiadditive $\Bb$-category. If $\Delta(\Qq) \subseteq \Qq_L \cap \Qq_R$, then  a $\Bb$-functor $F\colon\ul\Span(\Qq,\Qq_L,\Qq_R)\to\Dd$ is Segal if and only if it is coSegal.
	\end{proposition}

	Before discussing the proposition, let us record the main corollary:

	\begin{corollary}
		\label{cor:Extending_Contravariant_Direction}
		Assume $\Qq$ is locally inductible and truncated, let $\Dd$ be a $\Qq$-semiadditive $\Bb$-category, and assume that $\Delta(\Qq) \subseteq \Qq_L$. Then restricting along the inclusion $\ul{\Span}(\Qq, \Qq_L,\Qq)\hookrightarrow \ul{\Span}(\Qq)$ defines an equivalence of categories
		\[
		\Fun_{\coSeg}(\ul{\Span}(\Qq),\Dd) \iso \Fun_{\coSeg}(\ul{\Span}(\Qq,\Qq_L,\Qq),\Dd).
		\]
	\end{corollary}
	\begin{proof}
		By \Cref{prop:Segal_Equals_CoSegal}, we may equivalently show the claim for Segal functors. Unwinding definitions, this reduces to the claim that restriction along the inclusion $\ul{\Span}(\Qq, \Qq,\Qq_L)\hookrightarrow \ul{\Span}(\Qq)$ defines an equivalence
		\[
		\Fun_{\coSeg}(\ul{\Span}(\Qq),\Dd^\op) \iso \Fun_{\coSeg}(\ul{\Span}(\Qq,\Qq,\Qq_L),\Dd^\op).
		\]
		This is a special case of \Cref{lem:Segal_covariant}.
	\end{proof}

	The proposition should not be too surprising: as explained above, the coSegal condition amounts to demanding that for every $q\in\Qq$ a certain map $q_!F(\pt)\to F(q)$ is an equivalence, while the Segal condition amounts to saying that the dually defined map $F(q)\to q_*F(\pt)$ is an equivalence. The equivalence $\Nm_q\colon q_!F(\pt)\simeq q_*F(\pt)$ coming from $\Qq$-semiadditivity of $\Dd$ thus strongly suggests that these two conditions are equivalent. While this is true, relating the above two maps (defined basically in terms of maps in $\ul\Span(\Qq)$) to the norm map $\Nm_q$ of $\Dd$ turns out to be somewhat subtle and will take up the remainder of this subsection.

	We begin by describing the coSegal map explicitly in the above situation:

	\begin{lemma}
	\label{lem:Description_CoSegal_Map}
		Let $q\colon A \to B$ be a morphism in $\Qq$ and assume $\Delta(\Qq)\subseteq\Qq_R$. Then the coSegal map $q_!F(\id_A) \to F(q)$ is adjoint to
		\[
		F(\id_A) \xrightarrow{F(\Delta_q)} F(q^*(q)) = q^*F(q).
		\]
		Here the map $\Delta_q\colon \id_A \to q^*(q) = (\pr_1\colon A \times_B A \to A)$ is the morphism in the slice $\Qq_{/A}^R$ corresponding to the diagonal map $\Delta_q\colon A \to A \times_B A$.
	\end{lemma}
	\begin{proof}
		Consider the following diagram:
		\[
		\begin{tikzcd}
			F'(q) \ar[bend right=50, equal]{ddd} \dar{F'(q_!\Delta_q)}&\arrow[l,"\sim"'] q_!F'(\id_A) \rar["\mathrm{coSegal}", "\sim"'] \dar{q_!F'(\Delta_q)} &[1.5em] q_!F(\id_A) \dar{q_!F(\Delta_q)} \\
			F'(q_!q^*(q)) \dar[equal] &\arrow[l,"\sim"'] q_!F'(q^*(q)) \rar{\mathrm{coSegal}} \dar[equal] & q_!F(q^*(q)) \dar[equal] \\
			F'(q_!q^*(q)) \dar{F'(\pr_1)} &\arrow[l,"\sim"'] q_!q^*F'(q) \rar{\mathrm{coSegal}} \dar{\epsilon_q} & q_!q^*F(q) \dar{\epsilon_q} \\
			F'(q) \rar[equal] & F'(q) \rar{\mathrm{coSegal}} & F(q).
		\end{tikzcd}
		\]
		The right half of the diagram commutes by naturality of the coSegal map $F' \to F$. In the left half of the diagram we use that $F'\colon\ulbbU{\Qq}\to\Dd$ preserves $\Qq$-colimits; the bottom left square commutes because the counit $\epsilon_q\colon q_!q^*(q) \to q$ in $\ulbbU{\Qq}$ is given by the projection map $\pr_1\colon A \times_B A \to A$. The left vertical composite is the identity as $\Delta_q\colon A\to A \times_B A$ is a section of $\pr_1$. As the top of the diagram is the canonical identification $F'(q) = q_!(\id_A)$, the claim follows.
	\end{proof}

	\begin{corollary}\label{cor:BC-in-terms-of-span}
		Assume again that $\Delta(\Qq)\subseteq\Qq_R$. Let $\Dd$ be a $\Bb$-category admitting $\Qq$-colimits and let $F\colon\ulbbU{\Qq}^R\to\Dd$ be a $\Bb$-functor. Consider a pullback square
		\begin{equation*}
			\begin{tikzcd}
				A\arrow[r, "q'"]\arrow[d, "p'"'] \drar[pullback] & B\arrow[d, "p"]\\
				C\arrow[r, "q"'] & D
			\end{tikzcd}
		\end{equation*}
		in $\Qq$, expressing the relation $q'=p^*(q)$ in $\ulbbU{\Qq}$. Then the diagram
		\begin{equation*}
			\begin{tikzcd}
				q'_!F(\id_A) \rar[equal] \dar[swap]{\textup{coSegal}} & q_!'(p')^*F(\id_C)
				\arrow[r, "\textup{BC}_!"] & p^*q_!F(\id_C)\arrow[d, "p^*(\textup{coSegal})"] \\
				F(q') \ar[equal]{rr} && p^*F(q)
			\end{tikzcd}
		\end{equation*}
		commutes up to homotopy.
	\end{corollary}
	\begin{proof}
		By \Cref{lem:Description_CoSegal_Map}, this is equivalent to the commutativity of the following diagram:
		\[
		\begin{tikzcd}
			F(\id_A) \dar[swap]{F(\Delta_{q'})} \ar[equal]{rr} && {p'}^*F(\id_C) \dar{p^{\prime*}F(\Delta_q)} \\
			q^{\prime*}F(q') \rar[equal] & q^{\prime*}p^*F(q) \rar[equal] & p^{\prime*}q^*F(q).
		\end{tikzcd}
		\]
		But this is immediate from the fact that the image of the map $\Delta_{q} \colon C\to C\times_DC$ under the pullback functor ${p'}^*\colon \Qq_{/C} \to \Qq_{/A}$ is the diagonal $\Delta_{q'}\colon A \to A \times_{B} A$.
	\end{proof}

	The description of the coSegal map from \Cref{lem:Description_CoSegal_Map} naturally leads us to consider the following more general coSegal maps:

	\begin{construction}
	Let $F\colon \ulbbU{\Qq}^R \to \Dd$ be a $\Bb$-functor and assume that $\Delta(\Qq) \subseteq \Qq_R$. For morphisms $p\colon A \to B$ and $q \colon B \to C$ in $\Qq$, we define a \textit{coSegal map}
	\[
	\textup{coSegal}\colon q_!F(p) \to F(qp)
	\]
	as the map adjoint to $F(p) \xrightarrow{F(1,p)} F(q^*(qp)) = q^*F(qp)$. Here $(1,p)\colon p \to q^*(qp)$ is the morphism in the slice $\Qq_{/B}^R$ corresponding to the map $(1,p) \colon A \to A \times_C B$.
	\end{construction}

	\begin{remark}\label{rk:gen-coSegal-natural}
		On $\ul\Span(\Qq)$, the functor $q_\circ\colon p\mapsto qp$ is simply the left adjoint $q_!$ of $q^*$, and the maps $(1,p)$ form the unit, see Corollary~\ref{cor:Span_Complete_And_Cocomplete}. In particular, we see that the above generalized coSegal map `is' natural in $p$.
	\end{remark}

	\begin{proposition}
		\label{prop:Compatibility_CoSegal_And_Norm}
		Assume that $\Qq$ is locally inductible, let $\Qq_L,\Qq_R \subseteq \Qq$ be subclasses with $\Delta(\Qq)\subseteq\Qq_L\cap \Qq_R$, and let $\Dd$ be a $\Qq$-semiadditive $\Bb$-category. For any truncated map $q$ in $\Qq$ and any $\Bb$-functor $F\colon \ul{\Span}(\Qq,\Qq_L,\Qq_R) \to \Dd$ that is coSegal the composite
		\[
		\begin{tikzcd}
			q_!F(\id_A) \rar{\mathrm{coSegal}\;} &[1em] F(q) \rar{\mathrm{Segal\;}} & q_*F(\id_A)
		\end{tikzcd}
		\]
		is homotopic to the norm map $\Nm_q\colon q_! F(\id_A) \to q_* F(\id_A)$ in $\Dd$.
	\end{proposition}
	\begin{proof}
	As $\Qq_L$ contains the diagonal $\Delta_q\colon A \to A \times_B A$ of $q$ by assumption, the span
		\begin{equation*}
			A\times_BA\xleftarrow{\;\Delta_q\;} A\xrightarrow{\;=\;}A
		\end{equation*}
		defines a map $q^*q \to \id_A$ in $\ulSpan(\Qq,\Qq_L,\Qq_R)(A)=\Span(\Qq_{/A},\Qq_{/A}^L,\Qq_{/A}^R)$; we will denote this span by $\nabla_q$. The dual of \Cref{lem:Description_CoSegal_Map} shows that the Segal map $F(q) \to q_*F(\id_A)$ is adjoint to the composite
		\[
			q^*F(q) = F(q^*q) \xrightarrow{F(\nabla_q)} F(\id_A).
		\]
		The proposition is thus equivalent to the claim that the composite
		\begin{equation*}
			q^*q_!F(\id_A)\xrightarrow{\textup{coSegal}} q^*F(q) = F(q^*q)\xrightarrow{F(\nabla_q)} F(\id_A)
		\end{equation*}
		is homotopic to the adjoint norm map $\Nmadj_q$ of $\Dd$.

		To prove this, we consider the following diagram, where the top row spells out the definition of $\Nmadj_q\colon q^*q_!F(\id_A)\to F(\id_A)$:
		\begin{equation*}\hskip-54.84pt\hfuzz=55pt
			\begin{tikzcd}[cramped]
				q^*q_!F(\id_A)\arrow[d, "\textup{coSegal}"']\arrow[r, "\textup{BC}_!^{-1}"] &[.75em] \arrow[d, equal] \pr_{1!}\pr_2^* F(\id_A)\arrow[r, "\eta"] &[-.33em] \arrow[d,equal]\pr_{1!}\Delta_*\Delta^*\pr_2^* F(\id_A)\arrow[r, "\Nm_\Delta^{-1}"] & \pr_{1!}\Delta_!\Delta^*\pr_2^* F(\id_A)\arrow[d,equal]\arrow[r,"\sim"]  & F(\id_A)\arrow[ddd,"\sim"]\\
				F(q^*q)\arrow[r,"\textup{coSegal}^{-1}\,"'] & \pr_{1!}F(\id_{A\times_BA})\arrow[r, "\eta"']\arrow[dr, bend right=21pt, "F(\eta)"'] & \pr_{1!}\Delta_*\Delta^*F(\id_{A\times_BA})\arrow[from=d,"\textup{BC}_*=\text{Segal}"{description}]\arrow[r,phantom,"(*)"{description}] & \pr_{1!}\Delta_!\Delta^*F(\id_{A\times_BA})\arrow[d, "\textup{BC}_! = \textup{coSegal}"{description}]\\[1ex]
				&& \pr_{1!}F(\Delta_*\id_A)\arrow[r, "F(\Nm_\Delta^{-1})"']\arrow[d,equal]& \pr_{1!}F(\Delta_!\id_A)\arrow[d,equal]\\
				&&\pr_{1!}F(\Delta)\arrow[r,equal]&\pr_{1!}F(\Delta)\arrow[r, "\textup{coSegal}"'] & F(\pr_{1}\Delta)
			\end{tikzcd}
		\end{equation*}
		Here the square on the left commutes by Corollary~\ref{cor:BC-in-terms-of-span}. Moreover, as $\Delta$ is a truncated map in $\Qq_L\cap\Qq_R$ by assumption, Proposition~\ref{prop:adj-norm-vs-B-functor} shows that the rectangle $(*)$ commutes before inverting the norm equivalences; as the right hand vertical map is invertible by assumption, we conclude that also the rectangle with the inverted norm maps commutes. The rightmost rectangle commutes by direct inspection, and so does the triangle in the second column. To finish the proof it will therefore suffice to show that the bottom composite $F(q^*q)\to F(\id_A)$ is simply the map $F(\nabla)$. However by Remark~\ref{rk:gen-coSegal-natural}, the subcomposite $F(\pr_1)=F(q^*q)\to F(\pr_1\Delta)$ agrees with $F(\pr_{1{\circ}}\eta)$, so this is a straight-forward computation in $\ul\Span(\Qq,\Qq_L,\Qq_R)$, using that the unit map $\eta\colon\id_{A\times_BA}\to \Delta$ is given by the analogous span in $\Qq_{/A \times_B A}$:
		\begin{equation*}
			A\times_BA \xleftarrow{\;\Delta_q\;} A\xrightarrow{\;=\;} A.\qedhere
		\end{equation*}
	\end{proof}

	\begin{proof}[Proof of \Cref{prop:Segal_Equals_CoSegal}]
	By symmetry, it suffices to show that any coSegal functor $F\colon \ul{\Span}(\Qq,\Qq_L,\Qq_R) \to \Dd$ is also Segal. But this is immediate from \Cref{prop:Compatibility_CoSegal_And_Norm} since $\Nm_q$ is an equivalence.
	\end{proof}

	\subsection{The universal property}
	\label{sec:Universal_Property}
	Combining all results of this section, we will now prove the universal property of $\ul{\Span}(\Qq)$ for locally inductible $\Qq$ from \Cref{thm:universal_prop_par_spans}: for every $\Qq$-semiadditive $\Bb$-category $\Dd$ evaluation at the global section $\pt$ restricts to an equivalence of $\Bb$-categories
	\[
		\ev_{\pt}\colon \ul{\Fun}^{\Qoplus}(\ul{\Span}(\Qq),\Dd) \iso \Dd.
	\]
	The main part will be an induction proving the analogous statement for the truncations $\Qq_{\le n}$. In order to pass to the full span category, we will use:

	\begin{lemma}\label{lem:Colimit_Of_Subcategories}
		Let $\Vv$ be a $\Bb$-category. Assume we have an increasing chain $\Vv_0\subseteq\Vv_1\subseteq\cdots\subseteq\Vv$ of full $\Bb$-subcategories such that each $X\in\Vv(A)$ is \emph{locally in the $\Vv_n$'s} in the following sense: there exists a cover $(f_i\colon A_i\to A)_{i\in I}$ and for each $i\in I$ a natural number $n_i\in\mathbb N$ such that $f_i^*X\in\Vv_{n_i}(A_i)$.

		Then the inclusions exhibit $\Vv$ as the colimit in $\Cat(\Bb)$ of the $\Vv_n$'s.
		\begin{proof}
			Let us write $\Bb'\coloneqq\Fun^\textup{R}(\Bb^\op,\Spc)\simeq\Bb$. Identifying $\Cat$ with complete Segal spaces we then obtain a fully faithful functor
			\begin{equation*}
				\Cat(\Bb)=\Fun^\textup{R}(\Bb^\op,\Cat)\to\Fun^\textup{R}(\Bb^\op,\Fun(\Delta^\op,\Spc))\simeq\Fun(\Delta^\op,\Bb')
			\end{equation*}
			given in degree $[k]\in\Delta$ by $\Cc\mapsto\core\big(\Cc^{[k]}\big)$, also see \cite{martini2021yoneda}*{Proposition~3.5.1}. As fully faithful functors reflect colimits and since colimits in functor categories are pointwise, it will be enough to show that $\iota(\Vv^{[k]})$ is the colimit of the $\iota(\Vv^{[k]}_n)$'s. Clearly, each $\iota(\Vv^{[k]}_n)\to\iota(\Vv^{[k]})$ is fully faithful (i.e.~a monomorphism of $\Bb$-groupoids). Given now an object $X_\bullet=(X_0\to\cdots\to X_k)\in\Vv^{[k]}(A)$ we can find covers $(f_i^{(j)}\colon A_i^{(j)}\to A)_{i\in I_j}$ such that each $(f_i^{(j)})^*X_j$ is contained in some $\Vv^{[k]}_{n_{i,j}}$. Passing to a common refinement and setting $n_i=\max\{n_{i,0},\dots,n_{i,k}\}$ we see that $X_\bullet$ is locally in the $\Vv^{[k]}_n$'s. Replacing $\Vv$ by $\iota(\Vv^{[k]})$ we are therefore altogether reduced to proving the analogous statement in $\Bb'$.

			In $\Spc$, transfinite compositions of monomorphisms are monomorphisms; exhibiting $\Bb'\simeq\Bb$ as a left exact localization of a presheaf topos, we therefore see that the same holds true in $\Bb'$, and it only remains to show that the subgroupoid inclusion $\colim_n\Vv_n\to\Vv$ is essentially surjective. This follows immediately from the assumption that each $X\in\Vv(A)$ be locally in the $\Vv_n$'s.
		\end{proof}
	\end{lemma}

	\begin{proof}[Proof of \Cref{thm:universal_prop_par_spans}]
		The map $\ev_{\pt}$ is given in degree $B\in\Bb$ by
		\[\hskip-12.25pt\hfuzz=12.25pt
			\Fun^{\pi_B^{-1}\Qq\text-\times}(\pi_B^*\ul{\Span}(\Qq),\pi_B^*\Dd)\simeq
			\Fun^{\pi_B^{-1}\Qq\text-\times}(\ul{\Span}(\pi_B^{-1}\Qq),\pi_B^*\Dd) \to (\pi_B^*\Dd)(\id_B) =\Dd(B).
		\]
		Replacing $\Bb$ by $\Bb_{/B}$, it therefore suffices to prove the statement on underlying functor categories.

		By the previous lemma, the inclusions exhibit $\ul\Span(\Qq)$ as a colimit of the truncations $\ul\Span({\Qq_{\le n}})$, which by Lemma~\ref{lemma:cocont-enough-truncated} induces an equivalence
		\[
			{\Fun}^{\Qoplus}(\ul{\Span}(\Qq),\Dd)\iso\lim\nolimits_{n\ge-2}\Fun^{\Qq_{\le n}\text-\oplus}(\ul{\Span}(\Qq_{\leq n}),\Dd).
		\]
		Thus, it suffices to prove the theorem for $\Qq$ replaced by $\Qq_{\leq n}$ for all $n\ge-2$. As a functor $\Span(\Qq_{\leq n})\to \Dd$ is $\Qq_{\le n}$-semiadditive if and only if it is coSegal (Proposition~\ref{prop:Functor_CoSegal_Iff_QColimit_Preserving}), we are altogether reduced to prove that evaluation at the identity defines an equivalence
		\[
			\Fun_\coSeg(\ul\Span(\Qq_{\le n}),\Dd)\iso\Gamma\Dd.
		\]
		We proceed by induction on $n$. When $n = -2$, we get that $\ul{\Span}(\Qq_{\le n}) = \ul1$ is the terminal $\Bb$-category and every functor is coSegal, so the claim holds trivially. Assume that the result holds for $n-1$, and consider the inclusions
		\[
		\ul{\Span}(\Qq_{\leq n-1}) \hookrightarrow \ul{\Span}(\Qq_{\leq n}, \Qq_{\leq n-1},\Qq_{\leq n}) \hookrightarrow \ul{\Span}(\Qq_{\leq n}).
		\]
		As $\Delta(\Qq_{\leq n}) \subseteq \Qq_{\leq n-1}$, \Cref{lem:Extending_Covariant_Direction} and \Cref{cor:Extending_Contravariant_Direction} show that restriction along these inclusions induce equivalences
		\[
		\Fun_{\coSeg}(\ul{\Span}(\Qq_{\leq n}),\Dd) \iso \Fun_{\coSeg}(\ul{\Span}(\Qq_{\leq n}, \Qq_{\leq n-1},\Qq_{\leq n}), \Dd)
		\]
		and
		\[
		\Fun_{\coSeg}(\ul{\Span}(\Qq_{\leq n}, \Qq_{\leq n-1},\Qq_{\leq n}), \Dd) \iso \Fun_{\coSeg}(\ul{\Span}(\Qq_{\leq n-1}),\Dd).\qedhere
		\]
	\end{proof}

	Via Lemma~\ref{lem:Extending_Covariant_Direction} we also immediately obtain the following generalization:

	\begin{corollary}\label{cor:free-semiadd-more-colims}
		Let $\Qq_L\subseteq\Qq\subseteq\Bb$ be wide local subcategories closed under diagonals, and assume that $\Qq_L$ is locally inductible. Then $\ul\Span(\Qq,\Qq_L,\Qq)$ is $\Qq_L$-semiadditive and $\Qq$-cocomplete, and for any other such $\Dd$ evaluation at $\id_1$ induces an equivalence
		\[
			\ul\Fun_\Bb^{\Qq\text-\amalg}(\ul\Span(\Qq,\Qq_L,\Qq),\Dd)\iso\Dd.\qednow
		\]
	\end{corollary}

	\subsection{Examples}
	\label{subsec:Examples_univ_prop_spans}

	We will now specialize the universal property of parametrized spans to the examples given in \Cref{subsec:Examples}. Recall that in each of these examples the topos $\Bb$ is a presheaf topos $\PSh(T)$ on a small category $T$ and that the locally inductible subcategory $\Qq$ is of the form $Q_{\loc}$ for some pre-inductible subcategory $Q \subseteq \PSh(T)$. For easier reference, we explicitly spell out this special case of the theorem:

	\begin{theorem}
		\label{thm:universal_prop_par_spans_preinduct}
		For a pre-inductible subcategory $Q \subseteq \PSh(T)$, the free $Q$-semi\-ad\-ditive $T$-category is the $T$-category $\ul{\Span}(Q)$ given by
		\[
			\ul{\Span}(Q)\colon T\catop \to \Cat, \quad B \mapsto \Span(Q_{/B}).
		\]
	\end{theorem}
	\begin{proof}
		Taking $\Bb = \PSh(T)$ and $\Qq = Q_{\loc}$, \Cref{thm:universal_prop_par_spans} says that the free $Q_{\loc}$-semiadditive $\Bb$-category is given by the $\Bb$-category $\ul{\Span}(Q_{\loc})\colon \Bb\catop \to \Cat$ sending $B$ to $\Span((Q_{\loc})_{/B})$. By \Cref{prop:Check_Semiadditivity_Locally}, it follows that its underlying $T$-category is the free $Q$-semiadditive $T$-category. Since for $B \in T$ we have $(Q_{\loc})_{/B} = Q_{/B}$, the claim follows.
	\end{proof}

	\begin{example}[Ordinary semiadditivity]
		Taking $Q = \Fin \subseteq \Spc$ recovers the well-known fact that the category $\Span(\Fin)$ of spans of finite sets is the free semiadditive category on a single generator.
	\end{example}

	\begin{example}[$m$-semiadditivity]\label{ex:spans_universal_m_semi}
		For $-2 \leq m < \infty$, taking $Q = \Spc_m \subseteq \Spc$ recovers the fact that the category $\Span(\Spc_m)$ of spans of $m$-finite spaces is the free $m$-semiadditive category on a single generator, as was previously established by Harpaz \cite{harpaz2020ambidexterity}*{Theorem~1.1}. Similarly, taking $Q = \Spc_{\pi}$ shows that $\Span(\Spc_{\pi})$ is the free $\infty$-semiadditive category on a single generator. Analogous results hold for $\Span(\Spc^{(p)}_m)$ and $\Span(\Spc^{(p)}_{\pi})$ in the $p$-typical setting.
	\end{example}

	\begin{example}[Equivariant semiadditivity]
		For $Q = \FinGrpdfaith \subseteq \PSh(\Glo)$, \cite{CLL_Global}*{Lemma~5.2.3} provides a natural equivalence $Q_{/BG}\simeq\Fin_G$. Thus, we deduce that the assignment $G \mapsto \Span(\Fin_G)$ determines the free equivariantly semiadditive global category.
	\end{example}

	\begin{example}[$G$-semiadditivity]
		For a finite group $G$, taking $Q = \Fin_G \subseteq \Spc_G$ shows that the $G$-category $\ul\Span(\Fin_G)\colon G/H\mapsto {\Span}(\Fin_H)$ of spans of finite $G$-sets is the free $G$-semiadditive $G$-category.
	\end{example}

	\begin{example}[$P$-semiadditivity]
		More generally when $P \subseteq T$ is an atomic orbital subcategory of a small category $T$, we deduce that the $T$-category $\ul{\Span}(\ulfinPsets)$ is the free $P$-semiadditive $T$-category.
	\end{example}

	\begin{example}
			If $Q \subseteq T$ is an inductible subcategory, it is in particular a pre-inductible subcategory of $\PSh(T)$ by \Cref{ex:Inductible_Is_Preinductible}, and hence $\ul{\Span}(Q)$ is the free $Q$-semiadditive $T$-category.
	\end{example}

	\begin{remark}
		The previous example can be used to provide a strengthening of Harpaz's result from \Cref{ex:spans_universal_m_semi}. If we take $T$ to be the category $\Spc_m$ of $m$-finite spaces, then assigning to a category $\Cc$ the $T$-category $\Fun(\blank,\Cc)\colon \Spc_m\catop \to \Cat$ provides a fully faithful inclusion $\Cat \hookrightarrow \Cat_T$, whose essential image consists of those functors $\Spc_m^\op\to\Cat$ that preserve $m$-finite limits. Under this inclusion, the span category $\Span(\Spc_m)$ gets sent to the functor $A\mapsto \Span(\Spc_m)^A \simeq \Span((\Spc_m)_{/A})$, which is precisely the parametrized span category $\ul{\Span}(Q)$ for $T = Q = \Spc_m$. The previous example then tells us that $\ul{\Span}(Q)$ is free among all $\Spc_m$-semiadditive $T$-categories, strengthening Harpaz's statement that it is free among just those contained in the essential image of the inclusion $\Cat \hookrightarrow \Cat_T$.

		This strengthening of Harpaz's result has since been crucially used in work of Shay Ben-Moshe on trans\-chromatic characters \cite{ben-moshe-transchromatic}.
	\end{remark}

	Let us also make the corresponding special case of Corollary~\ref{cor:free-semiadd-more-colims} explicit, which will be a key ingredient in \cite{CLL_Span2} for establishing a universal property for $(\infty,2)$-categories of iterated spans:

	\begin{theorem}
		Let $T$ be any category and let $Q\subseteq S\subseteq T$ be left-cancelable subcategories closed under base change. Assume moreover that every map in $Q$ is truncated (i.e.~$Q$ is inductible). Then
		\[
			\ul\Span(S,Q,S)\colon A\mapsto \Span(S_{/A},S_{/A}\times_SQ,S_{/A})
		\]
		is the free $S$-cocomplete and $Q$-semiadditive $T$-category. More precisely, for any other such $\Dd$ evaluation at the identity section $\pt$ defines an equivalence
		\[
			\ul\Fun^{S\textup-\amalg}(\ul\Span(S,Q,S),\Dd)\iso\Dd.\qednow
		\]
	\end{theorem}

	\section{The \texorpdfstring{$\ul\Span(\Qq)$}{Span(Q)}-tensoring on \texorpdfstring{$\Qq$}{Q}-semiadditive \texorpdfstring{$\Bb$}{B}-categories}

 	In this section, we will show that the property for a $\Bb$-category to be $\Qq$-semiadditive can be characterized via a suitable notion of $\ul{\Span}(\Qq)$-tensorings, generalizing the analogous result of \cite{harpaz2020ambidexterity}*{Section~5.1} in the $m$-semiadditive situation. We will further discuss various useful consequences of this characterization.

	\begin{definition}
		Let $\Cc,\Dd,\Ee$ be $\Bb$-categories. We will use the term \emph{bifunctor} for a $\Bb$-functor $\blank\boxtimes\blank\colon\Cc\times\Dd\to\Ee$.

		If $\Cc$ and $\Ee$ are $\Qq$-cocomplete, then we say such a bifunctor preserves \emph{$\Qq$-colimits in the first variable} if the curried functor $\Cc\to\ul\Fun(\Dd,\Ee)$ is $\Qq$-cocontinuous, or equivalently if the curried functor $\Dd\to\ul\Fun(\Cc,\Ee)$ factors through the full subcategory $\ul\Fun^{\Qq\text-\amalg}(\Cc,\Ee)$. Analogously, we define what it means for a bifunctor to preserve $\Qq$-colimits in the second variable (if $\Dd$ and $\Ee$ have $\Qq$-colimits), or in both variables (if all three of them have $\Qq$-colimits).
	\end{definition}

	Of course, there is a dual notion of \emph{preserving $\Qq$-limits} in some or all of the variables.

	\begin{remark}
		Unraveling the definitions, a functor $F\colon\Cc\times\Dd\to\Ee$ preserves $\Qq$-colimits in the first variable if and only if for every $q\colon A\to B$ in $\Qq$, $X\in\Cc(A)$, and $Y\in\Dd(B)$ the \emph{projection map}
		\begin{equation*}
			q_!(X\boxtimes q^*Y)\to q_!X\boxtimes Y,
		\end{equation*}
		defined as the mate of the naturality equivalence $q^*(\blank\boxtimes Y)= q^*(\blank)\boxtimes q^*Y$ is an equivalence.
	\end{remark}

	\begin{example}\label{rem:tensor_over_spaces}
	Let $\Cc$ be $\Qq$-cocomplete. By Proposition~\ref{prop:uni_prop_U_Q} the evaluation functor $\ev_\pt\colon\ul\Fun^{\Qcoprod}(\ulbbU{\Qq},\Cc)\to\Cc$ is an equivalence, so it has a unique section (automatically an equivalence, hence in particular $\Qq$-cocontinuous). In other words, there is a unique bifunctor $\blank\otimes\blank\colon\ulbbU{\Qq}\times\Cc\to\Cc$ that preserves $\Qq$-colimits in each variable and restricts to the identity on $\{\pt\}\times\Cc$.
	\end{example}

	Using the universal property of spans we can extend this tensoring in the case that $\Cc$ is $\Qq$-\emph{semiadditive}:

	\begin{corollary}\label{cor:tensoring-from-semiadd}
		Let $\Cc$ be $\Qq$-semiadditive. Then the above tensoring $\ulbbU{\Qq}\times\Cc\to\Cc$ extends uniquely to a bifunctor $\ul\Span(\Qq)\times\Cc\to\Cc$. Moreover, this tensoring again preserves $\Qq$-colimits in each variable separately.
		\begin{proof}
			Note that any such extension is necessarily $\Qq$-cocontinuous in each variable by \Cref{cor:test_colimit_pres_span}. Conversely, Theorem~\ref{thm:universal_prop_par_spans} gives by the same argument as in the previous example a unique bifunctor $\ul\Span(\Qq)\times\Cc\to\Cc$ preserving $\Qq$-colimits in each variable and restricting to the identity on $\{\id\}\times\Cc$. Its restriction to $\ulbbU{\Qq}\times\Cc$ then again preserves $\Qq$-colimits in each variable by another application of \Cref{cor:test_colimit_pres_span}, so it necessarily agrees with the canonical tensoring.
		\end{proof}
	\end{corollary}

	The main goal for the rest of this section will be to prove the following converse:

	\begin{theorem}\label{thm:semiadd-from-tensoring}
		Let $\Dd$ be $\Qq$-cocomplete and let $\Vv$ be $\Qq$-semiadditive. Assume there exists a bifunctor $\blank\boxtimes\blank\colon\Vv\times\Dd\to\Dd$ preserving $\Qq$-colimits in each variable separately together with a global section $\mathbb I\in\Gamma\Vv$ such that $\blank\boxtimes\blank$ restricts to the identity on $\{\mathbb I\}\times\Dd$. Then $\Dd$ is $\Qq$-semiadditive.
	\end{theorem}

	In practice one applies the previous theorem in the case that $\Vv$ is $\ul\Span(\Qq)$; however, the notation in the proof is simplified by assuming $\Vv$ is arbitrary.

	\subsection{Bifunctors and the adjoint norm} The key idea to prove the theorem will be to use the maps $\mu\colon\id\to q_!q^*$ in $\Vv$ to construct analogous transformations in $\Cc$ that together with the adjoint norm maps will exhibit $q_!$ as right adjoint to $q^*$. To do so, we will first have to understand the interaction of general bifunctors with these sorts of maps better.

	\begin{lemma}\label{lemma:norm-vs-bifunctor-q*}
		Let $\Cc$, $\Dd$, and $\Ee$ be $\Bb$-categories and assume that $\Cc$ and $\Ee$ are $\Qq$-cocomplete and $\Qq_{\le n-1}$-semiadditive. Let $- \boxtimes -\colon \Cc \times \Dd \to \Ee$ be a bifunctor preserving $\Qq$-colimits in the first variable and let $q\colon A \to B$ be an $n$-truncated map in $\Qq$.
		\begin{enumerate}
			\item The diagram
			\begin{equation*}
				\begin{tikzcd}
					q^*q_!(\id\boxtimes q^*) \arrow[r, "\Nmadj"]\arrow[d, "\textup{proj}"', "\sim"] & \id\boxtimes q^*\\
					q^*(q_!\boxtimes\id)\arrow[r, "="'] & q^*q_!\boxtimes q^*\arrow[u, "\Nmadj\boxtimes q^*"']
				\end{tikzcd}
			\end{equation*}
			of natural transformations between functors $\Cc(A)\times\Dd(B)\to\Ee(A)$ commutes.
			\item Assume that $\Cc$ is even $\Qq_{\le n}$-semiadditive. Then the composite
			\begin{equation*}
				q^*\boxtimes q^* = q^*(\id\boxtimes\id) \xrightarrow{q^*(\mu\boxtimes\id)} q^*(q_!q^*\boxtimes\id) \xrightarrow{\textup{proj}^{-1}} q^*q_!(q^*\boxtimes q^*)\xrightarrow{\Nmadj} q^*\boxtimes q^*
			\end{equation*}
			is the identity.
		\end{enumerate}
	\end{lemma}
	\begin{proof}
		Unravelling definitions, the first part is an instance of Proposition~\ref{prop:adj-norm-vs-B-functor}-$(1)$ applied to the composite $\pi_A^*\Cc\to\ul\Fun(\pi_A^*\Dd,\pi_A^*\Ee)\to(\pi_A^*\Ee)^{\Dd(A)}$. The second part then follows immediately from this together with the triangle identity for the adjunction $q^*\dashv q_!$ in $\Cc$.
	\end{proof}

	We will also need the following complementary result:

	\begin{proposition}\label{prop:norm-vs-bifunctor-q!}
		Let $\Cc,\Dd,\Ee$ be $\Qq$-cocomplete $\Qq_{\le n-1}$-semiadditive $\Bb$-categories, and let $\blank\boxtimes\blank\colon\Cc\times\Dd\to\Ee$ be a bifunctor preserving $\Qq$-colimits in each variable. Moreover, let $q$ be any $n$-truncated map in $\Qq$. Then:
		\begin{enumerate}
			\item The following diagram commutes:
			\begin{equation}\label{diag:norm-bifunc-balanced}
				\begin{tikzcd}
					q_!(q^*\boxtimes q^*q_!)\arrow[d, "q_!(q^*\boxtimes \Nmadj)"']\arrow[r, "\textup{proj}"] &[1em] q_!q^*\boxtimes q_!\arrow[d, "\textup{proj}^{-1}"] \\
					q_!(q^*\boxtimes \id) &\arrow[l, "q_!(q^*\Nmadj \boxtimes\id)"] q_!(q^*q_!q^*\boxtimes\id)\rlap.
				\end{tikzcd}
			\end{equation}
			\item Assume that $\Cc$ is even $\Qq_{\le n}$-semiadditive. Then the following composite is the identity:
			\begin{equation*}
				\id\boxtimes q_!\xrightarrow{\mu\boxtimes\id} q_!q^*\boxtimes q_!\xrightarrow{\textup{proj}^{-1}} q_!(q^*\boxtimes q^*q_!)\xrightarrow{q_!(q^*\boxtimes\Nmadj)} q_!(q^*\boxtimes\id)\xrightarrow{\textup{proj}} \id\boxtimes q_!.
			\end{equation*}
		\end{enumerate}
	\end{proposition}

	Unlike the previous result, this does not seem to directly translate to a statement about one-variable functors, making its proof a bit more involved computationally. We start with the following Beck--Chevalley lemma whose proof we leave to the reader:

	\begin{lemma}
		Let $\Cc,\Dd,\Ee$ be $\Qq$-cocomplete $\Bb$-categories and let $\blank\boxtimes\blank\colon \Cc\times\Dd\to\Ee$ be any bifunctor. Assume moreover we have a commutative diagram
		\begin{equation}\label{diag:base-change-boxtimes-balanced}
			\begin{tikzcd}
				A\arrow[d,"g'"']\arrow[r, "f'"] & B\arrow[d, "g"]\\
				C\arrow[r, "f"'] & D
			\end{tikzcd}
		\end{equation}
		in $\Qq$. Then the diagram
		\[\begin{tikzcd}[cramped]
			{f_!g'_!(f^{\prime*}g^*\boxtimes g^{\prime*})} & {f_!(g'_!f^{\prime*}g^*\boxtimes\id)} &[1em] {f_!(f^*g_!g^*\boxtimes\id)} & {g_!g^*\boxtimes f_!} \\
			{g_!f'_!(f^{\prime*}g^*\boxtimes g^{\prime*})} & {g_!(g^*\boxtimes f'_!g^{\prime*})} & {g_!(g^*\boxtimes g^*f_!)} & {g_!g^*\boxtimes f_!}
			\arrow["{\textup{proj}}", from=1-1, to=1-2]
			\arrow["{\BC_!\boxtimes\id}", from=1-2, to=1-3]
			\arrow["{\textup{proj}}", from=1-3, to=1-4]
			\arrow["{\textup{proj}}"', from=2-3, to=2-4]
			\arrow["{\textup{proj}}"', from=2-1, to=2-2]
			\arrow["{g^*\boxtimes\BC_!}"', from=2-2, to=2-3]
			\arrow[Rightarrow, no head, from=1-4, to=2-4]
			\arrow["\sim"', from=1-1, to=2-1]
		\end{tikzcd}\]
		commutes, where the unlabeled equivalence is induced by $(\ref{diag:base-change-boxtimes-balanced})$.\qed
	\end{lemma}

	\begin{proof}[Proof of Proposition~\ref{prop:norm-vs-bifunctor-q!}]
	Consider the following diagram:
	\begin{equation*}\hskip-46.44pt\hfuzz=46.44pt
		\begin{tikzcd}[cramped]
			q_!(q^*\boxtimes q^*q_!)\arrow[d,"q_!(q^*\boxtimes\BC_!^{-1})"'] \arrow[r, "\textup{proj}"] &[1em] q_!q^*\boxtimes q_!\arrow[r, "\textup{proj}^{-1}"] &[1em] q_!(q^*q_!q^*\boxtimes\id)\arrow[d, "q_!(\BC_!^{-1}\boxtimes\id)"]\\[.5ex]
			q_!(q^*\boxtimes\pr_{2!}\pr_1^*)\arrow[d, "q_!(q^*\boxtimes\pr_{2!}\mu)"']\arrow[r, "\textup{proj}^{-1}"'] & q_!\pr_{2!}(\pr_2^*q^*\boxtimes\pr_1^*)\simeq q_!\pr_{1!}(\pr_2^*q^*\boxtimes\pr_1^*)\arrow[r, "\textup{proj}"'] & q_!(\pr_{1!}\pr_2^*q^*\boxtimes\id)\arrow[d, "q_!(\pr_{1!}\mu\boxtimes\id)"]\\[.5ex]
			q_!(q^*\boxtimes\pr_{2!}\Delta_!\Delta^*\pr_1^*)\arrow[d,"\sim"'] && q_!(\pr_{1!}\Delta_!\Delta^*\pr_2^*q^*\boxtimes\id)\arrow[d, "\sim"]\\
			q_!(q^*\boxtimes\id)\arrow[rr,equal] && q_!(q^*\boxtimes\id)\rlap.
		\end{tikzcd}
	\end{equation*}

	Note that the right-hand vertical column spells out the definition of $q_!(\Nmadj\boxtimes\id)$, while the left-hand column agrees with $q_!(q^*\boxtimes\Nmadj)$ by virtue of Remark~\ref{rk:projections-swapped}.
	Thus, the first statement amounts to saying that the total rectangle commutes.

	To prove this we first note that the top rectangle commutes by the previous lemma. To show that the bottom rectangle commutes, we expand it as follows:
		\begin{equation*}\hskip-80.2pt\hfuzz=80.2pt
		\begin{tikzcd}[cramped]
			q_!(\pr_{1!}\pr_2^*q^*\boxtimes\id)\arrow[d, "q_!(\pr_{1!}\mu\boxtimes\id)"'] & \arrow[l, "\textup{proj}"'] q_!\pr_{1!}(\pr_2^*q^*\boxtimes\pr_1^*)\arrow[d, "q_!\pr_{1!}\mu"'] \arrow[r,"\sim"] & q_!\pr_{2!}(\pr_2^*q^*\boxtimes\pr_1^*)\arrow[r, "\textup{proj}"]\arrow[d, "q_!\pr_{2!}\mu"] & q_!(q^*\boxtimes\pr_{2!}\pr_1^*)\arrow[d, "q_!(q^*\boxtimes \pr_{2!}\mu)"]\\
			q_!(\pr_{1!}\Delta_!\Delta^*\pr_{2}^*q^*\boxtimes\id)\arrow[dr, bend right=15pt, "\sim"'] & q_!\pr_{1!}\Delta_!(\Delta^*\pr_2^*q^*\boxtimes\Delta^*\pr_{1}^*)\arrow[r, "\sim"]\arrow[d,"\sim"'] & \arrow[d,"\sim"]q_!\pr_{2!}\Delta_!(\Delta^*\pr_2^*q^*\boxtimes\Delta^*\pr_1^*)\arrow[d, "\sim"] & q_!(q^*\boxtimes\pr_{2!}\Delta_!\Delta^*\pr_1^*)\arrow[dl, bend left=15pt, "\sim"]
			\\
			& q_!(q^*\boxtimes\id) \arrow[r,equal] & q_!(q^*\boxtimes\id)
		\end{tikzcd}
	\end{equation*}
	The top square in the middle column commutes by naturality, and the bottom middle square commutes since the composite homotopy $q\pr_1\Delta\simeq q\pr_2\Delta\simeq q$ agrees with the homotopy induced by $\pr_1\Delta\simeq\id$. To see that the two pentagons commute we begin by observing that they are symmetric, and so it suffices to spell out the left-hand one. For this we expand it once again:
		\begin{equation*}\hskip-11.5pt\hfuzz=11.6pt
		\begin{tikzcd}[cramped]
			q_!(\pr_{1!}\pr_{2}^*q^*\boxtimes\id)\arrow[d, "\mu"'] & \arrow[l, "\textup{proj}"'] q_!\pr_{1!}(\pr_2^*q^*\boxtimes\pr_1^*)\arrow[d, "\mu"]\arrow[r, "\mu"] & q_!\pr_{1!}\Delta_!\Delta^*(\pr_2^*q^*\boxtimes\pr_1^*)\arrow[d,equal]\\
			q_!(\pr_{1!}\Delta_!\Delta^*\pr_{2}^*q^*\boxtimes\id)\arrow[d, "\sim"'] &\arrow[l, "\textup{proj}"'] q_!\pr_{1!}(\Delta_!\Delta^*\pr_2^*q^*\boxtimes\pr_1^*)\arrow[d,"\sim"]&\arrow[l, "\,\textup{proj}"'] q_!\pr_{1!}\Delta_!(\Delta^*\pr_{2}^*q^*\boxtimes\Delta^*\pr_{1}^*)\arrow[d, "\sim"] \\
			q_!(\pr_{1!}\Delta_!q^*\boxtimes\id)\arrow[d, "\sim"'] & \arrow[l, "\textup{proj}"] q_!\pr_!(\Delta_!q^*\boxtimes\pr_1^*) & \arrow[l, "\textup{proj}"] q_!\pr_{1!}\Delta_!(q^*\boxtimes\Delta^*\pr_1^*)\arrow[d, "\sim"]\\
			q_!(q^*\boxtimes\id) \arrow[rr,equal] && q_!(q^*\boxtimes\id)\rlap.
		\end{tikzcd}
	\end{equation*}
	Here the top left square as well as the two squares in the middle row commute by naturality, the top right square commutes by Lemma~\ref{lemma:norm-vs-bifunctor-q*} (note that $\Delta$ is $(n-1)$-truncated), and the bottom rectangle commutes by direct inspection. This finishes the proof of the first statement.

	For the second statement, we consider the diagram
	\begin{equation*}
		\begin{tikzcd}
			q_!(q^*\boxtimes\id)\arrow[r, "(q^*\mu\boxtimes\id)"]\arrow[d, "\textup{proj}"', "\sim"] &[2em] q_!(q^*q_!q^*\boxtimes\id)\arrow[d, "\textup{proj}","\sim"'] \arrow[r, "q_!(\Nmadj\boxtimes\id)"] &[3em] q_!(q^*\boxtimes\id)\\
		\id\boxtimes q_!\arrow[r, "\mu\boxtimes\id"'] & q_!q^*\boxtimes\id\arrow[r, "\textup{proj}^{-1}"',"\sim"] & q_!(q^*\boxtimes q^*q_!) \arrow[u, "q_!(q^*\boxtimes\Nmadj)"']
		\end{tikzcd}
	\end{equation*}
	where the left square commutes by naturality while the right-hand square commutes by part $(1)$. As the composite of the top row is the identity by the triangle identity in $\Cc$, the claim follows.
\end{proof}

	Using this, we can now easily prove the theorem:

	\begin{proof}[Proof of Theorem~\ref{thm:semiadd-from-tensoring}]
		It suffices by induction to show that if $\Cc$ is $\Qq_{\le n}$-semiadditive and if $q\colon A\to B$ is any $(n+1)$-truncated map, then $\Nmadj\colon q^*q_!\to\id$ is the counit of an adjunction.

		We claim that the natural map $m\colon\id\to q_!q^*$ defined as the composite
		\begin{equation*}
			\id=\mathbb I\boxtimes\id\xrightarrow{\;\mu\boxtimes\id\;} q_!q^*\mathbb I\boxtimes\id \xrightarrow{\;\textup{proj}^{-1}\,} q_!(q^*\mathbb I\boxtimes q^*)=q_!q^*
		\end{equation*}
		provides a compatible unit, i.e.~we have to verify the triangle identities.	The identity $(\Nmadj q^*)\circ (q^* m)=\id$ is simply a special case of Lemma~\ref{lemma:norm-vs-bifunctor-q*}-$(2)$, while $(q_!\Nmadj)\circ(m q_!)=\id$ follows from Proposition~\ref{prop:norm-vs-bifunctor-q!}-$(2)$ as the projection map $q_!(q^*\boxtimes\id)\to \id\boxtimes q_!$ is simply the identity when restricted to $\mathbb I$ in the first component.
	\end{proof}

	\subsection{Applications}
	The characterization of semiadditivity in terms of the existence of a $\ul\Span(\Qq)$-tensoring has various interesting consequences:

	\begin{theorem}\label{thm:Q-x-out-of-semiadd}
		Let $\Cc,\Dd$ be $\Bb$-categories such that $\Cc$ is $\Qq$-semiadditive and $\Dd$ is $\Qq$-complete. Then $\ul\Fun^{\Qq\text-\times}(\Cc,\Dd)$ is again $\Qq$-semiadditive.
	\begin{proof}
		By Proposition~\ref{prop:fun-colimits}${}^\op$, $\ul\Fun(\Cc,\Dd)$ is $\Qq$-complete. If $F\colon\Cc\to\ul\Fun(\ul A,\Dd)$ defines any object of $\ul\Fun(\Cc,\Dd)(A)$ and $q\colon A\to B$ is any map in $\Qq$, then $q_*F$ is simply given by the composition
		\[
			\Cc\xrightarrow{\;F\;}\ul\Fun(\ul A,\Dd)\xrightarrow{\,q_*\,}\ul\Fun(\ul B,\Dd),
		\]
		hence $\Qq$-continuous as a composition of $\Qq$-continuous functors. In other words, $\ul\Fun^{\Qq\text-\times}(\Cc,\Dd)$ is closed under $\Qq$-limits and hence in particular itself $\Qq$-complete. It will therefore suffice by Theorem~\ref{thm:semiadd-from-tensoring}${}^\op$ to construct a functor
			\begin{equation*}
				\ul\Span(\Qq)\times\ul\Fun^{\Qq\text-\times}(\Cc,\Dd)\to\ul\Fun^{\Qq\text-\times}(\Cc,\Dd)
			\end{equation*}
			preserving $\Qq$-limits in each variable and restricting to the identity on $\{\pt\}\times\ul\Fun^{\Qq\text-\times}(\Cc,\Dd)$. Adjoining over, this amounts to constructing a $\Qq$-continuous section
			\begin{equation}\label{eq:cmon-cotensoring-to-construct}
			\ul\Fun^{\Qq\text-\times}(\Cc,\Dd)\to\ul\Fun^{\Qq\text-\times}(\ul\Span(\Qq),\ul\Fun^{\Qq\text-\times}(\Cc,\Dd))
			\end{equation}
			of the evaluation functor.

			For this we observe that we have by Corollary~\ref{cor:tensoring-from-semiadd}${}^\op$ a functor $-\boxtimes-\colon\ul\Span(\Qq)\times\Cc\to\Cc$ preserving $\Qq$-limits in each variable and restricting to the identity on $\{\pt\}\times\Cc$. Using this, we consider the composite
			\begin{equation*}
				\ul\Fun(\Cc,\Dd)\xrightarrow{\;\boxtimes^*}\ul\Fun(\ul\Span(\Qq)\times\Cc,\Dd)\isoname{\,\textup{adjunction}\,} \ul\Fun\big(\ul\Span(\Qq),\ul\Fun(\Cc,\Dd)\big).
			\end{equation*}
			We claim that this restricts to the desired section $(\ref{eq:cmon-cotensoring-to-construct})$. Indeed, one easily checks that this is $\Qq$-continuous (in fact, it preserves all limits that exist in $\Dd$) and a section. The claim that it lands in $\ul\Fun^{\Qq\text-\times}(\ul\Span(\Qq),\ul\Fun^{\Qq\text-\times}(\Cc,\Dd))$ amounts to saying that for every $\Qq$-continuous $F\colon\Cc\to\Dd^{\ul A}$ the composite $F(\blank\boxtimes\blank)\colon\ul\Span(\Qq)\times\Cc\to\Dd^{\ul A}$ preserves $\Qq$-limits in each variable. This is clear by assumption on $\boxtimes$.
		\end{proof}
	\end{theorem}

	Finally, let us use the tensoring to show that the left and right adjoints $q_!$ and $q_*$ also agree for non-truncated $q$, and to moreover upgrade this to a parametrized comparison.

	\begin{construction}
		Fix $q\colon A\to B$ in $\Qq$, and let $\Cc$ be $\Qq$-semiadditive. We define a natural transformation $m$ from the identity of $\pi_B^*\Cc$ to the composite
		\[
			\pi_B^*\Cc\xrightarrow{\;q^*\;}\Cc(A\times_B\blank)\xrightarrow{\;q_!\;}\pi_B^*\Cc
		\]
		as follows: write $\boxtimes$ for the (essentially unique) $\ul\Span(\Qq)$-tensoring of $\Cc$ and let $\pt=(\id_B\colon B\to B)$ denote the prefered global section of $\pi_B^*\ul\Span(\Qq)$; then we define $m$ as the composite
		\[
			\id_{\pi_B^*\Cc} = \pt\boxtimes\id \xrightarrow{p\boxtimes\id} q_!q^*\pt\boxtimes\id \iso q_!(q^*\pt\boxtimes q^*)=q_!q^*
		\]
		where $p\colon\pt\to q_!q^*\pt$ is represented by the span $B\xleftarrow{\,q\,}A\xrightarrow{\,=\,}B$ and the unlabeled equivalence comes from the projection formula, i.e.~it is the adjunct of the map $q^*\pt\boxtimes q^*\to q^*q_!q^*\pt\boxtimes q^*$ induced by the unit.

		Unravelling definitions, we see that for any $f\colon B'\to B$ and $q'\coloneqq f^*(q)\colon A'\to B'$, the value of $m$ on $f$ is precisely the transformation $m_{q'}\colon \id\to q_!'q^{\prime*}$ considered in the proof of Theorem~\ref{thm:semiadd-from-tensoring}. In particular, we have shown that it is the unit of an adjunction $q^{\prime*}\dashv q'_!$ whenever $q'$ is truncated, with corresponding counit $\Nmadj$.
	\end{construction}

	\begin{theorem}\label{thm:norm-non-truncated}
		Let $\Cc$ be $\Qq$-semiadditive and let $q\colon A\to B$ be any map in $\Qq$, not necessarily truncated. Then the above transformation $m\colon\id\to q_!q^*$ is the unit of a parametrized adjunction $q^*\dashv q_!$ of $\Bb_{/B}$-categories. In particular, there is a natural equivalence $\Nm_q\colon q_!\iso q_*$ of $\Bb_{/B}$-functors.
	\end{theorem}

	As explained before, the parametrized transformation $\Nm_q$ recovers the usual inductively defined norm on underlying categories whenever $q$ is truncated.

	\begin{proof}
		Fix $q\colon A\to B$ in $\Qq$; to simplify notation, we will replace $\Bb$ by $\Bb_{/B}$, so that $B$ is terminal. We have to show that $m$ is the unit of an adjunction, or equivalently that $q_!$ admits a left adjoint $q^!$ and that the adjoined map $q^!\to q^*$ is an equivalence.

		For the construction of $q^!$, we once more consider the covering sieve $\Sigma\subseteq\Bb$ of the terminal object given by those $f\colon B'\to 1$ such that the pulled back map $q'\coloneqq f^*(q)\colon B'\times A\to B'$ is truncated. By the above, we then have for each such $q'$ an adjunction $(q')^! = q^{\prime*}\dashv q'_!$ with unit $m_{q'}$; the fact that $m$ is a natural transformation of $\Bb$-functors then translates to saying that for every map $g\colon B''\to B'$ in $\Sigma$ the Beck--Chevalley transformation $(q'')^!(A\times g)^*\to g^*(q')^!$ of $(A\times g)^*$ is just the naturality equivalence $(q'')^*(A\times g)^*\to g^*(q')^*$ again, in particular invertible. Thus, Proposition~\ref{prop:adj-criterion-cover} shows that the parametrized left adjoint $q^!$ exists.

		Consider now the natural transformation of $\Bb$-functors $\widetilde m\colon q^!\to q^*$ induced by $m$. By construction of $q^!$, $\tilde m_{B'}$ is an equivalence (even the identity) whenever $B'\in\Sigma$. Given an arbitrary $B\in\Bb$, we can now cover it by objects of $\Sigma$ (for example via the projections $C\times B\to B$ with $C\in\Sigma$). It therefore follows immediately from naturality and descent that $\widetilde m$ is an equivalence as claimed.
	\end{proof}

	\begin{corollary}\label{cor:Characterization_Q_Semiadditivity_nontruncated}
		Let $\Cc$ be a $\Qq$-complete and $\Qq$-cocomplete $\Bb$-category. The following are equivalent:
		\begin{enumerate}
			\item $\Cc$ is $\Qq$-semiadditive.
			\item For every $q\colon A\to B$ in $\Qq$ there exists an equivalence of parametrized functors $q_!\simeq q_*\colon\ulFun(\ul A,\pi_B^*\Cc)\to\pi_B^*\Cc$ (a priori unrelated to the norms).
			\item For every $q\colon A\to B$ in $\Qq$, the functor $q_!\colon\ulFun(\ul A,\pi_B^*\Cc)\to\pi_B^*\Cc$ admits a parametrized left adjoint.
			\item For every $q\colon A\to B$ in $\Qq$, the functor $q_!\colon\ul\Fun(\ul A,\pi_B^*\Cc)\to\pi_B^*\Cc$ preserves $\Qq$-limits.
		\end{enumerate}
		\begin{proof}
			Clearly, $(2)\Rightarrow(3)\Rightarrow(4)$. The implication $(1)\Rightarrow(2)$ is the content of the previous theorem, while $(4)\Rightarrow(1)$ follows from Proposition~\ref{prop:Characterization_Q_Semiadditivity}.
		\end{proof}
	\end{corollary}

	\section{The universal property of \texorpdfstring{$\Qq$}{Q}-commutative monoids}
	In this section, we introduce the $\Bb$-category $\ul{\CMon}^{\Qq}(\Dd)$ of \textit{$\Qq$-commutative monoids} in a $\Qq$-complete $\Bb$-category $\Dd$, and show that the forgetful functor $\ul{\CMon}^{\Qq}(\Dd) \to \Dd$ exhibits it as the \textit{$\Qq$-semiadditive completion} of $\Dd$.

	\subsection{\texorpdfstring{\except{toc}{$\bm{\Qq}$}\for{toc}{$\Qq$}}{Q}-commutative monoids}\label{subsec:univ_prop_cmon}
	The following is the key definition of this section:
	\begin{definition}[Commutative monoids]
		Given a $\Qq$-complete $\Bb$-category $\Dd$, we define its $\Bb$-category of \textit{$\Qq$-commutative monoids} as
		\[
		\ul\CMon^{\Qq}(\Dd) \coloneqq \ul{\Fun}^{\Qq\text-\times}(\ul{\Span}(\Qq), \Dd).
		\]
		We let $\mathbb{U} \coloneqq \ev_{\pt}\colon \ul\CMon^{\Qq}(\Dd) \to \Dd$ denote the evaluation functor at the point.
	\end{definition}

	Informally speaking, we may think of a $\Qq$-commutative monoid in $\Dd$ as a global section $M$ of $\Dd$ equipped with certain `parametrized addition/transfer maps.' Recall from \Cref{prop:uni_prop_U_Q}$\catop$ that every global section $M$ uniquely extends to a $\Qq$-continuous functor $\ulbbU{\Qq}\catop \to \Dd$ given at level $B \in \Bb$ by sending a map $q\colon A \to B$ in $\Qq$ to the $q$-indexed product $(M_B)^A \coloneqq q_*q^*M_B \in \Dd(B)$. Enhancing $M$ to a $\Qq$-commutative monoid in $\Dd$ is then equivalent to providing an extension of this functor along the inclusion $\ulbbU{\Qq}\catop \hookrightarrow \ul{\Span}(\Qq)$, which we may interpret as providing a suitably coherent collection of `addition/transfer maps' $\smallint_{\!q}\colon (M_B)^A \to M_B$.

	\begin{remark}
	For $A\in\Bb$ a $\Bb_{/A}$-functor $F\colon \ul\Span(\pi_A^{-1}\Qq)\simeq\pi_A^*\ul\Span(\Qq)\to\pi_A^*\Dd$ belongs to $\ul\CMon^{\Qq}(\Dd)(A)$ if and only if it is a Segal functor in the sense of Definition~\ref{defi:segal-functor}, see Proposition~\ref{prop:Functor_CoSegal_Iff_QColimit_Preserving}${}^\op$.
	\end{remark}

	Let us note the following immediate consequence of Theorem~\ref{thm:Q-x-out-of-semiadd}:

	\begin{corollary}
		In the above situation, $\ul\CMon^\Qq(\Dd)$ is $\Qq$-semiadditive.\qed
	\end{corollary}

	In fact, it is the universal $\Qq$-semiadditive completion of $\Dd$ in the following sense:

	\begin{theorem}
		\label{thm:universal_prop_par_Qcom_monoids}
		We have an adjunction $\textup{incl}\colon\Cat^{\Qq\text-\oplus}_{\Bb} \rightleftarrows \Cat^{\Qq\text-\times}_{\Bb}\noloc\ul\CMon^\Qq$, with counit given by the evaluation functor $\mathbb U=\ev_\pt\colon\ul\CMon^\Qq(\Dd)\to\Dd$.
		\begin{proof}
			First observe that $\ul\CMon^\Qq$ indeed lands in $\Qq$-semiadditive categories by the previous corollary. Moreover, Proposition~\ref{prop:fun-colimits}${}^\op$ shows that it preserves $\Qq$-continuous functors and that $\mathbb U$ is $\Qq$-continuous. Thus, it only remains to show that $\mathbb U\colon\ul\CMon^\Qq(\Dd)\to\Dd$ is an equivalence for every $\Qq$-semiadditive $\Dd$ and that $\ul\CMon^\Qq(\mathbb U)\colon\ul\CMon^\Qq(\ul\CMon^\Qq(\Dd))\to\ul\CMon^\Qq(\Dd)$ is an equivalence for every $\Qq$-complete $\Dd$.

			The first statement is precisely the content of Theorem~\ref{thm:universal_prop_par_spans}. Similarly, appealing to the previous corollary once more, we know that $\mathbb U_{\ul\CMon^\Qq(\Dd)}$ is an equivalence for every $\Qq$-complete $\Dd$. Thus, it will suffice for the second statement that the automorphism of $\ul\Fun(\ul\Span(\Qq),\ul\Fun(\ul\Span(\Qq),\Dd))$ exchanging the two span-factors induces an automorphism of $\ul\CMon^\Qq(\ul\CMon^\Qq(\Dd))$. But this follows immediately from the observation that the adjunction equivalence $\ul\Fun(\ul\Span(\Qq),\ul\Fun(\ul\Span(\Qq),\Dd))\simeq\ul\Fun(\ul\Span(\Qq)\times\ul\Span(\Qq),\Dd)$ identifies $\ul\CMon^\Qq(\ul\CMon^\Qq(\Dd))$ with the full subcategory of functors preserving $\Qq$-limits in each variable separately.
		\end{proof}
	\end{theorem}

	We can further refine this universal property to a statement about parametrized functor categories, generalizing previous results due to Nardin for equivariant semiadditivity \cite{nardin2016exposeIV}*{Corollary~5.11.1 and Theorem~6.5} and due to Harpaz for higher non-parametrized semiadditivity \cite{harpaz2020ambidexterity}*{Corollary~5.15}:

	\begin{theorem}
		\label{thm:universal_prop_par_Qcom_monoids2}
		Let $\Cc$ be $\Qq$-semiadditive and let $\Dd$ be $\Qq$-complete. Then postcomposition with $\mathbb U$ defines an equivalence of $\Bb$-categories
		\begin{equation*}
			\ul\Fun^{\Qq\text-\times}(\Cc,\ul\CMon^\Qq(\Dd))\iso\ul\Fun^{\Qq\text-\times}(\Cc,\Dd).
		\end{equation*}
		\begin{proof}
			It suffices to show that for any $\Qq$-complete $\Bb$-category $\Tt$ the induced map
			\begin{equation}\label{diag:Q-semiadd-fun-univ-prop-hom-test}
				{\hom}_{\Cat(\Bb)^{\Qq\text-\times}}\big(\Tt,\ul\Fun^{\Qq\text-\times}(\Cc,\ul\CMon^\Qq(\Dd))\big)\to{\hom}_{\Cat(\Bb)^{\Qq\text-\times}}\big(\Tt,\ul\Fun^{\Qq\text-\times}(\Cc,\Dd)\big)
			\end{equation}
			is an equivalence. However, using the adjunction equivalence for $\ul\Fun$ and keeping track of $\Qq$-limit conditions as in the proof of the previous theorem, this map agrees up to equivalence with the map
			\begin{equation*}
				{\hom}_{\Cat^{\Qq\text-\times}_\Bb}\big(\Cc,\ul\CMon^\Qq(\ul\Fun^{\Qq\text-\times}(\Tt,\Dd))\big)\to{\hom}_{\Cat^{\Qq\text-\times}_\Bb}\big(\Cc,\ul\Fun^{\Qq\text-\times}(\Tt,\Dd)\big)
			\end{equation*}
			induced by $\mathbb U$, so this is a consequence of the previous theorem.
		\end{proof}
	\end{theorem}

	We close this subsection by giving two variants of the above universal property. Both of these rely on the following observation:

	\begin{lemma}\label{lemma:U-pres-refl-lim}
		Let $\mathcal R\subseteq\Bb$ be any local class. If $\Dd$ is $\mathcal R$-complete and $\Qq$-complete, then $\ul\CMon^\Qq(\Dd)$ is $\mathcal R$-complete and $\mathbb U\colon\ul\CMon^\Qq(\Dd)\to\Dd$ preserves and reflects $\mathcal R$-limits.

		Similarly, if $K$ is any non-parametrized category and $\Dd$ admits fiberwise $K$-shaped limits, then so does $\ul\CMon^\Qq(\Dd)$, and $\mathbb U$ preserves and reflects $K$-shaped limits.
		\begin{proof}
			Proposition~\ref{prop:fun-colimits} shows that the full functor category $\Fun(\ul\Span(\Qq),\Dd)$ has all $\mathcal R$-limits if $\Dd$ has them and that the evaluation functor $\ev_\pt$ preserves $\mathcal R$-limits in this case. As $\CMon^\Qq(\Dd)\subseteq\ul\Fun(\ul\Span(\Qq),\Dd)$ is closed under $\mathcal R$-limits by the same argument as in the beginning of the proof of Theorem~\ref{thm:Q-x-out-of-semiadd}, we then get the same statement for $\mathbb U\colon\ul\CMon^\Qq(\Dd)\to\Dd$.

			The existence and preservation of fiberwise limits follows in the same way from Remark~\ref{rk:fun-colimits-full}. To finish the proof it then suffices to prove that $\mathbb U$ is conservative for every $\Qq$-complete $\Dd$. But $\mathbb U$ factors as the composite
			\[
				\ul\Fun^{\Qq\text-\times}(\ul\Span(\Qq),\Dd)\xrightarrow{\,\textup{res}\;}\ul\Fun^{\Qq\text-\times}(\ulbbU{\Qq}^\op,\Dd)\xrightarrow{\;\ev\;}\Dd
			\]
			and the first functor is conservative as $\ulbbU{\Qq}^\op\subseteq\ul\Span(\Qq)$ is a wide subcategory, while the second functor is even an equivalence by Proposition~\ref{prop:uni_prop_U_Q}${}^\op$.
		\end{proof}
	\end{lemma}

	\begin{corollary}
		\label{cor:Unique_Lift_To_CMon}
		Let $\Cc$ be $\Qq$-semiadditive, let $\Dd$ be $\Qq$-complete, and let $F\colon\Cc\to\Dd$ be $\Qq$-continuous. Then $F$ lifts uniquely to a functor $\Cc\to\ul\CMon^\Qq(\Dd)$.
		\begin{proof}
			By Theorem~\ref{thm:universal_prop_par_Qcom_monoids} there is a unique such lift that is in addition $\Qq$-continuous, while the previous lemma shows that in fact \emph{any} lift is $\Qq$-continuous.
		\end{proof}
	\end{corollary}

	\begin{corollary}\label{cor:smashing}
		Let $\Cc,\Dd$ be complete $\Bb$-categories and assume $\Cc$ is $\Qq$-semiadditive. Then postcomposition with the forgetful functor induces an equivalence
		\[
			\ul\Fun^\textup{R}(\Cc,\ul\CMon^\Qq(\Dd))\to\ul\Fun^\textup{R}(\Cc,\Dd)
		\]
		of $\Bb$-categories of continuous functors.
		\begin{proof}
			This follows from Theorem~\ref{thm:universal_prop_par_Qcom_monoids2}, observing that by Lemma~\ref{lemma:U-pres-refl-lim} a functor $\Cc\to\ulFun(\ul A,\ul\CMon^\Qq(\Dd))$ is continuous if and only if the induced functor $\Cc\to\ulFun(\ul A,\Dd)$ is so.
		\end{proof}
	\end{corollary}

	\subsection{Non-parametrized (higher) semiadditivity}
	Let us make explicit how \Cref{thm:universal_prop_par_Qcom_monoids2} recovers various results from non-parametrized higher category theory, and in particular Harpaz's result alluded to above. Recall once more that taking global section defines an equivalence $\Cat(\Spc)\iso\Cat$, and hence $\Spc$-parametrized functor categories are just non-parametrized functor categories between the underlying categories. By Example~\ref{ex:non-param-colim}, the subcategory of $\Qq$-limit preserving functors then consists precisely of those functors that preserve $A$-shaped limits in the usual sense for every $A\in\Qq_{/1}\subseteq\Spc\subseteq\Cat$.

	\begin{example}[Commutative monoids]
		Applying the theorem to the pre-inductible subcategory $\Fin \subseteq \Spc$ of finite sets recovers the well-known result that for every category $\Cc$ with finite products the forgetful functor
		\[
			\Fun^{\times}(\Span(\Fin),\Cc) \xrightarrow{\ev_{\pt}} \Cc
		\]
		exhibits its source as the universal semiadditive category equipped with a finite-product-preserving functor to $\Cc$. In other words, we have an equivalence of categories $\CMon(\Cc) \simeq \Fun^{\times}(\Span(\Fin),\Cc)$, as was previously established by Cranch \cite{cranch}*{Theorem~5.4} (for an ad-hoc construction of $\Span(\Fin)$) or in \cite{BachmannHoyois2021Norms}*{Proposition~C.1} (using Barwick's construction of $\Span$).
	\end{example}

	\begin{example}[$m$-commutative monoids, {\cite{harpaz2020ambidexterity}*{Corollary~5.14}}]
		More generally, consider the pre-inductible subcategory $\Spc_m \subseteq \Spc$ of $m$-finite spaces for some $-2 \leq m < \infty$, and let $\Cc$ be a category with $m$-finite limits. Then Theorem \ref{thm:universal_prop_par_Qcom_monoids2} translates to saying that the forgetful functor
		\[
			\CMon^m(\Cc) := \Fun^{m\text{-fin}}(\Span(\Spc_m), \Cc) \xrightarrow{\ev_{\pt}} \Cc
		\]
		exhibits $\CMon^m(\Cc)$ as the universal $m$-semiadditive category equipped with an $m$-finite limit-preserving functor to $\Cc$, a fact previously proven by Harpaz \cite{harpaz2020ambidexterity}*{Corollary~5.14}. If we instead consider the subcategory $\Spc_{\pi} \subseteq \Spc$, we also obtain the analogous statement for $m = \infty$, previously proven by Carmeli, Schlank, and Yanovski \cite{CSY2021AmbiHeight}*{Proposition~2.1.16}.
	\end{example}

	\begin{example}[$p$-typical $m$-commutative monoids]
		As a new variant of the previous example, we may consider for $-2 \leq m < \infty$ the pre-inductible subcateory $\Spc^{(p)}_m \subseteq \Spc$ of \textit{$p$-typical} $m$-finite spaces from \Cref{ex:p_typical}. For a category $\Cc$ admitting $\Spc^{(p)}_m$-indexed limits, we define the category $\CMon^m_{(p)}(\Cc)$ of \textit{$p$-typically $m$-commutative monoids in $\Cc$} as the full subcategory
		\[
			\CMon^m_{(p)}(\Cc) \subseteq \Fun(\Span(\Spc^{(p)}_m), \Cc),
		\]
		of functors which preserve $\Spc^{(p)}_m$-indexed limits. The evaluation functor to $\Cc$ then enjoys the `$p$-typical analogue' of the previous universal property.
	\end{example}

		\subsection{\texorpdfstring{\for{toc}{$\Qq$}\except{toc}{$\bm{\Qq}$}}{Q}-stability} Building on the result of \Cref{subsec:univ_prop_cmon}, we can now introduce a notion of \emph{$\Qq$-stability} generalizing our work in \cite{CLL_Global}. For this let us first recall the notion of \emph{fiberwise stability} from \cite{nardin2016exposeIV}*{Definition~3.5} and \cite{martiniwolf2022presentable}*{Definition~7.3.4}:

	\begin{definition}
		A $\Bb$-category $\Dd\colon\Bb^\op\to\Cat$ is called \emph{fiberwise stable} if it factors through the non-full subcategory $\Cat^{\textup{ex}}\subseteq\Cat$ of stable categories and exact functors. We write $\Cat(\Bb)^\textup{fib\,ex}\mathrel{:=}\Fun^\textup{R}(\Bb^\op,\Cat^\textup{ex})$ and call its maps \emph{fiberwise exact}.\footnote{Note that $\Cat^{\textup{ex}}$ admits all limits, which are computed in $\Cat$, by \cite{HA}*{Theorem 1.1.4.4}.}
	\end{definition}

	\begin{remark}
		Write $\Cat^\textup{lex}$ for the category of left exact functors between categories with finite limits, and set $\Cat(\Bb)^\textup{fib\,lex}\coloneqq\Fun^\textup{R}(\Bb^\op,\Cat^\textup{lex})$. Then the inclusion $\Cat(\Bb)^\textup{fib\,ex}\hookrightarrow\Cat(\Bb)^\textup{fib\,lex}$ admits a right adjoint $\Sp^\textup{fib}$ given by postcomposing with the right adjoint $\Sp$ of $\Cat^\textup{lex}\hookrightarrow\Cat^\textup{ex}$.
	\end{remark}

	\begin{definition}
		We say that $\Dd$ is \emph{$\Qq$-stable} if it is both $\Qq$-semiadditive and fiberwise stable. We write $\Cat(\Bb)^{\Qq\text{-ex}}\coloneqq\Cat(\Bb)^{\Qq\text-\oplus}\cap\Cat(\Bb)^{\text{fib}\,\text{ex}}$ for the category whose objects are the $\Qq$-stable categories and whose morphisms are the functors that are both fiberwise exact and $\Qq$-semiadditive.
	\end{definition}

	\begin{lemma}\label{lemma:sp-fib-q-ex}
		The adjunction $\textup{incl}\colon\Cat(\Bb)^\textup{fib\,ex}\rightleftarrows\Cat(\Bb)^\textup{fib\,lex}\noloc\Sp^\textup{fib}$ restricts to an adjunction $\Cat(\Bb)^{\Qq\textup{-ex}}\rightleftarrows\Cat(\Bb)^{\Qq\text-\oplus,\,\textup{lex}}$.
		\begin{proof}
			Recall that the functor $\Sp=\Fun_*^\textup{exc}(\Spc^\text{fin}_*,\blank)$ can be extended to an $(\infty,2)$-functor; all that we will need below is that it induces a functor on homotopy $(2,2)$-categories. Observe now that each $q^*\colon\Cc(B)\to\Cc(A)$ has a \emph{left exact} left adjoint $q_!\simeq q_*$ by $\Qq$-semiadditivity, so $2$-functoriality of $\Sp$ implies that also $\Sp(q^*)=q^*\colon\Sp^{\fib}(\Cc)(B)\to\Sp^{\fib}(\Cc)(A)$ has a left adjoint given by $\Sp(q_!)$ with the induced unit and counit. The Beck--Chevalley maps for these left adjoints are then again obtained from the Beck--Chevalley maps in $\Cc$ via applying $\Sp$, and in particular the left adjoints again satisfy base change, i.e.~$\Sp^{\fib}(\Cc)$ is $\Qq$-cocomplete. In the same way, one shows that $\Sp^{\fib}(\Cc)$ is $\Qq$-complete and that $\Sp^{\fib}$ sends $\Qq$-continuous functors to $\Qq$-continuous functors.

			It will then be enough to show by Proposition~\ref{prop:Characterization_Q_Semiadditivity} that for any pullback
			\begin{equation*}
				\begin{tikzcd}
					A\times_BA\arrow[r, "\pr_1"]\arrow[d,"\pr_2"']\arrow[dr,phantom,"\lrcorner"{very near start}] & A\arrow[d, "q"]\\
					A\arrow[r, "q"'] & B
				\end{tikzcd}
			\end{equation*}
			in $\Qq$ the double Beck--Chevalley map $q_!\pr_{1*}\to q_*\pr_{2!}$ for $\Sp^{\fib}(\Cc)$ is an equivalence. This however follows again immediately from $2$-functoriality of $\Sp^{\fib}$ and the corresponding statement for $\Cc$.
		\end{proof}
	\end{lemma}

	\begin{definition}
		Let $\Cc$ be a $\Bb$-category with $\Qq$-limits and finite fiberwise limits. We define $\ul\Sp^\Qq(\Cc)\coloneqq \ul\Sp^{\fib}(\ul\CMon^{\Qq}(\Cc))$, and we write $\Omega^\infty\colon\ul\Sp^\Qq(\Cc)\to\Cc$ for the composite
		\[
		\ul\Sp^{\fib}(\ul\CMon^{\Qq}(\Cc))\xrightarrow{\,\Omega^\infty\,} \ul\CMon^\Qq(\Cc)\xrightarrow{\;\mathbb U\;}\Cc.
		\]
	\end{definition}

	\begin{proposition}
		This defines a functor $\ul\Sp^\Qq\colon\Cat(\Bb)^{\Qq\text-\times,\,\textup{fib\,lex}}\to\Cat(\Bb)^{\Qq\textup{-ex}}$ right adjoint to the inclusion, with counit given by $\Omega^\infty$.
		\begin{proof}
			It suffices by Lemma~\ref{lemma:sp-fib-q-ex} to show that the adjunction $\textup{incl}\colon\Cat^{\Qq\text-\oplus}_\Bb\rightleftarrows\Cat^{\Qq\text-\times}_\Bb\noloc\ul\CMon^\Qq$ from Theorem~\ref{thm:universal_prop_par_Qcom_monoids} restricts to $\Cat^{\Qq\text-\oplus,\,\textup{fib\,lex}}_\Bb\rightleftarrows\Cat^{\Qq\text-\times,\,\textup{fib\,lex}}_\Bb$. This is immediate from \Cref{lemma:U-pres-refl-lim}.
		\end{proof}
	\end{proposition}

	\subsection{A presentable universal property for \texorpdfstring{\for{toc}{$\Qq$}\except{toc}{$\bm\Qq$}}{Q}-commutative monoids}
	If $\Dd$ has $\Qq$-limits, then Theorem~\ref{thm:universal_prop_par_Qcom_monoids2} identifies (certain) functors \textit{into} $\ul\CMon^\Qq(\Dd)$; the goal of this subsection is to give (under a mild smallness assumption) a similar property for functors \textit{out of} $\ul\CMon^\Qq(\Dd)$ whenever $\Dd$ is a presentable $\Bb$-category. For this we first recall:

	\begin{definition}
		A $\Bb$-category $\Cc\colon\Bb^\op\to\Cat$ is called \emph{presentable} if it is $\Bb$-cocomplete and factors through the non-full subcategory $\Pr^\textup{L}\subseteq\Cat$, i.e.~each $\Cc(A)$ is presentable and each $f^*\colon\Cc(B)\to\Cc(A)$ is a left adjoint.

		We write $\Pr(\Bb)^\textup{L}$ for the category of presentable $\Bb$-categories and left adjoint functors, and $\Pr(\Bb)^\textup{R}$ for the category of presentable $\Bb$-categories and right adjoint functors.
	\end{definition}

	\begin{remark}
		\cite{martiniwolf2022presentable} instead defines presentable $\Bb$-categories as accessible Bousfield localizations of presheaf categories. This is equivalent to the above definition by Theorem~6.2.4 of \emph{op.\ cit.}
	\end{remark}

	\begin{remark}
		By definition, every presentable $\Bb$-category is in particular cocomplete. Moreover, an easy application of the non-parametrized Special Adjoint Functor Theorem shows that every presentable $\Bb$-category is also complete, also see \cite{martiniwolf2022presentable}*{Corollary~6.2.5}.
	\end{remark}

	Because we have not bounded the size of $\Qq$, the $\Bb$-category $\ul\CMon^\Qq(\Cc)$ does not necessarily have to be presentable, even if $\Cc$ was presentable. To fix this, we introduce:

	\begin{definition}
		We say that a wide local subcategory $\Qq\subseteq\Bb$ is \emph{slicewise small} if the category $\Qq_{/A}$ is (essentially) small for every $A\in\Bb$, i.e.~if $\ulbbU{\Qq}$ is a small $\Bb$-category.
	\end{definition}

	\begin{example}
		If $Q\subseteq\PSh(T)$ is a small pre-inductible category, then the corresponding locally inductible subcategory $Q_{\loc}$ of $\PSh(T)$ is slicewise small. This follows immediately from the fact that $\ulbbU{\Qq}$ is the limit extension of the small $T$-category $A\mapsto Q_{/A}$.

		In particular, all examples of locally inductible categories from Section~\ref{subsec:Examples} are slicewise small.
	\end{example}

	\begin{proposition}\label{prop:cmon-in-pres-is-pres}
		Assume that $\Qq$ is slicewise small and let $\Dd$ be presentable. Then the inclusion $\ul\CMon^\Qq(\Dd)\subseteq\ul\Fun_\Bb(\ul\Span(\Qq),\Dd)$ is an accessible Bousfield localization. In particular, $\ul\CMon^\Qq(\Dd)$ is again presentable.
		\begin{proof}
			As seen in the proof of Lemma~\ref{lemma:U-pres-refl-lim}, $\ul\CMon^\Qq(\Dd)$ is complete and the inclusion is continuous. It will therefore suffice to show that we have an accessible Bousfield localization in each individual level $A\in\Bb$: the pointwise left adjoints will then assemble into a $\Bb$-left adjoint by Remark~\ref{rk:adjoint-via-limits}.

			For this recall the description $\ul\Fun_\Bb(\ul\Span(\Qq),\Dd)(A)\simeq\Fun_{\Bb_{/A}}(\pi_A^*\ul\Span(\Qq),\pi_A^*\Dd)$, under which $\ul\CMon^\Qq(\Dd)(A)$ corresponds precisely to the $\pi_A^*\Qq$-continuous functors (see \Cref{const:coprod_pres_func_cat}). In other words, after replacing $\Bb$ by $\Bb_{/A}$ it will suffice to give a set of maps $S$ such that a functor $F\in\Fun_\Bb(\ul\Span(\Qq),\Dd)$ preserves $\Qq$-limits if and only if it is $S$-local.

			We now observe that since each $\Dd(B)$ for $B\in\Bb$ is presentable, we can find a set $\mathcal T_B\subseteq\Dd(B)$ of objects jointly detecting equivalences. Moreover, observe that for any $X\in\ul\Span(\Qq)(B)$ the evaluation functor $\Fun_{\Bb}(\ul\Span(\Qq),\Dd)\to \Dd(B), F\mapsto F_A(B)$ agrees up to the equivalence from the Yoneda lemma with restriction along the map $\ul B\to \Cc$ classifying $X$, so it is a right adjoint (with left adjoint given by parametrized Kan extension).

			Fix now any map $q\colon A\to B$ in $\Qq$ and $X\in\ul\Span(\Qq)(A)$. Then the comparison map $q_*F_A(X)\to F_Bq_*(X)$ in $\Dd(B)$ is natural in $F$ and hence so is the induced map $\hom(T,q_*F_A(X))\to\hom(T,F_Bq_*(X))$ of spaces for any $T\in \mathcal T_B$. The source and target of this map are corepresentable by the above, so this map agrees by the Yoneda lemma with $\hom(t,F)$ for some suitable map $t=t_{q,X,T}$ in $\Fun_\Bb(\ul\Span(\Qq), \Dd)$. By choice of $\Tt_B$ we see that $q_*F_A\to F_Bq_*$ is an equivalence if and only if $F$ is local with respect to the set of all $t_{q,X,T}$ with $T\in\Tt_B$ and $X$ running through objects of $\ul\Span(\Qq)$ (up to equivalence).

			Pick now a small category $\mathcal M$ together with a left exact localization $L\colon\PSh(\mathcal M)\to\Bb$, yielding a set $\Bb_0\coloneqq L(\mathcal M)$ of objects of $\Bb$ such that every $B\in\Bb$ can be covered by elements of $\Bb_0$. By \Cref{lemma:cocont-local}, we then see that $F$ is $\Qq$-continuous if and only if the Beck--Chevalley map $Fq_*\to q_*F$ is an equivalence for all $q\colon A\to B$ in $\Qq$ such that $B\in\Bb_0$. By choice of $\Bb_0$ and assumption on $\Qq$, there is only a set worth of such maps (up to equivalence). Thus, we may take $S$ to be the \emph{set} of all $t_{q,X,T}$ for such $q$ and for $X$ and $T$ as before.
		\end{proof}
	\end{proposition}

	\begin{corollary}\label{cor:cmon-free-functor}
		Assume that $\Qq$ is slicewise small and let $\Dd$ be presentable. Then the forgetful functor $\mathbb U\colon\ul\CMon^\Qq(\Dd)\to\Dd$ admits a left adjoint $\mathbb P$.
		\begin{proof}
			We may factor $\mathbb U$ as the composite
			\begin{equation*}
				\ul\CMon^\Qq(\Dd)\hookrightarrow\ul\Fun_\Bb(\ul\Span(\Qq),\Dd)\xrightarrow{\ev_\pt}\ul\Fun_\Bb(\ul 1,\Dd)\simeq\Dd.
			\end{equation*}
			The first functor admits a left adjoint by the previous proposition, while the second one admits a left adjoint via parametrized left Kan extension.
		\end{proof}
	\end{corollary}

	\begin{proposition}\label{prop:CMon-Q-presentable-adjunction}
		If $\Qq$ is slicewise small, then the adjunction
		\[
		\textup{incl}\colon\Cat(\Bb)^{\Qq\text-\oplus}\rightleftarrows\Cat(\Bb)^{\Qq\text-\times}\noloc\ul\CMon^\Qq
		\]
		restricts to an adjunction $\Pr(\Bb)^\textup{R,\,$\Qq$-$\oplus$}\rightleftarrows\Pr(\Bb)^\textup{R}$.
		\begin{proof}
			The previous proposition and corollary show that $\ul\CMon^\Qq$ restricts on objects accordingly and that the counit $\mathbb U$ lies in $\Pr^\textup{R}$. Moreover, the unit is even an equivalence as the inclusion is fully faithful, so  it only remains to show that for any adjunction $F\colon\Cc\rightleftarrows\Dd\noloc G$ of presentable $\Bb$-categories, $\ul\CMon^\Qq(G)$ is again a right adjoint. But indeed, the composition of $\ul\Fun(\Span(\Qq),F)$ with the localization $\ul\Fun_\Bb(\ul\Span(\Qq),\Dd)\to\ul\CMon^\Qq(\Dd)$ is easily seen to restrict to the desired left adjoint.
		\end{proof}
	\end{proposition}

	Dualizing we get:

	\begin{corollary}
		If $\Qq$ is slicewise small, the inclusion $\Pr(\Bb)^\textup{L,\,$\Qq$-$\oplus$}\hookrightarrow\Pr(\Bb)^\textup{L}$ admits a left adjoint given on objects by $\Dd\mapsto\ul\CMon^\Qq(\Dd)$ and with unit given by the left adjoints $\mathbb P$ of the forgetful maps.\qed
	\end{corollary}

	The usual argument internalizes this to the following equivalence of $\Bb$-categories:

	\begin{theorem}\label{thm:universal_prop_par_Qcom_monoids_presentable}
		Assume $\Qq$ is slicewise small, let $\Cc$ be presentable, and let $\Dd$ be $\Qq$-semiadditive and presentable. Then restriction along $\mathbb P\colon\Cc\to\ul\CMon^\Qq(\Cc)$ induces an equivalence
		\begin{equation*}
			\ul\Fun^\textup{L}_\Bb(\ul\CMon^\Qq(\Cc),\Dd)\iso\ul\Fun_\Bb^\textup{L}(\Cc,\Dd).\qednow
		\end{equation*}
	\end{theorem}

	\subsection{Fiberwise modules}
	We can also prove a presentable universal property for $\ul{\Sp}^\Qq(\Cc)$ when $\Cc$ is presentable. In fact the only thing relevant about the property of stability is that it is equivalent to being a module over the idempotent object $\Sp$ in $\PrL$. We present the argument in this generality.

	\begin{definition}[See \cite{CSY2021AmbiHeight}*{Definition~5.2.4}]
		A \emph{mode} is an idempotent object in $\Pr^\textup{L}$, i.e.~a pair $(\Ee,E)$ of a presentable category $\Ee$ together with an object $E\in\Ee$ such that the map $\Ee\simeq{\Spc}\otimes\Ee\to\Ee\otimes\Ee$ induced by $E$ is an equivalence.

		Given any mode $(\Ee,E)$, a \emph{module} over it is a presentable category $\Ff$ such that the map $\Ff\simeq{\Spc}\otimes\Ff\to\Ee\otimes\Ff$ induced by $E$ is an equivalence. We write $\Mod_\Ee\subseteq\Pr^\textup{L}$ for the full subcategory of $\Ee$-modules.
	\end{definition}

	As usual, we will just refer to $\Ee$ as a mode when the object $E\in\Ee$ is understood.

	\begin{remark}
		In the above setting, $\Ee$ actually admits a (unique) commutative algebra structure with unit $E$ \cite{HA}*{Proposition 4.8.2.9}, and for this algebra structure the forgetful functor from $\Ee$-modules (in the usual sense) to presentable categories is fully faithful with essential image $\Mod_\Ee$, see \cite{HA}*{Proposition~4.8.2.10}, justifying the terminology.

		As a direct consequence, the inclusion $\Mod_\Ee\hookrightarrow\Pr^\textup{L}$ admits a left adjoint given by $\Ee\otimes\blank$; in particular, $\Mod_\Ee$ is closed under limits.
	\end{remark}

	\begin{remark}\label{rk:tensor-product-lim}
		As a left adjoint, $\Ee\otimes\blank\colon\Pr^\textup{L}\to\Pr^\textup{L}$ preserves all colimits; we will need below that it also preserves certain limits, namely limits of diagrams $X\colon K\to\Pr^\textup{L}$ such that all structure maps $X(k\to\ell)$ are also \emph{right} adjoints. Indeed, on $\Pr^\textup{L}\cap\Pr^\textup{R}$ the functoriality of the Lurie tensor product $\Ee\otimes\Cc=\Fun^\textup{R}(\Ee^\op,\Cc)$ is simply given by postcomposition, so the statement is clear.
	\end{remark}

	\begin{example}
		The pair $(\Spc,1)$ is a mode, and every presentable category is a $\Spc$-module.
	\end{example}

	\begin{example}\label{ex:stable-mode}
		The pair $(\Sp,\mathbb S)$ is a mode, and the $\Sp$-modules are precisely the \emph{stable} presentable categories, see \cite{HA}*{Proposition 4.8.2.18}.
	\end{example}

	\begin{example}
		The pair $(\Set,1)$ is a mode, and the $\Set$-modules are precisely the presentable $1$-categories, see \cite{HA}*{Proposition 4.8.2.15}.
	\end{example}

	\begin{example}\label{ex:Ab-mode}
		The pair $(\Ab,\mathbb Z)$ is a mode, and the $\Ab$-modules are precisely the presentable additive $1$-categories; this is immediate from the previous example together with \cite{groth-gepner-nikolaus}*{Theorem~4.6}.
	\end{example}

	\begin{example}[cf.~\cite{harpaz2020ambidexterity}*{Corollary~5.21}]\label{ex:Q-semiadd-mode}
		Let $\Qq\subseteq\Spc$ be locally inductible. We claim that $(\CMon^\Qq(\Spc),\mathbb P(1))$ is a mode whose modules are precisely the $\Qq$-semiadditive categories.

		Indeed, if $\Dd$ is presentable, then the map $\Dd\to\CMon^\Qq(\Spc)\otimes\Dd$ is left adjoint to the forgetful map $\Fun^\textup{R}(\Dd^\op,\CMon^\Qq(\Spc))\to\Fun^\textup{R}(\Dd^\op,\Spc)\simeq\Dd$, and Corollary~\ref{cor:smashing} shows that this is an equivalence whenever $\Dd$ is $\Qq$-semiadditive and cocomplete (so that $\Dd^\op$ is complete). It follows immediately that $\CMon^\Qq(\Spc)$ is a mode and that every presentable $\Qq$-semiadditive category is a module over it. On the other hand, one easily checks that $\Fun^\textup{R}(\Dd^\op,\CMon^\Qq(\Spc))$ is $\Qq$-semiadditive for every presentable $\Dd$, so conversely every $\CMon^\Qq(\Spc)$-module is $\Qq$-semiadditive.
	\end{example}

	\begin{definition}
		A fiberwise presentable $\Bb$-category $\Cc\colon\Bb^\op\to\Pr^\textup{L}$ is called a \emph{fiberwise $\Ee$-module} if it factors through the full subcategory $\Mod_\Ee\subseteq\Pr^\textup{L}$, i.e.~if every $\Cc(A)$ is an $\Ee$-module. We write $\Mod_\Ee(\Bb)\subseteq\Pr^\textup{L}(\Bb)$ for the full subcategory spanned by those presentable categories that are in addition fiberwise $\Ee$-modules.
	\end{definition}

	\begin{example}\label{ex:fw-stable-mode}
		By Example~\ref{ex:stable-mode}, a (fiberwise) presentable $\Bb$-category is fiberwise stable if and only if it is a fiberwise $\Sp$-module.
	\end{example}

	\begin{example}\label{ex:qfib-mode}
		By \Cref{lem:fiberwise-semiadd} and \Cref{ex:Q-semiadd-mode}, a (fiberwise) presentable $\Qq$-semiadditive $\Bb$-category is always a fiberwise $\CMon^{Q_{\fib}}(\Spc)$-module.
	\end{example}

	\begin{lemma}\label{lemma:fiberwise-module}
		Let $\Ee$ be any mode. Then the inclusion $\Mod_\Ee(\Bb)\subseteq\Pr^\textup{L}(\Bb)$ admits a left adjoint given by applying $\Ee\otimes\blank$ pointwise, with unit $\Cc\to\Ee\otimes\Cc$ induced by the map $\Spc\to\Ee$.

		Moreover, this restricts to an adjunction $\Mod_\Ee^{\Qq\text-\oplus}(\Bb)\rightleftarrows\Pr^{\textup{L},\,\Qq\text-\oplus}(\Bb)$.
		\begin{proof}
			First observe that the pointwise tensor product $\Ee\otimes\Cc$ is indeed a $\Bb$-category for any presentable $\Bb$-category $\Cc$ by Remark~\ref{rk:tensor-product-lim}. Next, let us show that $\Ee\otimes\Cc$ is again $\Bb$-cocomplete, hence presentable. For this we recall that the tensor product lifts to an $(\infty,2)$-functor; all we will need below is that it is a functor on the homotopy $2$-category of $\Pr^\textup{L}$ (which also follows immediately from its construction as a functor category), and hence sends adjunctions to adjunctions. Given now any $f\colon A\to B$ in $\Bb$, both $f^*$ and $f_!$ are left adjoints, so they form an internal adjunction in $\Pr^\textup{L}$ and hence induce an adjunction $\Ee\otimes f_!\dashv\Ee\otimes f^*$. Given any $g\colon B'\to B$, the base change map $(\Ee\otimes f'_!)(\Ee\otimes g^{\prime*})\to(\Ee\otimes g^*)(\Ee\otimes f_!)$ is then induced via $2$-functoriality from the base change map $f'_!g^{\prime*}\to g^*f_!$, so it is invertible as the latter one is. Finally, the same $2$-functoriality argument together with Proposition~\ref{prop:adj-criterion-MW} shows that the canonical map $\Cc\to\Ee\otimes\Cc$ is indeed a parametrized left adjoint.

			If $\Cc$ is now $\Qq$-semiadditive, then the functor $q_*$ is itself a map in $\Pr^\textup{L}$ for any $q\colon A\to B$ in $\Qq$ (as $q_*\simeq q_!$), so the right adjoint of $\Ee\otimes q^*$ is given by $\Ee\otimes q_*$, with the induced unit and counit. Arguing as before, we see that the double Beck--Chevalley map $(\Ee\otimes\pr_{1!})(\Ee\otimes q_*)\to(\Ee\otimes\pr_{2*})(\Ee\otimes q_!)$ is invertible, so that $\Ee\otimes\Cc$ is $\Qq$-semiadditive, proving the second statement.
		\end{proof}
	\end{lemma}

	Let us specialize this to the stable case (Example~\ref{ex:fw-stable-mode}):

	\begin{corollary}\label{cor:Q-stable-presentable}
		Assume $\Qq$ is slicewise small. The inclusion $\Pr(\Bb)^\textup{L,\,$\Qq$-ex}\hookrightarrow\Pr(\Bb)^\textup{L}$ admits a left adjoint $\ul\Sp^\Qq\coloneqq{\Sp}\otimes\ul\CMon^\Qq$. For every presentable $\Cc$, the unit $\Sigma^\infty_+\colon\Cc\to\ul\Sp^\Qq(\Cc)$ is given by the composite
		\begin{equation*}
			\Cc\xrightarrow{\;\mathbb P\;}\ul\CMon^\Qq(\Cc)\xrightarrow{\,\Sigma^\infty_+\,}{\Sp}\otimes\ul\CMon^\Qq(\Cc)=\ul\Sp^\Qq(\Cc).
		\end{equation*}
		\begin{proof}
			Combine Proposition~\ref{prop:CMon-Q-presentable-adjunction} with Lemma~\ref{lemma:fiberwise-module}.
		\end{proof}
	\end{corollary}

	As before, we formally deduce the following internal version:

	\begin{theorem}
		Assume $\Qq$ is slicewise small. Let $\Cc$ be any presentable $\Bb$-category, and let $\Dd$ be presentable and $\Qq$-stable. Then restriction along $\Sigma^\infty_+\colon\Cc\to\ul\Sp^\Qq(\Cc)$ induces an equivalence
		\begin{equation*}
			\ul\Fun_\Bb^\textup{L}(\ul\Sp^\Qq(\Cc),\Dd)\iso\ul\Fun_\Bb^\textup{L}(\Cc,\Dd).\qednow
		\end{equation*}
	\end{theorem}

	\section{Sheaves with transfers}
	Given a locally inductible subcategory $\Qq \subseteq \Bb$, our definition of $\Qq$-commutative monoids in a general $\Bb$-category $\Cc$ makes heavy use of the language of parametrized category theory. In this section, we will see that in suitable cases $\Qq$-commutative monoids in $\Cc$ admit a concrete non-parametrized description in terms of \textit{sheaves with transfers}:

	\begin{definition}
		 Let $\Ee$ be a presentable category. Recall that an $\Ee$-valued sheaf on $\Bb$ is a limit-preserving functor $\Bb\catop \to \Ee$. We define an  \textit{$\Ee$-valued sheaf with $\Qq$-transfers on $\Bb$} to be a functor $M\colon \Span(\Bb,\Bb,\Qq) \to \Ee$ for which the restriction
		\[
		\Bb\catop \simeq \Span(\Bb,\Bb,\iota\Bb) \hookrightarrow \Span(\Bb,\Bb,\Qq) \xrightarrow{\;M\;} \Ee
		\]
		is an $\Ee$-valued sheaf on $\Bb$. We denote by
		\[
			\Shv(\Bb;\Ee) \subseteq \Fun(\Bb\catop,\Ee) \qquad \text{ and } \qquad \Shv^{\Qq}(\Bb;\Ee) \subseteq \Fun(\Span(\Bb,\Bb,\Qq),\Ee)
		\]
		the resulting full subcategories of sheaves and sheaves with transfers, respectively.
	\end{definition}

	\begin{remark}
	Suppose $\Bb = \Shv_{\tau}(\Aa)$ is the sheaf topos on some subcanonical Grothendieck site $(\Aa,\tau)$. Suppose further that for all maps $A\rightarrow B$ in $\Qq$ such that $B\in \Aa$, also $A\in \Aa$. Then the category $\Shv^{\Qq}(\Bb;\Ee)$ is equivalent to the subcategory $\Shv^{\Qq}_{\tau}(\Aa;\Ee)$ of $\Fun(\Span(\Aa,\Aa,\Aa \cap \Qq),\Ee)$ spanned by those functors whose underlying functor $\Aa\catop \to \Ee$ is a $\tau$-sheaf, see \Cref{prop:Smaller_Spans} below.

	In particular, when $\Bb = \PSh_{\Sigma}(\Aa)$ is the sifted cocompletion of an extensive category $\Aa$, the notion of sheaves with transfers reduces to that of a \textit{Mackey functor}; this special case will be discussed in more detail in \Cref{subsec:Mackey_Functors}.
	\end{remark}

	\begin{construction}
		For every $A\in\Bb$, write $\Bb_{/A}[\Qq]\coloneqq \Bb_{/A}\times_{\Bb}\Qq$ for the wide subcategory of $\Bb_{/A}$ given by those maps over $A$ whose underlying map in $\Bb$ belongs to $\Qq$. Then the assignment $A\mapsto(\Bb_{/A},\Bb_{/A},\Bb_{/A}[\Qq])$ extends to a functor $\Bb\to\AdTrip$ via pushforward, hence giving rise to a $\Bb$-presheaf of categories
		\[
		A\mapsto\Fun(\Span(\Bb_{/A},\Bb_{/A},\Bb_{/A}[\Qq]),\Ee).
		\]
		We denote by $\ul\Shv^{\Qq}(\Bb;\Ee)$ the full subfunctor given in degree $A$ by the subcategory $\Shv^{\Bb_{/A}[\Qq]}(\Bb_{/A};\Ee)$. In a similar fashion, we obtain a $\Bb$-presheaf $\ul\Shv(\Bb;\Ee)$ and the inclusions $(\Bb_{/A},\Bb_{/A},\core\Bb_{/A})\hookrightarrow(\Bb_{/A},\Bb_{/A},\Bb_{/A}[\Qq])$ induce a forgetful map $\mathbb U\colon\ul\Shv^{\Qq}(\Bb;\Ee)\to\ul\Shv(\Bb;\Ee)$ of $\Bb$-presheaves.

		Since $\Shv(\Bb_{/B};\Ee) \simeq \Shv(\Bb_{/B}) \otimes \Ee$ in $\PrL$, the functor $\ul{\Shv}(\Bb;\Ee)\colon \Bb\catop \to \Cat$ agrees with the $\Bb$-category $\Omega_{\Bb} \otimes \Ee$ that \cite{martiniwolf2022presentable}*{Section~8.3} associate to the presentable category $\Ee$.
	\end{construction}

	The main result of this section is the following:

	\begin{theorem}\label{thm:span-BBQ}
		Let $\Ee$ be a presentable category. Then there is a unique equivalence
		\[
		\ul\Shv^{\Qq}(\Bb;\Ee) \iso \ul\CMon^\Qq(\ul\Shv(\Bb;\Ee))
		\]
		of $\Bb$-presheaves over $\ul\Shv(\Bb;\Ee)$. In particular, $\ul\Shv^{\Qq}(\Bb;\Ee)$ is a $\Bb$-category.
	\end{theorem}

	Before moving to the proof of the theorem, let us explain how it allows us to describe the free presentable $\Qq$-semiadditive and $\Qq$-stable $\Bb$-categories in non-parametrized terms:

	\begin{corollary}\label{cor:free-semiadditive-Mackey}
		Assume $\Qq$ is slicewise small. Then the $\Bb$-category $\ul\Shv^{\Qq}(\Bb;\Spc)$ is the free presentable $\Qq$-semiadditive $\Bb$-category, i.e.~for every other presentable $\Qq$-semiadditive $\Dd$ evaluation at a certain `free sheaf with transfers' $\mathbb P(1)\in\Shv^{\Qq}(\Bb;\Spc)$ induces an equivalence
		\begin{equation*}
			\ul\Fun^\textup{L}(\ul\Shv^{\Qq}(\Bb;\Spc),\Dd)\iso\Dd.
		\end{equation*}
	\end{corollary}

	We will give a concrete description of $\mathbb P(1)$ in Corollary~\ref{cor:free-is-corep}.

	\begin{proof}
		Note that for $\Ee = \Spc$ we have $\ul\Shv(\Bb;\Ee) = \ul{\Spc}_{\Bb}$. The claim thus follows from the string of equivalences
		\[
			\ul\Fun^\textup{L}(\ul\Shv^{\Qq}(\Bb;\Spc),\Dd)\overset{\ref{thm:span-BBQ}}{\simeq} \ul\Fun^\textup{L}_\Bb(\ul\CMon^\Qq(\ul{\Spc}_{\Bb}),\Dd) \overset{\ref{thm:universal_prop_par_Qcom_monoids_presentable}}{\simeq} \ul\Fun_\Bb^\textup{L}(\ul{\Spc}_{\Bb},\Dd) \iso \Dd,
		\]
		where the last equivalence holds by \cite{martiniwolf2021limits}*{Theorem~7.1.1}.
	\end{proof}

	\begin{corollary}\label{cor:free-stable-Mackey}
	Assume $\Qq$ is slicewise small. Then the $\Bb$-category $\ul\Shv^{\Qq}(\Bb;\Sp)$ is the free presentable $\Qq$-stable $\Bb$-category, i.e.~for any other such $\Dd$ evaluation at a certain global section $\mathbb S\in\Shv^{\Qq}(\Bb;\Sp)$ induces an equivalence
	\begin{equation*}
		\ul\Fun^\textup{L}(\ul\Shv^{\Qq}(\Bb;\Sp),\Dd)\iso\Dd.
	\end{equation*}
	\begin{proof}
		Combining Corollaries~\ref{cor:Q-stable-presentable} and~\ref{cor:free-semiadditive-Mackey}, the free presentable $\Qq$-stable category is given by ${\Sp}\otimes\ul\Shv^{\Qq}(\Bb;\Spc)$. Using the explicit description of ${\Sp}\otimes\blank$ as $\Fun^{\textup{R}}(\Sp^\op,\blank)$, this is immediately seen to be equivalent to $\ul\Shv^{\Qq}(\Bb;\Sp)$.
	\end{proof}
	\end{corollary}

	\begin{remark}\label{rk:mackey-mode-univ-prop}
		More generally, if $\Ee$ is any mode, then \Cref{lemma:fiberwise-module} shows that $\ul\Shv^{\Qq}(\Bb;\Spc)\otimes\Ee\simeq\ul\Shv^{\Qq}(\Bb;\Ee)$ is the free presentable $\Qq$-semiadditive fiberwise $\Ee$-module.
	\end{remark}

	\begin{remark}[Fiberwise semiadditivity, redux]\label{rk:fiberwise-mackey}
		Let $\Ff\subseteq\Spc$ be locally inductible such that the unique left exact left adjoint $\Spc\to\Bb$ maps $\Ff$ into $\Qq$; for example, we could take the maximal such class $\Qq_{\fib}\subseteq\Spc$ from Lemma~\ref{lem:fiberwise-semiadd}. Combining said lemma with Example~\ref{ex:Q-semiadd-mode}, we see that every presentable $\Qq$-semiadditive category $\Cc$ is a fiberwise $\CMon^{\Ff}(\Spc)$-module, so that $\CMon^{\Ff}(\Spc)\otimes\Cc\iso\Cc$. Specializing $\Cc$, we obtain an equivalence $\ul\Shv^{\Qq}(\Bb;\CMon^{\Ff}(\Ee))\iso\ul\Shv^{\Qq}(\Bb;\Ee)$.

		In particular, we see that if $\Qq$ contains the map $1\amalg 1\to 1$, then the free presentable $\Qq$-semiadditive $\Bb$-category can be equivalently described as $\ul\Shv^\Qq(\Bb;\CMon)$.
	\end{remark}

	We also record the following useful corollary:

	\begin{corollary}
	Given a $\Qq$-semiadditive $\Bb$-category $\Dd$, every $\Qq$-continuous functor $\Dd \to \ul\Shv(\Bb;\Ee)$ admits a unique lift to $\ul\Shv^\Qq(\Bb;\Ee)$: the forgetful functor
	\[
		\ul\Fun^{\Qprod}(\Dd, \ul\Shv^{\Qq}(\Bb;\Ee)) \to \ul\Fun^{\Qprod}(\Dd,\ul\Shv(\Bb;\Ee))
	\]
	is an equivalence.
	\end{corollary}
	\begin{proof}
	This is immediate from \Cref{thm:span-BBQ} and the universal property of $\ul{\CMon}^\Qq$ from \Cref{thm:universal_prop_par_Qcom_monoids2}.
	\end{proof}

	In particular, every $\Qq$-semiadditive $\Bb$-category admits a canonically enrichment in sheaves with transfers:

	\begin{corollary}
		\label{cor:Unique_Mackey_Functor_Enrichment}
		For every $\Qq$-semiadditive $\Bb$-category $\Dd$, there is a unique functor \[\ul\Hom^{\Qq}_{\Dd}\colon \Dd\catop \times \Dd \to \ul\Shv^\Qq(\Bb;\Spc)\] that lifts the parametrized hom-functor $\ul\Hom_{\Dd}\colon \Dd\catop \times \Dd \to \ul{\Spc}_{\Bb}$. \qed
	\end{corollary}

	\subsection{Reduction to presheaves with transfers} We will now move towards the proof of \Cref{thm:span-BBQ}. It will be technically convenient to deduce the theorem from a more general result about presheaves with transfers. Throughout, let $\Aa$ be any category, not necessarily small, and let $\Qq\subseteq\Aa$ be left-cancelable and closed under base change.

	\begin{definition}
		For any category $\Ee$, we write $\ul\PSh^\Qq(\Aa;\Ee)$ for the functor \[\Aa^\op\to\Cat, A\mapsto\Fun(\Span(\Aa_{/A},\Aa_{/A},\Aa_{/A}[\Qq]),\Ee)\] with functoriality via pushforward. We will refer to global sections of this as \emph{presheaves with $\Qq$-transfers}.

		We further define $\ul\PSh(\Aa;\Ee)\coloneqq\ul\PSh^{\iota\Aa}(\Aa;\Ee)\colon A\mapsto\Fun((\Aa_{/A})^\op,\Ee)$, and we write $\mathbb U\colon\ul\PSh^\Qq(\Aa;\Ee)\to\ul\PSh(\Aa;\Ee)$ for the evident forgetful map.
	\end{definition}

	\begin{remark}
		If $\Aa=T$ is small, the category $\ul\PSh(\Aa;\Ee)$ is denoted $\ul\Ee_T$ and called the \emph{category of $T$-objects} in \cite{CLL_Global}*{Example~2.1.11}.
	\end{remark}

	\begin{lemma}
		The $\Aa$-category $\ul\PSh(\Aa;\Ee)$ has all $\Qq$-limits.
		\begin{proof}
			Let $q\colon A\to B$ be any map in $\Qq$. As $\Qq$ is closed under base change, $\Aa_{/q}\colon\Aa_{/A}\to\Aa_{/B}$ has a right adjoint given by pullback along $q$, and this satisfies base change with respect to pushforward along maps in $\Aa$ by the pasting law for pullbacks. The claim now follows simply from $2$-functoriality of $\PSh$.
		\end{proof}
	\end{lemma}

	\begin{construction}\label{constr:cartesian-fibration-UQ}
		If $\Aa$ has all pullbacks, then the target map $t\colon\Ar(\Aa)\to\Aa$ is a cartesian fibration classifying the functor $\Aa^\op\to\Cat,A\mapsto\Aa_{/A}$. If $\Qq\subseteq\Aa$ is closed under base change, then this restricts to a cartesian fibration $t\colon\Ar(\Aa)^\Qq\to\Aa$ where the source denotes the full subcategory spanned by the maps in $\Qq$; this then classifies the functor $\ulbbU{\Qq}\colon A\mapsto\Qq_{/A}$ considered before if $\Aa=\Bb$ is a topos.

		If now $\Aa$ is arbitrary, then embedding it in a pullback-preserving way into a category with all pullbacks, we see that $t\colon\Ar(\Aa)^\Qq\to\Aa$ is still a cartesian fibration. We denote the straightening $\Aa^\op\to\Cat$ again by $\ulbbU{\Qq}\colon A\mapsto\Qq_{/A}$; note that this agrees with the previous functor of the same name if $\Aa=\Bb$ is a topos. Taking spans levelwise then as before gives us a functor $\ul\Span(\Qq)\colon A\mapsto\Span(\Qq_{/A})$.
	\end{construction}

	\begin{proposition}\label{prop:mackey-presheaves}
		There is an equivalence \[\Phi\colon\ul\PSh^\Qq(\Aa;\Ee)\iso\ul\Fun_\Aa^{\Qq\text-\times}\big(\ul\Span(\Qq),\ul\PSh(\Aa;\Ee)\big)\] of $\Aa$-categories over $\ul\PSh(\Aa;\Ee)$.
	\end{proposition}

	The proof of the proposition will take up the next two subsections; for now let us record that it immediately implies the theorem:

	\begin{proof}[Proof of Theorem~\ref{thm:span-BBQ}, assuming Proposition~\ref{prop:mackey-presheaves}]
		Applying the previous proposition for $\Aa=\Bb$ (and using the comparison of internal homs from the proof of Proposition~\ref{prop:Yoneda}), it only remains to show that $F\colon\Span(\Bb_{/A},\Bb_{/A},\Bb_{/A}[\Qq])\to\Ee$ restricts to a sheaf $(\Bb_{/A})^\op\to\Ee$ if and only if $\Phi(F)\colon\pi_A^*\ul\Span(\Qq)\to\pi_A^*\ul\PSh(\Bb;\Ee)$ factors through $\ul\Shv(\Bb;\Ee)$. As $\Phi$ is a functor over $\ul\PSh(\Bb;\Ee)$, it is clear that $F$ restricts to a sheaf if and only if $\Phi(F)(\id_A)\colon\Bb_{/A}^\op\to\Ee$ is a sheaf. However, in this case we have for any $f\colon B\to A$ in $\Bb$ and any $q\colon C\to B$ in $\Qq$ equivalences \[\Phi(F)(q)\simeq\Phi(F)(q_*(fq)^*\id_A)\simeq q_*(fq)^*\Phi(F)(\id_A)\] by $\Qq$-continuity; the claim follows as $\ul\Shv(\Bb;\Ee)\subseteq\ul\PSh(\Bb;\Ee)$ is closed under $\Qq$-limits by the above description of limits and local cartesian closure of $\Bb$.
	\end{proof}

	\subsection{Comparison of underlying categories}
	Before establishing the full param\-etrized equivalence from \Cref{prop:mackey-presheaves}, we will prove in this subsection that there exists an equivalence on global sections:
	\[
		\Fun(\Span(\Aa,\Aa,\Qq), \Ee) \iso \Fun^{\Qq\text-\times}(\ul\Span(\Qq),\ul\PSh(\Aa;\Ee)).
	\]
	The outline of the proof is as follows:

	\begin{enumerate}
		\item As we will recall below, there is a $1{:}1$-correspondence between $\Aa$-functors $F\colon \ul{\Span}(\Qq)\to\ul\PSh(\Aa;\Ee)$ and non-parametrized functors $\widetilde{F}\colon \smallint\ul{\Span}(\Qq)\to\Ee$ from the total category of the cocartesian unstraightening of $\ul{\Span}(\Qq)$. Following \cite{HHLNa}, we describe this unstraightening $\smallint\ul\Span(\Qq)$ explicitly in terms of certain spans in the arrow category $\Ar(\Aa)$ of $\Aa$.
		\item Next, we prove that a parametrized functor $F\colon\ul\Span(\Qq)\to\ul\PSh(\Aa;\Ee)$ preserves $\Qq$-limits if and only if its associated functor $\widetilde F\colon\smallint\ul\Span(\Qq)\to\Ee$ inverts a certain explicit class of maps $\Ww$, see Proposition~\ref{prop:charact-Q-semiadd}. In particular, $\Qq$-continuous functors $\ul\Span(\Qq)\to\ul\PSh(\Aa;\Ee)$ correspond to functors out of the localization of $\smallint\ul\Span(\Qq)$ at $\Ww$.
		\item Finally, we show in Proposition~\ref{prop:local-BBQ} that this localization is given by the span category $\Span(\Aa,\Aa,\Qq)$.
	\end{enumerate}

	Let us start by making the unstraightening $\smallint\ul\Span(\Qq)$ explicit:

	\begin{proposition}
		Let $\Ar(\Aa)^\Qq\subseteq\Ar(\Aa)$ again denote the full subcategory spanned by the maps of $\Qq$ and write $\Ar(\Aa)^{\Qq,\,\fw}\subseteq\Ar(\Aa)^\Qq$ for the wide subcategory of all maps inverted by $t\colon\Ar(\Aa)^\Qq\to\Aa$.

		Then $(\Ar(\Aa)^\Qq,\Ar(\Aa)^\Qq,\Ar(\Aa)^{\Qq,\,\fw})$ is an adequate triple, and
		\[
			t\colon\Span_{\fw}(\Ar(\Aa)^\Qq)\coloneqq\Span(\Ar(\Aa)^\Qq,\Ar(\Aa)^\Qq,\Ar(\Aa)^{\Qq,\,\fw})\to\Span(\Aa,\Aa,\iota\Aa)\simeq\Aa^\op
		\]
		is a cocartesian fibration classifying the functor $\ul\Span(\Qq)$ from Construction~\ref{constr:cartesian-fibration-UQ}.
	\end{proposition}
	\begin{proof}
		This is an instance of \cite{HHLNa}*{Theorem~3.9}, using that $t\colon\Ar(\Aa)^\Qq\to\Aa$ is (by definition) the cartesian fibration classifying $\ulbbU{\Qq}\colon A\mapsto\Qq_{/A}$.
	\end{proof}

	\begin{lemma}\label{prop:unparametrize-span}
		For every presentable category $\Ee$, there is a natural equivalence
		\[
		\Fun_\Aa(\ul{\Span}(\Qq),\ul\PSh(\Aa;\Ee)) \iso \Fun(\Span_\fw(\Ar(\Aa)^\Qq), \Ee).
		\]
	\end{lemma}
	\begin{proof}
		Applying \cite{CLL_Global}*{Lemma~2.2.13} in a larger universe, there is a natural equivalence
		\[
			\Fun_\Aa(\ul{\Span}(\Qq),\ul\PSh(\Aa;\Ee)) \simeq \Fun(\smallint \ul{\Span}(\Qq), \Ee);
		\]
		The claim is now immediate from the above explicit description of $\smallint\ul\Span(\Qq)$.
	\end{proof}

	Now that we have obtained a description of $\Aa$-functors $\ul{\Span}({\Qq}) \to \ul{\PSh}(\Aa;\Ee)$ as non-parametrized functors out of an explicit span category, we would like to identify which of them correspond to $\Qq$-continuous parametrized functors. This is addressed by the following result:

	\begin{proposition}\label{prop:charact-Q-semiadd}
		Consider an $\Aa$-functor $F\colon \ul{\Span}({\Qq}) \to \ul{\PSh}(\Aa;\Ee)$, and let
		\[
		\widetilde{F}\colon \Span(\Ar(\Aa)^\Qq,\Ar(\Aa)^\Qq,\Ar(\Aa)^{\Qq,\fw}) \to \Ee
		\]
		denote the associated functor from \Cref{prop:unparametrize-span}. Then $F$ preserves $\Qq$-limits if and only if $\widetilde{F}$ inverts the collection $\Ww_s^{\Span}$ of all maps of the form
		\begin{equation}\label{eq:Inverted_spans}
			\begin{tikzcd}
				{C} & C & C \\
				B & { A} & A
				\arrow[Rightarrow, no head, from=1-2, to=1-1]
				\arrow[Rightarrow, no head, from=1-2, to=1-3]
				\arrow["q", from=2-2, to=2-1]
				\arrow["qf"', from=1-1, to=2-1]
				\arrow["f", from=1-3, to=2-3]
				\arrow["f", from=1-2, to=2-2]
				\arrow[Rightarrow, no head, from=2-3, to=2-2]
			\end{tikzcd}
		\end{equation}
		for composable morphisms $f$ and $q$ in $\Qq$.
	\end{proposition}

	The proof relies on the following simple observation:

	\begin{lemma}\label{lemma:source-const-loc}
		Let $\Aa$ be any category and let $\Qq$ be a wide subcategory. Then the source map $s\colon\Ar(\Aa)^\Qq\to\Aa$ is a localization at the class $\Ww_s$ of maps of the form
		\begin{equation}\label{diag:counit-s-const}
			\begin{tikzcd}
				A\arrow[d, "f"']\arrow[r,equals] & A\arrow[d,"qf"]\\
				B\arrow[r, "q"'] & C
			\end{tikzcd}
		\end{equation}
		with $q$ and $f$ in $\Qq$. Moreover, $s$ admits a left adjoint $\const$ given by the inclusion of constant arrows.
		\begin{proof}
			It is clear that the inclusion of constant arrows is left adjoint and right inverse to $s$. To complete the proof it now suffices to observe that $s$ inverts the maps $(\ref{diag:counit-s-const})$ and that the counit $\const\circ s\to\id$ is levelwise of this form.
		\end{proof}
	\end{lemma}

	\begin{proof}[Proof of Proposition~\ref{prop:charact-Q-semiadd}]
		We first recall from \Cref{cor:test_colimit_pres_span} and \Cref{prop:uni_prop_U_Q} that $F$ preserves $\Qq$-limits if and only if its restriction to $\ulbbU{\Qq}^\op$ is right Kan extended from the point. Similarly, the invertibility condition on $\tilde F$ only depends on its restriction to the subfibration $t\colon(\Ar(\Aa)^\Qq)^\op\to \Aa^\op$ classifying $\ulbbU{\Qq}^\op$; by naturality, we are therefore reduced to proving that a functor $F\colon\ulbbU{\Qq}^\op\to\ul\PSh(\Aa;\Ee)$ is right Kan extended from the point if and only if the associated functor $\widetilde F\colon (\Ar(\Aa)^\Qq)^\op\to\Ee$ inverts the maps $(\Ww_s)^\op$ from $(\ref{diag:counit-s-const})$.

		For this let us consider the naturality square
		\begin{equation*}
			\begin{tikzcd}
				\Fun(\ulbbU{\Qq}^\op,\ul\PSh(\Aa;\Ee))\arrow[r, "\sim"]\arrow[d, "\textup{res}"'] & \Fun((\Ar(\Aa)^\Qq)^\op,\Ee)\arrow[d,"\textup{res}=(\const^\op)^*"]\\
				\Fun(\ul1,\ul\PSh(\Aa;\Ee))\arrow[r, "\sim"'] & \Fun(\Aa^\op,\Ee)
			\end{tikzcd}
		\end{equation*}
		associated to the map $\ul 1\to \ulbbU{\Qq}$ classifying the point.

		The horizontal maps are equivalences and the vertical maps admit right adjoints; it then follows formally that the top horizontal map restricts to an equivalence between the essential images of these adjoints. The right adjoint of the vertical arrow on the left is precisely right Kan extension. On the other hand, by Lemma~\ref{lemma:source-const-loc} we have an adjunction $s^\op\dashv\const^\op$, so that the right adjoint of $(\const^\op)^*\colon\Fun((\Ar(\Aa)^\Qq)^\op,\Ee)\to\Fun(\Aa^\op,\Ee)$ is given by $(s^\op)^*$. Appealing to the lemma once more, the essential image of this functor is precisely characterized by the above invertibility condition.
	\end{proof}

	As a consequence of the previous result, a functor $\Span_\fw(\Ar(\Aa)^\Qq)\to\Ee$ preserves $\Qq$-limits if and only if it factors through the localization of $\Span_\fw(\Ar(\Aa)^\Qq)$ at the maps of the form $(\ref{eq:Inverted_spans})$. We will now give an explicit description of this localization:

	\begin{construction}
		Consider the source functor $s \colon \Ar(\Aa)^\Qq \to \Aa$ once more. By left-cancelability of $\Qq$, this maps $\Ar(\Aa)^{\Qq,\,\fw}$ into $\Qq$. As $\Qq$ is closed under base change, this then further shows that the pullback of a map in  $\Ar(\Aa)^{\Qq,\,\fw}$ along a map in $\Ar(\Aa)^\Qq$ is just computed pointwise, so that $s$ preserves all requisite pullbacks. Altogether, we see that $s$ defines a map of adequate triples $(\Ar(\Aa)^{\Qq},\Ar(\Aa)^{\Qq},\Ar(\Aa)^{\Qq,\,\fw}) \to (\Aa,\Aa,\Qq)$, and thus induces a functor
		\begin{equation*}
			s\colon \Span_\fw(\Ar(\Aa)^\Qq) \to \Span(\Aa,\Aa,\Qq).
		\end{equation*}
	\end{construction}

	\begin{proposition}\label{prop:local-BBQ}
		This is a localization at the class of maps $\Ww^{\Span}_s$ from $(\ref{eq:Inverted_spans})$.
		\begin{proof}
			By the localization criterion from \cite{CHLL_Bispans}*{Theorem~4.1.1}, it will be enough to show that $s\colon \Ar(\Aa)^\Qq\to\Aa$ is a localization at the maps $\Ww_s$ and that $s\colon\Ar(\Aa)^{\Qq,\,\textup{fw}}\to\Qq$ is a right fibration. However, the first statement is an instance of Lemma~\ref{lemma:source-const-loc}, while for the second statement it is enough to observe that $\Ar(\Qq)^{\fw}=\Ar(\Aa)^{\Qq,\,\fw}$ consists precisely of the cartesian edges of the cartesian fibration $s\colon\Ar(\Qq)\to \Qq$.
		\end{proof}
	\end{proposition}

	Combining the above results, we can now prove the equivalence from \Cref{prop:mackey-presheaves} on underlying non-parametrized categories:

	\begin{proposition}\label{prop:unparam-mackey-presheaf}
		There is a natural equivalence of non-parametrized categories $\Fun^{\Qq\text-\times}(\ul\Span(\Qq),\ul\PSh(\Aa;\Ee))\simeq \Fun(\Span(\Aa,\Aa,\Qq), \Ee)$.
		\begin{proof}
			Combining \Cref{prop:unparametrize-span} and \Cref{prop:charact-Q-semiadd}, the left hand side is equivalent to the full subcategory $\mathcal F\subseteq\Fun(\Span_\fw(\Ar(\Aa)^\Qq),\Ee)$ spanned by the functors inverting $\Ww^{\Span}_s$. On the other hand, Proposition~\ref{prop:local-BBQ} shows that precomposing with $s$ induces an equivalence between $\Fun(\Span(\Aa,\Aa,\Qq), \Ee)$ and the same $\Ff$.
		\end{proof}
	\end{proposition}

	\subsection{Proof of \Cref{prop:mackey-presheaves}}
	We will now show how one can upgrade the non-parametrized equivalence of Proposition~\ref{prop:unparam-mackey-presheaf} to a parametrized equivalence, yielding a proof of \Cref{prop:mackey-presheaves} and thus completing the proof of Theorem~\ref{thm:span-BBQ}. The basic idea will be to reduce this to the unparametrized statement with $\Aa$ replaced by $\Aa_{/A}$ for all $A \in \Aa$; however, some care has to be taken to get all coherences straight.

	\begin{observation}\label{obs:parametrized-model-grothendieck}
		Let $\Cc$ be an $\Aa$-category and let $\Ee$ be presentable. Combining the categorical Yoneda lemma with \cite{CLL_Global}*{Lemma~2.2.13}, we obtain an equivalence
		\begin{equation*}
			\ul\Fun_\Aa(\Cc,\ul\PSh(\Aa;\Ee))(A)\iso\Fun(\smallint(\Cc\times\ul A),\Ee)=\Fun(\smallint\Cc\times_{\Aa^\op}(\Aa_{/A})^\op,\Ee)
		\end{equation*}
		natural in $\Cc$, $\Ee$, and in $A\in\Aa^\op$.
	\end{observation}

	Below we will apply this to $\Cc=\ul\Span(\Qq)$, in which case we have the same explicit description of the cocartesian unstraightening as before. Let us also describe the resulting pullback explicitly:

	\begin{lemma}\label{lemma:ar-pb}
		Let $\Aa$ be any category and let $A\in\Aa$. Then $(\Ar(\pi_A),t)\colon\Ar(\Aa_{/A})\to\Ar(\Aa)\times_\Aa \Aa_{/A}$ is an equivalence of cartesian fibrations over $\Aa$. Moreover, this can be made natural in $A\in\Aa$ (with respect to the functoriality via postcomposition).
		\begin{proof}
			It is clear that this is a map of cartesian fibrations, so it is enough to show that it underlies an equivalence in $\Fun(\Aa,\Cat_{/\Aa})\simeq\Fun(\Aa,\Cat)_{/\const\,\Aa}$.

			We begin by making some cocartesian unstraightenings explicit. 	The cocartesian unstraightening of $\Aa_{/\bullet}\colon\Aa\to\Cat$ is the fibration $t\colon\Ar(\Aa)=\Fun([1],\Aa)\to\Aa$. As unstraightening commutes with $\Cat$-tensors, it also commutes with $\Cat$-cotensors, so the unstraightening $\Xx\to\Aa$ of $\Ar(\Aa_{/\bullet})$ is given by the cotensor $t^{[1]}$ in $\Cat_{/\Aa}^\textup{cocart}$, i.e.~by the pullback
			\begin{equation*}
				\begin{tikzcd}
					\Xx\arrow[r]\arrow[d]\arrow[dr,phantom,"\lrcorner"{very near start}] & \Fun([1]\times[1],\Aa)\arrow[d, "{(\blank,1)}^*"]\\
					\Aa\arrow[r, "\const"'] & \Fun([1],\Aa)
				\end{tikzcd}
			\end{equation*}
			where $(\blank,1)\colon[1]\to[1]\times[1]$ denotes the map classifying the edge $(0,1)\to(1,1)$. The composite
			\begin{equation*}
				\Xx\to\Fun([1]\times[1],\Aa)\xrightarrow{((-,0)^*,(1,1)^*)}\Fun([1],\Aa)\times\Aa
			\end{equation*}
			then straightens to a natural transformation given pointwise by $\Ar(\pi_A)$, while the target map $\Ar(\Aa_{/\bullet})\to\Aa_{/\bullet}$ unstraightens to the map $\Xx\to\Fun([1],\Aa)$ induced by restricting to the edge $(1,0)\to(1,1)$. Altogether, we get a commutative square of maps of cocartesian fibrations
			\begin{equation*}
				\begin{tikzcd}
					\Xx\arrow[dr, "{(1,1)}^*"{description}]\arrow[dd, "{(1,\blank)}^*"']\arrow[rr,"{((-,0)^*,(1,1)^*)}"] &&[-1em] \Fun([1],\Aa)\times\Aa\arrow[dd, "t\times\Aa"]\arrow[dl, "\pr_2"{description}]\\
					&\Aa
					\\
					\Fun([1],\Aa)\arrow[ur, "t"{description}]\arrow[rr, "{(s,t)}"'] && \arrow[ul, "\pr_2"{description}]\Aa\times\Aa
				\end{tikzcd}
			\end{equation*}
			such that the induced map on pullbacks pointwise straightens to the map $(\Ar(\pi_A),t)$. Moreover, the diagonal composite $\Xx\to\Aa\times\Aa$ straightens to the structure map $\Ar(\Aa_{/\bullet})\to\const\,\Aa$, so it only remains to show that this is a pullback square in $\Cat$.

			By direct inspection, the pullback is given by $\Fun(\Lambda^2_1,\Aa)\simeq\Fun([2],\Aa)$ with the comparison map $\Xx\to\Fun([2],\Aa)$ induced by restriction along the map $f\colon[2]\to[1]\times[1]$ classifying $(0,0)\to(1,0)\to(1,1)$. The claim therefore amounts to saying that $f$ induces an equivalence $[2]\to \big([1]\times[1]\big)\big/\big([1]\times\{1\}\big)$. However, one immediately checks that an inverse equivalence is induced by the map $[1]\times[1]\to[2],(a,b)\mapsto\min\{2,a+42b\}$.
		\end{proof}
	\end{lemma}

	\begin{observation}
		The equivalence from the previous lemma restricts to natural equivalences of cartesian fibrations
		\begin{equation*}
			\Ar(\Aa_{/A})^\Qq\coloneqq\Ar(\Aa_{/A})^{\Aa_{/A}[\Qq]}\iso\Ar(\Aa)^\Qq\times_\Aa\Aa_{/A}
		\end{equation*}
		for all $A\in\Aa$. By direct inspection, this identifies the weak equivalences $\Ww_s\subseteq\Ar(\Aa_{/A})^{\Qq}$ from Lemma~\ref{lemma:source-const-loc} with $\Ww_s\times_\Aa\Aa_{/A}$.

		Similarly, one checks that it restricts to an equivalence
		\begin{equation*}
			\Ar(\Aa_{/A})^{\Qq,\,\fw}\iso\Ar(\Aa)^{\Qq,\,\fw}\times_{\core\Aa}\core(\Aa_{/A}).
		\end{equation*}

		Using that $\Span$ preserves pullbacks of adequate triples, we therefore get a natural commutative diagram
		\begin{equation*}
			\begin{tikzcd}[cramped]
				\big({\Ar(\Aa_{/A})}^{\Qq}\big)^\op\arrow[d,hook]\arrow[r,"\sim"]&\big({\Ar(\Aa)}^\Qq\big)^\op\times_{\Aa^\op}(\Aa_{/A})^\op\arrow[d,hook]\\
				\Span_\fw\big({\Ar(\Aa_{/A})}^{\Qq}\big)\arrow[r, "\sim"'] & \Span_\fw\big({\Ar(\Aa)}^\Qq)^\op\times_{\Aa^\op}(\Aa_{/A})^\op
			\end{tikzcd}
		\end{equation*}
		where the horizontal maps are equivalences of \emph{co}cartesian fibrations, and the lower one identifies $\Ww^{\Span}_s$ with $\Ww^{\Span}_s\times_{\Aa^\op}(\Aa_{/A})^\op$.
	\end{observation}

	Combining this with Observation~\ref{obs:parametrized-model-grothendieck} we get a natural equivalence
	\begin{equation*}
		{\ul\Fun_\Aa}\big({\ul\Span(\Qq)},\ul\PSh(\Aa;\Ee)\big)(A)\iso{\Fun}\big({\Span_\fw}\big({\Ar(\Aa_{/A})}^{\Qq}\big),\Ee\big).
	\end{equation*}

	\begin{corollary}\label{cor:par-Q-semiadd-inv}
		Consider any $F\in\ul\Fun(\ul\Span(\Qq),\ul\PSh(\Aa;\Ee))(A)$, with associated functor
		$\widetilde F\colon\Span_\fw(\Ar(\Aa_{/A})^\Qq)\to\Ee$. Then $F$ belongs to the full subcategory $\ul\Fun^{\Qq\text-\times}(\ul\Span(Q),\ul\PSh(\Aa;\Ee))(A)$ if and only if $\widetilde F$ inverts all maps in $\Ww_s^{\Span}$.
		\begin{proof}
			As in the proof of Proposition~\ref{prop:charact-Q-semiadd}, both conditions only rely on the restriction to backwards arrows. We now have a commutative square
			\begin{equation*}
				\begin{tikzcd}
					\Aa_{/A}\arrow[r,equals]\arrow[d, "\const"'] & \Aa\times_\Aa\Aa_{/A}\arrow[d, "\const\times_\Aa\Aa_{/A}"]\\
					\Ar(\Aa_{/A})^\Qq\arrow[r,"\sim"'] & \Ar(\Aa)^\Qq\times_\Aa\Aa_{/A}
				\end{tikzcd}
			\end{equation*}
			and hence altogether a commutative square
			\begin{equation*}
				\begin{tikzcd}
					\ul\Fun_\Aa\big(\ulbbU{\Qq}^\op,\ul\PSh(\Aa;\Ee)\big)(A)\arrow[d,"\pt^*"']\arrow[r,"\sim"]&\Fun\big((\Ar(\Aa_{/A})^\Qq)^\op,\Ee\big)\arrow[d, "(\const^\op)^*"]\\
					\ul\Fun_\Aa\big(\ul 1,\ul\PSh(\Aa;\Ee)\big)(A)\arrow[r, "\sim"'] & \Fun((\Aa_{/A})^\op,\Ee)\rlap.
				\end{tikzcd}
			\end{equation*}
			By the same formal Beck--Chevalley yoga as before, an object of the top left corner is $\Qq$-continuous if and only if the top horizontal equivalence maps it into the essential image of the right adjoint of the right-hand vertical map. Replacing $\Aa$ by $\Aa_{/A}$, this essential image was identified in the proof of Proposition~\ref{prop:charact-Q-semiadd} as precisely those functors satisfying the above invertibility condition.
		\end{proof}
	\end{corollary}

	\begin{proof}[Proof of Proposition~\ref{prop:mackey-presheaves}]
		By the above, we have a map of $\Aa$-categories
		\[
			\ul\Fun^{\Qq\text-\times}_\Aa\big(\ul\PSh(\Aa;\Ee)\big)\to{\Fun}\big(\Span_\fw(\Ar(\Aa_{/\bullet})^\Qq),\Ee\big)
		\]
		that induces an equivalence onto the full subcategory $\Ff$ spanned in degree $A\in\Aa$ by those functors that invert $\Ww_s^{\Span}$. On the other hand, we have a natural map $s\colon \Span_\fw(\Ar(\Aa_{/\bullet})^\Qq)\to\Span(\Aa_{/\bullet},\Aa_{/\bullet},\Aa_{/\bullet}[\Qq])$, and using Proposition~\ref{prop:local-BBQ} with $\Aa$ replaced by $\Aa_{/A}$ this likewise induces an equivalence onto $\Ff$, yielding an equivalence $\ul\PSh^\Qq(\Aa;\Ee)\simeq\ul\Fun^{\Qq\text-\times}_\Aa(\ul\Span(\Qq),\ul\PSh(\Aa;\Ee))$. It remains to show that this equivalence is compatible with the forgetful functors.

		We will show more generally that our equivalence is compatible with passing to a smaller left-cancelable $\Qq'\subseteq\Qq$ closed under base change. This is clear for restriction along $s$. For the map $\ul\Fun_\Aa(\ul\Span(\Qq),\ul\PSh(\Aa;\Ee))\to\Fun(\Span_\fw(\Ar(\Aa_{/\bullet})^\Qq),\Ee)$ note that this holds for the intermediate composite $\ul\Fun_\Aa(\ul\Span(\Qq),\ul\PSh(\Aa;\Ee))\to\ul\Fun(\Span_\fw(\Ar(\Aa)^\Qq)\times_{\Aa^\op}(\Aa_{/\bullet})^\op,\Ee)$ simply by naturality. Finally, the equivalence
		$
			\Span_\fw(\Ar(\Aa)^\Qq)\times_{\Aa^\op}(\Aa_{/\bullet})^\op\simeq\Span_\fw(\Ar(\Aa_{/\bullet})^\Qq)
		$
		was construced as restriction of a fixed equivalence $\Span(\Ar(\Aa))\times_{\Aa^\op}(\Aa_{/\bullet})^\op\simeq\Span(\Ar(\Aa_{/\bullet})),$ so it is again compatible with passing to a subclass.
	\end{proof}

	\subsection{The free sheaf with transfers}
	\label{subsec:Free_Mackey_Functor}
	For a finite group $G$, an easy application of the Yoneda lemma shows that the free Mackey functor $\Span(\Fin_G)\to\Ab$ is corepresented by the $1$-point set. The analogue holds in our general setting of sheaves with transfers, except that proving that the corepresented functor actually is a sheaf with transfers is not entirely trivial. This will rely on the following computation:

	\begin{lemma}\label{lemma:corep-Mackey}
		Let $\Aa$ be a category with a terminal object $1$, and let $\Qq\subseteq\Aa$ be closed under base change and diagonals. Then the restriction of \[\hom(1,\blank)\colon\Span(\Aa,\Aa,\Qq)\to\Spc\] to $\Aa^\op$ is equivalent to $\iota\ulbbU{\Qq}=\iota\Qq_{/\blank}\colon\Aa^\op\to\Spc$, and this equivalence sends $\id_1\in\hom(1,1)$ to the object $\id_1\in\iota(\Qq_{/1})$.
		\begin{proof}
			We will prove this by computing the cocartesian unstraightening of the restriction of $\hom(1,\blank)$ to $\Aa\catop$, and show it agrees with the unstraightening of $\iota \ulbbU{\Qq}$.
			We first give an explicit description of the forgetful functor $\pi\colon\Span(\Aa,\Aa,\Qq)_{1/}\to\Span(\Aa,\Aa,\Qq)$, which is the unstraightening of $\hom(1,\blank)\colon \Span(\Aa,\Aa,\Qq)\to\Spc$. To this end, consider the adequate triple $Q_{[1]}$ from \cite{HHLNa}*{Lemma~2.5 and Definition~2.16}: the underlying category is the full subcategory of $\Fun(\Lambda^2_0,\Aa)$ spanned by the functors sending $0\to2$ to a map in $\Qq$. The backward maps consist of all diagrams
			\begin{equation*}
				\begin{tikzcd}
					X_1\arrow[d] &\arrow[l] X_0\arrow[r]\arrow[d]\arrow[dr,phantom,"\lrcorner"{very near start}] & X_2\arrow[d]\\
					Y_1 &\arrow[l] Y_0\arrow[r] & Y_2
				\end{tikzcd}
			\end{equation*}
			such that the right-hand square is a pullback. The forward maps are given by those natural transformations that are pointwise in $\Qq$ and for which the \emph{left-hand} square is a pullback.

			By Corollary~2.22 of \emph{op.~cit.}, we may identify the functor $(s,t)\colon\Ar(\Span(\Aa,\Aa,\Qq))\to\Span(\Aa,\Aa,\Qq)^{\times2}$ with the map $(\ev_1,\ev_2)\colon\Span(Q_{[1]})\to\Span(\Aa,\Aa,\Qq)^{\times2}$. Pulling back to $\{1\}$ in the first factor and using that $\Span$ preserves limits, we obtain the following description of $\pi\colon\Span(\Aa,\Aa,\Qq)_{1/}\to\Span(\Aa,\Aa,\Qq)$: the source is the category of spans in $\Ar(\Aa)^\Qq$ of the form
			\begin{equation*}
				\begin{tikzcd}
					Y_0\arrow[d] &\arrow[dl,phantom,"\llcorner"{very near start}]\arrow[l]X_0\arrow[d]\arrow[r,equals] & Z_0\arrow[d]\\
					Y_2 & \arrow[l] X_2\arrow[r, "q"'] & Z_2
				\end{tikzcd}
			\end{equation*}
			where the left-hand square is a pullback and $q$ belongs to $\Qq$ (note that compared to the previous diagram this has been rotated by $\frac32\pi$ radians); the forgetful map is then given by the target map.

			We thus obtain the unstraightening of $\hom(1,\blank)\vert_{\Aa\catop}$ by restricting this forgetful functor to $\Aa\catop\simeq\Span(\Aa,\Aa,\iota\Qq)$ in the target. The resulting functor is the target map $t\catop\colon (\Ar(\Aa)^{\Qq}_{\mathrm{cart}})\catop \to \Aa\catop$, where $\Ar(\Aa)^{\Qq}_{\mathrm{cart}}$ is the wide subcategory of $\Ar(\Aa)^{\Qq}$ spanned by the cartesian squares. As these are precisely the cartesian morphisms for the cartesian fibration $t\colon \Ar(\Aa)^{\Qq} \to \Aa$ classified by $\ulbbU{\Qq}$, we conclude that this resulting functor is indeed the cocartesian unstraightening of $\iota \ulbbU{\Qq}$.

			The claim about the image of $\id_1$ follows immediately by unravelling the above equivalences.
		\end{proof}
	\end{lemma}

	\begin{remark}
		As part of the above proof, we have seen that the cocartesian unstraightening of $\hom(1,\blank)\colon\Span(\Aa,\Aa,\Qq)\to\Spc$ is given by
		\begin{equation}\label{eq:hom1-full}
			\Span(t)\colon\Span(\Ar(\Aa)^\Qq,\Ar(\Aa)^\Qq_{\text{cart}},\Ar(\Aa)^\Qq_{\text{cocart}})\to\Span(\Aa,\Aa,\Qq),
		\end{equation}
		where $\Ar(\Aa)^\Qq_{\text{cocart}}$ denotes the wide subcategory of those maps inverted by the source map (i.e.~the cocartesian edges of $t\colon\Ar(\Qq)\to\Qq$).

		On the other hand, since $\ulbbU{\Qq}\colon\Aa\to\Cat$ is $\Qq$-cocomplete, we can apply Barwick's \emph{unfurling construction} to it: this is the cocartesian fibration
		\[
			\Span(t)\colon\Span(\Ar(\Aa)^\Qq,\Ar(\Aa)^\Qq_{\text{cart}},\Ar(\Qq))\to\Span(\Aa,\Aa,\Qq),
		\]
		whose restriction to $\Aa^\op$ is the cocartesian unstraightening of $\ulbbU{\Qq}$, while its restriction to $\Qq$ encodes the functoriality via the left adjoints to pullback (i.e., postcomposition), see \cite{HHLNa}*{Example~3.4}. By Theorem~3.1 of \emph{op.\ cit.}, $(\ref{eq:hom1-full})$ is precisely the subcategory of cocartesian edges of this fibration; in other words: $\hom(1,\blank)$ as a functor on all of $\Span(\Aa,\Aa,\Qq)$ is obtained from the unfurling of $\ulbbU{\Qq}=\Qq_{/\blank}$ by passing to groupoid cores pointwise. In particular, its covariant functoriality in $\Qq$ is given by postcomposition.
	\end{remark}

	\begin{corollary}
		Let $\Bb$ be a topos and let $\Qq\subseteq\Bb$ be locally inductible. Then $\hom(1,-)\colon\Span(\Bb,\Bb,\Qq)\to\Spc$ is a sheaf with transfers.
		\begin{proof}
			This is immediate from the lemma as $\ulbbU{\Qq}$ is a $\Bb$-category.
		\end{proof}
	\end{corollary}

	\begin{corollary}\label{cor:free-is-corep}
		The free sheaf with transfers $\mathbb P(1)\colon\Span(\Bb,\Bb,\Qq)\to\Spc$ is corepresented by $1$.
		\begin{proof}
			By the non-parametrized Yoneda lemma, $\hom_{\Span}(1,\blank)$ corepresents evaluation at $1$ on the category of all functors $\Span(\Bb,\Bb,\Qq)\to\Spc$. As the same holds true on $\Mack^{\Qq}(\Bb;\Spc)$ for $\mathbb P(1)$ by adjointness, and since $\hom_{\Span}(1,\blank)\in\Mack^{\Qq}(\Bb;\Spc)$ by the previous corollary, the claim follows.
		\end{proof}
	\end{corollary}

	\section{Examples and applications}
	In this section, we discuss various examples and applications of our results. We start in \Cref{subsec:Smaller_Spans} by proving a general result which lets us in many practical situations reduce the amount of data needed to encode sheaves with transfers. In \Cref{subsec:Mackey_Functors} we discuss a special case where sheaves with transfers simplify to Mackey functors. In the remainder of the section we then specialize our results to the contexts of equivariant semiadditivity and `very $G$-semiadditivity,' and discuss various interesting consequences and applications.

	\subsection{Sheaves with transfers on Grothendieck sites}
	\label{subsec:Smaller_Spans}

	In most cases of interest, the information encoded in a sheaf with transfers $F\colon\Span(\Bb,\Bb,\Qq)\to\Ee$ is highly redundant:

	\begin{example}
		If $\Bb$ is the topos of $\infty$-groupoids and $\Qq = \Fin_{\loc}$ is the class of finite covering maps, then the associated category of sheaves with transfers does not immediately recover the definition of the category of commutative monoids as $\Fun^\textup{$\times$}(\Span(\Fin),\Ee)$ but instead describes it in a somewhat bloated way as a subcategory of $\Fun(\Span(\Spc,\Spc,\Qq),\Ee)$.
	\end{example}

	\begin{example}
		For a topos $\Bb$, if we take $\Qq=\core\Bb$ to consist of only the equivalences in $\Bb$, then sheaves with transfers are simply limit-preserving functor $F\colon\Bb^\op\to\Ee$. If $\Bb=\PSh(T)$, such a functor is completely determined by its restriction along the Yoneda embedding. More generally, if $\Bb = \Shv_{\tau}(\Aa)$ is given by sheaves on some Grothendieck site $(\Aa,\tau)$ we may equivalently describe $F$ as a functor $\Aa^\op\to\Ee$ satisfying $\tau$-descent: by \cite{SAG}*{Proposition 1.3.1.7} there is an equivalence
		\[
		\Shv(\Bb;\Ee) \simeq \Shv_{\tau}(\Aa;\Ee).
		\]
	\end{example}

	The goal of this subsection is to generalize the last equivalence to the case for sheaves with non-trivial transfers, as long as $\Qq$ is determined by the intersection $\Aa \cap \Qq$ inside $\Bb = \Shv_{\tau}(\Aa)$.

	\begin{definition}[Sheaves with transfers on a Grothendieck site]
		Let $\Aa$ be a small category equipped with a Grothendieck topology $\tau$, and let $Q \subseteq \Aa$ be a wide $\tau$-local subcategory closed under base change and diagonals. Given a complete category $\Ee$, we define an \textit{$\Ee$-valued sheaf with $Q$-transfers} on $(\Aa,\tau)$ to be a functor $M\colon \Span(\Aa,\Aa,Q) \to \Ee$ whose restriction $M\vert_{\Aa\catop}\colon \Aa\catop \to \Ee$ is a $\tau$-sheaf. The categories of sheaves with $\Qq$-transfers on the slices $(\Aa_{/A},\tau)$ for varying $A\in\Aa$ then assemble into a functor $\ul\Shv^Q_\tau(\Aa;\Ee)\colon\Aa^\op\to\Ee$.
	\end{definition}

	\begin{example}
		If $\Qq\subseteq\Shv_\tau(\Aa)$ is locally inductible, then its preimage in $\Aa$ satisfies the above assumptions.
	\end{example}

	\begin{proposition}
		\label{prop:Smaller_Spans}
		Let $\Qq\subseteq\Bb$ be a locally inductible subcategory. Assume we have a small full subcategory $\Aa\subseteq\Bb$ equipped with a subcanonical Grothendieck topology $\tau$ such that the following conditions are satisfied:
		\begin{enumerate}
			\item The inclusion $\Aa\hookrightarrow\Bb$ extends to an equivalence $\Shv_\tau(\Aa)\iso\Bb$.
			\item The subcategory $\Aa$ is closed under maps in $\Qq$ in the following sense: for every $B\in\Aa$ and $A\to B$ in $\Qq$, also $A\in\Aa$.
		\end{enumerate}
		Then the inclusion $(\Aa,\Aa,\Aa\cap\Qq)\hookrightarrow(\Bb,\Bb,\Qq)$ is a map of adequate triples and the restriction functor $\Fun(\Span(\Bb,\Bb,\Qq),\Ee)\to\Fun(\Span(\Aa,\Aa,\Aa\cap\Qq),\Ee)$ admits a right adjoint, restricting to an adjoint equivalence
		\begin{equation*}
			\Shv^{\Qq}(\Bb;\Ee)\simeq \Shv^{\Aa\cap\Qq}_{\tau}(\Aa;\Ee).
		\end{equation*}
		\begin{proof}
			Consider a pullback square
			\begin{equation*}
				\begin{tikzcd}
					A'\arrow[r, "f'"]\arrow[dr,phantom,"\lrcorner"{very near start}]\arrow[d, "q'"'] & A\arrow[d, "q\in\Qq\cap\Aa"]\\
					B'\arrow[r, "f\in\Aa"'] & B
				\end{tikzcd}
			\end{equation*}
			in $\Bb$ with $f\in\Aa$ and $q\in\Qq\cap\Aa$ as indicated. Since $\Qq$ is closed under base change, $q$ belongs again to $\Qq$, so the second assumption implies that all four objects belong to $\Aa$. It follows immediately that $(\Aa,\Aa,\Aa\cap\Qq)$ is an adequate triple and that $(\Aa,\Aa,\Aa\cap\Qq)\hookrightarrow(\Bb,\Bb,\Qq)$ is indeed a map of adequate triples.

			We now observe that the second assumption on $\Qq$ guarantees that the induced map $\iota\colon\Span(\Aa,\Aa,\Aa\cap\Qq)\to\Span(\Bb,\Bb,\Qq)$ is fully faithful. We claim that we have an adjunction
			\begin{equation*}
				\iota^*\colon\Fun(\Span(\Bb,\Bb,\Qq),\Ee)\rightleftarrows\Fun(\Span(\Aa,\Aa,\Aa\cap\Qq),\Ee)\noloc\iota_*
			\end{equation*}
			with fully faithful right adjoint and such that $\iota_*X|_{\Bb^\op}$ is right Kan extended from $X|_{\Aa^\op}$. Indeed, after embedding $\Ee$ in a limit preserving way into a very large category $\widehat\Ee$ with large limits, this is an instance of Proposition~\ref{prop:restriction-final}${}^\op$ (for $\Bb=\Spc$) as the second assumption on $\Aa$ guarantees that the subcategory $\Span(\Aa,\Aa,\Aa\cap\Qq)\subseteq\Span(\Bb,\Bb,\Qq)$ is compatible with the canonical factorization system. By the above explicit description of $\iota_*X|_{\Bb^\op}$ this right adjoint then actually restricts accordingly.

			It is then clear that $\iota^*$ restricts to a functor $\Shv^{\Qq}(\Bb;\Ee)\to\Shv^{\Aa\cap\Qq}_\tau(\Aa;\Ee)$. On the other hand, appealing to the above description of $\iota_*X|_{\Bb^\op}$ once more shows that $\iota_*$ restricts to an essentially surjective functor in the other direction since a functor $\Bb^\op\to\Ee$ is continuous if and only if it is right Kan extended from an $\Aa$-sheaf by \cite{SAG}*{Proposition 1.3.1.7}.
		\end{proof}
	\end{proposition}

	\begin{corollary}
		\label{cor:Smaller_Spans}
		In the above situation, restriction along $\Aa\hookrightarrow\Bb$ induces an equivalence
		\begin{equation*}
			\ul\Shv^{\Qq}(\Bb;\Ee)|_{\Aa^\op}\iso\ul\Shv^{\Aa\cap\Qq}_\tau(\Aa;\Ee).
		\end{equation*}
		\begin{proof}
			It is clear that the inclusion induces an $\Aa^\op$-natural map $\ul\Shv^{\Qq}(\Bb;\Ee)|_{\Aa^\op}\to\ul\Shv^{\Aa\cap\Qq}_\tau(\Aa;\Ee)$, and the previous proposition with $\Aa$ replaced by $\Aa_{/A}$ for varying $A\in\Aa$ shows that is indeed an equivalence.
		\end{proof}
	\end{corollary}

	\begin{corollary}
		\label{cor:fs-Mackey-sheaf}
		Let $\Aa\subseteq\Bb$ be equipped with a topology $\tau$ satisfying the assumptions of Proposition~\ref{prop:Smaller_Spans}. Identifying $\Bb$-categories with $\tau$-sheaves of categories on $\Aa$, we have:
		\begin{enumerate}[(1)]
			\item The free presentable $\Qq$-semiadditive $\Bb$-category is given by $\ul{\Shv}^{\Aa \cap \Qq}_{\tau}(\Aa;\Spc)$.
			\item The free presentable $\Qq$-stable $\Bb$-category is given by $\ul{\Shv}^{\Aa \cap \Qq}_{\tau}(\Aa;\Sp)$.
		\end{enumerate}
	\end{corollary}

	\begin{proof}
		In light of \Cref{cor:free-semiadditive-Mackey} and \Cref{cor:free-stable-Mackey}, these are direct consequences of the previous corollary applied to the two cases $\Ee = \Spc$ and $\Ee = \Sp$.
	\end{proof}

	\begin{example}[Presheaf topoi]
		\label{ex:preinductible-topology}
		The conditions of \Cref{prop:Smaller_Spans} are in particular satisfied in the case where $\Bb = \PSh(T)$ is a presheaf topos on some small category $T$ and where $\Qq = Q_{\loc}$ is obtained from a small pre-inductible subcategory $Q \subseteq \PSh(T)$. In this case, we let $\Aa \subseteq \PSh(T)$ denote the essential image of the inclusion $Q \hookrightarrow \PSh(T)$, and equip it with the Grothendieck topology $\tau$ in which collection $\{f_i\colon A_i \to A\}$ generates a covering sieve if and only if the map $\bigsqcup_{i \in I} A_i \to A$ is an effective epimorphism in $\PSh(T)$, or equivalently if every morphism $B \to A$ from a representable object $B \in T$ factors through one of the morphisms $f_i$.

		To see that the conditions of \Cref{prop:Smaller_Spans} are satisfied, first note that condition (2) is clear. For condition (1), notice that the image of the full inclusion $T \hookrightarrow \Aa \hookrightarrow \Shv_{\tau}(\Aa)$ consists of completely compact objects which generate $\Shv_{\tau}(\Aa)$ under colimits, so that restriction along this functor defines an equivalence $\Shv_{\tau}(\Aa) \iso \PSh(T)$ by \cite{HTT}*{Corollary~5.1.6.11}. It is clear that this equivalence restricts on $\Aa$ to the inclusion $\Aa \hookrightarrow \PSh(T)$, showing condition (1).
	\end{example}

	\begin{example}[Trivial descent]
		\label{ex:Trivial_Descent}
		As an extreme special case of \Cref{ex:preinductible-topology}, assume that $T$ is a small category equipped with an inductible subcategory $Q \subseteq T$. Then $\Qq = Q_\loc\subseteq \PSh(T)$ and $\Aa = T\subseteq \PSh(T)$, equipped with the trivial Grothendieck topology, satisfy the assumptions of \Cref{cor:fs-Mackey-sheaf}. We conclude that the free presentable $Q$-semiadditive $T$-category $\ul\CMon^Q(\ul{\Spc}_T)$ is given by the functor
		\[
		A\mapsto \Fun(\Span(T_{/A},T_{/A},T_{/A}[Q]), \Spc).
		\]
		Note that the functors are not required to satisfy any sort of descent or limit-preservation condition. In particular, we obtain:
	\end{example}

	\begin{corollary}
		\label{cor:Trivial_Descent_Enrichment}
		Let $T$ be a small category and let $Q \subseteq T$ be an inductible subcategory. For every $Q$-semiadditive $T$-category $\Cc$, there is a unique collection of functors
		\[
		\Hom^Q_{\Cc(A)}\colon \Cc(A)\catop \times \Cc(A) \to \Fun(\Span(T_{/A}, T_{/A}, T_{/A}[Q]), \Spc)
		\]
		which are natural in $A \in T\catop$ and whose underlying functors $\Cc(A)\catop \times \Cc(A) \to \Spc$ given by evaluation at $\id_A$ are the Hom-functors.
	\end{corollary}

	\begin{proof}
		Given the identification of the previous example, the corollary follows immediately from the fact that the parametrized hom functor $\ul{\Hom}_{\Cc}(-,-)\colon {\Cc}\catop \times \Cc \to \ul{\Spc}_T$ uniquely lifts through the functor $\mathbb U\colon \ul\CMon^Q(\ul{\Spc}_T)\to\ul\Spc_T$, which holds by \Cref{cor:Unique_Lift_To_CMon}.
	\end{proof}

	\subsection{Mackey functors}\label{subsec:Mackey_Functors}
	We will now specialize \Cref{prop:Smaller_Spans} to the case where the topos $\Bb$ is given by the sifted cocompletion $\PSh_{\Sigma}(\Aa) = \Fun^{\times}(\Aa\catop, \Spc)$ of an \emph{extensive category} $\Aa$. For this recall:

	\begin{definition}\label{defi:extensive}
		A category $\Aa$ is called \textit{extensive} if it satisfies the following conditions:
		\begin{enumerate}
			\item The category $\Aa$ has finite coproducts.
			\item For all $A,B\in\Aa$ the square
			\[
				\begin{tikzcd}
					\emptyset\arrow[r]\arrow[d] & A\arrow[d]\\
					B\arrow[r] & A\amalg B
				\end{tikzcd}
			\]
			is a pullback square.
			\item For all $n\ge0$ and all families $(A_i'\to A_i)_{i=1}^n$ of maps in $\Aa$, the square
			\[
				\begin{tikzcd}
					A_i'\arrow[r]\arrow[d] & \coprod_{i=1}^nA_i'\arrow[d]\\
					A_i\arrow[r] & \coprod_{i=1}^nA_i
				\end{tikzcd}
			\]
			is a pullback square.
		\end{enumerate}
	\end{definition}

	Every extensive category $\Aa$ admits a canonical Grothendieck topology generated by the finite coproduct decompositions. The corresponding sheaf category $\Shv_{\sqcup}(\Aa)$ agrees with $\PSh_{\Sigma}(\Aa)$ by \cite{BachmannHoyois2021Norms}*{Lemma~2.4}. In particular $\PSh_{\Sigma}(\Aa)$ is a topos, and $\Ee$-valued sheaves on $\PSh_{\Sigma}(\Aa)$ correspond to product-preserving functors $\Aa\catop \to \Ee$:
	\[
		\Shv(\PSh_{\Sigma}(\Aa);\Ee) \simeq \Fun^{\times}(\Aa,\Ee).
	\]

	\begin{definition}
		Let $Q \subseteq \Aa$ be a wide subcategory whose morphisms are closed under base change, disjoint unions, and diagonals. Given a category $\Ee$ admitting finite products, we define an $\Ee$-valued \textit{$Q$-Mackey functor on $\Aa$} to be a functor $\Span(\Aa,\Aa,Q) \to \Ee$ that preserves finite products. We write
		\[
			\Mack^Q(\Aa;\Ee) := \Fun^{\times}(\Span(\Aa,\Aa,Q), \Ee)
		\]
		for the category of $Q$-Mackey functors. The assignment $A \mapsto \Mack^Q(\Aa_{/A};\Ee)$ then defines a functor $\ul\Mack^Q(\Aa;\Ee) \colon \Aa\catop \to \Cat$.
	\end{definition}

	Note that a functor $F\colon \Span(\Aa,\Aa, Q) \to \Ee$ preserves products if and only if its restriction to $\Aa\catop$ preseves finite products, see e.g.\ \cite{CHLL_NRings}*{Proposition~2.2.5}, which is in turn equivalent to the condition that it is a sheaf with respect to the finite coproduct topology. In particular, \Cref{cor:Smaller_Spans} specializes to:

	\begin{corollary}\label{cor:extensive-mackey}
		Let $\Bb = \PSh_{\Sigma}(\Aa)$ for an extensive category $\Aa$, and let $Q \subseteq \Aa$ be a wide subcategory whose morphisms are closed under base change, disjoint unions, and diagonals. Let $\Qq := Q_{\loc}$. Then restriction along the inclusion $\Aa \subseteq \Bb$ defines an equivalence
		\[
			\ul\Shv^{\Qq}(\Bb;\Ee)\vert_{\Aa\catop} \iso \ul\Mack^Q(\Aa;\Ee). \qednow
		\]
	\end{corollary}

	Specializing \Cref{cor:fs-Mackey-sheaf} to this situation, it follows that the free presentable $\Qq$-semiadditive and $\Qq$-stable $\Bb$-categories are given by space-valued and spectrum-valued Mackey functors, respectively.

	As a special case, we may consider a presheaf category $\Bb = \PSh(T)$, so that $\Bb$-categories correspond to $T$-categories. We write $\mathbb F_T$ for the finite coproduct completion of $T$, so that $\Bb \simeq \PSh_{\Sigma}(\mathbb F_T)$. Given a wide subcategory $P \subseteq T$, we further write $\mathbb F^P_T$ for the wide subcategory of $\mathbb F_T$ whose maps are finite coproducts of maps $\coprod_{i=1}^n A_i\to B$ with each $A_i\to B$ in $P$. If $P\subseteq T$ is atomic orbital (see Example~\ref{ex:Finite_P_Sets}) then the subcategory $\mathbb F^P_T$ is pre-inductible, with the corresponding notion of semiadditivity given by $P$-semiadditivity for $T$-categories.

	\begin{corollary}\label{cor:atomic-orbital-Mackey}
		For an atomic orbital subcategory $P \subseteq T$, the free $P$-semiadditive presentable $T$-category is given by
		\[\ul\Mack^{P}_{T}\colon A\mapsto \Fun^{\times}(\Span((\mathbb F_T)_{/A},(\mathbb F_T)_{/A},(\mathbb F_T)_{/A}[\mathbb F^P_T]),\Spc).\]
		More precisely, for every $P$-semiadditive presentable $T$-category $\Dd$, evaluation at $\hom(1,\blank)\colon \Span(\mathbb F_T,\mathbb F_T,\mathbb F_T^P)\to\Spc$ defines an equivalence
		\begin{equation*}
			\ul\Fun_T^\textup{L}(\ul\Mack^P_T,\Dd)\iso\Dd.
		\end{equation*}
		Similarly, the free $P$-stable presentable $T$-category is given by \[A\mapsto  \Fun^{\times}(\Span((\mathbb F_T)_{/A},(\mathbb F_T)_{/A},(\mathbb F_T)_{/A}[\mathbb F^P_T]),\Sp).\]

		More generally, if $P\subseteq T$ is any wide subcategory such that $\mathbb F^P_T$ is pre-inductible, then these define the free presentable $\mathbb F^P_T$-semiadditive and -stable $T$-categories, respectively.
	\end{corollary}

	An independent proof of this corollary (excluding the last sentence) has been given concurrently by P\"utzst\"uck \cite{puetzstueck} using the theory of \emph{algebraic patterns} of \cite{BHS_Algebraic_Patterns}.

	\begin{proof}
		The description of the free $P$-semiadditive presentable $T$-category is the special space $\Ee=\Spc$ of the previous corollary, while the stable statement is the special case $\Ee=\Sp$.

		Finally, Corollary~\ref{cor:free-is-corep} shows that $\hom(1,\blank)$ is the universal element, and so the equivalence 	$\ul\Fun_T^\textup{L}(\ul\Mack^P_T,\Dd)\iso\Dd$ is given by evaluation at $\hom(1,\blank)$ as stated.
	\end{proof}

	\begin{remark}
		We can also describe the universal element of the above model of the free $P$-stable presentable $T$-category as follows:

		As $\mathbb F^P_T$ contains all fold maps $X\amalg X\to X$, Remark~\ref{rk:fiberwise-mackey} implies that $\Mack^P_T(\Spc)$ is equivalent to $\Fun^\oplus(\Span(\mathbb F_T,\mathbb F_T,\mathbb F_T^P), \CMon)$ via the forgetful map.

		As the delooping functor $\CMon\to\Sp$ is left adjoint to the forgetful functor (in particular semiadditive), it induces a functor $\Fun^\oplus(\Span(\mathbb F_T,\mathbb F_T,\mathbb F_T^P), \CMon)\to\Fun^\oplus(\Span(\mathbb F_T,\mathbb F_T,\mathbb F_T^P), \Sp)$ left adjoint to the forgetful functor. By adjointness, this then sends (the lift of) $\hom(1,\blank)$ to the universal element $\mathbb S$; in other words,~$\mathbb S$ is given by pointwise delooping the unique $E_\infty$-monoid structure on $\hom(1,\blank)$.
	\end{remark}

	\subsection{Equivariant and global homotopy theory}
	We will now explain how to use this to prove, in a unified way, Mackey functor descriptions of various categories studied in equivariant homotopy theory.

	Firstly, we can reprove the (by now classical) Mackey functor description of $G$-equivariant spectra \cite{cmnn}*{Theorem~A.1} for a finite group $G$ as well as its refinement to equivariantly commutative monoids recently established by Marc \cite{marc-n-infty}:

	\begin{corollary}\label{cor:equiv-spectra}
		There is an equivalence, natural in $G\in\Orb$, between $\Mack^G\coloneqq\Fun^\oplus(\Span(\mathbb F_G), \CMon)$ and the \hbox{($\infty$-)}category of Shimakawa's $G$-equivariant special $\Gamma$-spaces \cite{shimakawa}.

		Similarly, there is a natural equivalence between genuine $G$-spectra \textup(say, in the incarnation of symmetric $G$-spectra \cite{hausmann-equivariant}\hskip 0pt minus 1pt\textup){\hskip0pt minus 2.5pt} and $\Mack^G(\Sp)\hskip0pt minus .5pt\coloneqq\hskip0pt minus .5pt\Fun^\oplus(\Span(\mathbb F_G), \Sp)$.
		\begin{proof}
			By \cite{CLL_Clefts}*{Theorem~7.17} the categories of $G$-equivariant special $\Gamma$-spaces make up the free presentable equivariantly semiadditive $\Orb$-category. The same holds for the categories $\Mack^G$ by Corollary~\ref{cor:atomic-orbital-Mackey}, proving the first statement.

			The second statement follows similarly from \cite{CLL_Clefts}*{Theorem~9.5}.
		\end{proof}
	\end{corollary}

	Similarly, we obtain Mackey functor models for $G$-global homotopy theory:

	\begin{corollary}\label{cor:global-spectra}
		There exists an equivalence, natural in $G\in\Glo$, between the \hbox{($\infty$-)}category of $G$-global special $\Gamma$-spaces in the sense of \cite{g-global}*{Definition~2.2.50} and $\Fun^\oplus(\Span(\FinGrpd_{/BG},\FinGrpd_{/BG}, \FinGrpd_{/BG}[\FinGrpdfaith]),\CMon)$, where as before $\FinGrpd=\mathbb F_\Glo$ is the $(2,1)$-category of finite $1$-groupoids and $\FinGrpdfaith$ denotes the wide subcategory of faithful functors.

		Similarly, there exists a natural equivalence between the category of $G$-global spectra \cite{g-global}*{Theorem~3.1.40} and $\Fun^\oplus(\Span(\FinGrpd_{/BG},\FinGrpd_{/BG}, \FinGrpd_{/BG}[\FinGrpdfaith]),\Sp)$.
		\begin{proof}
			These follow as before as these models make up the universal  presentable equivariantly semiadditive and equivariantly stable global categories, respectively, by \cite{CLL_Global}*{Theorems~5.3.1 and~7.3.2}.
		\end{proof}
	\end{corollary}

	\begin{remark}
		In the special case $G=1$ the above models recover Schwede's ultra-commutative monoids and global spectra \cite{schwede2018global} for the global family of finite groups. In this setting they first appeared as \cite{global-mackey}*{Theorems~4.22 and~5.17}.
	\end{remark}

	We can further use this to describe the free presentable globally semiadditive and globally stable global categories (Example~\ref{ex:global-semiadditivity}):

	\begin{corollary}
		The assignment $\Mack^\Glo_\Glo\colon G\mapsto \Fun^\oplus(\Span(\mathscr F_{/BG}),\CMon)$ defines the free presentable globally semiadditive global category. Similarly, the free presentable globally stable global category is given by $G\mapsto\Fun^\oplus(\Span(\mathscr F_{/BG}),\Sp)$.\qed
	\end{corollary}

	\begin{remark}\label{rk:1-semiadd-embed}
		Note that compared to the category $\smash{\Mack^{\smash{\Spc_1}}_{\Spc}}$ of $1$-commutative monoids, we have fewer limit conditions in $\Mack^\Glo_\Glo$, i.e.~the two notions do \emph{not} agree. Instead, the above descriptions tell us that $1$-commutative monoids embed fully faithfully into `fully globally commutative monoids' as those objects whose underlying global space is in the image of the fully faithful right adjoint of the forgetful functor $U=\ev_1\colon\PSh(\Glo)\to\Spc$. Such global spaces are called \emph{cofree} in \cite{schwede2018global}*{Definition~1.2.28} or \emph{Borel complete} in \cite{CLL_Adams}.
	\end{remark}

	\begin{remark}
		Compared to the objects of classical global homotopy theory, the above `fully global' versions come with extra structure in the form of `deflations,' additive transfers along surjective group homomorphisms.

		In addition to the non-equivariant examples arising via the previous remark, several interesting ultra-commutative monoids like the infinite orthogonal, unitary, and symplectic groups $\textbf{O}$, $\textbf{U}$, and $\textbf{Sp}$ \cite{schwede2018global}*{Examples~2.3.6, 2.3.7, and~2.3.9}, as well as various global spectra occuring in nature like the sphere, the global algebraic $K$-theory of any $\mathbb Q$-algebra \cite{schwede-k}*{Definition~10.2 and Remark~10.7}, or global complex topological $K$-theory $\textbf{KU}$ \cite{schwede2018global}*{Construction~6.4.9} and its real analogue $\textbf{KO}$ are expected to enhance accordingly, making these fully global categories interesting objects of study. As another example, fully global Mackey functors arising from $K$-theoretic constructions have recently been applied to height 1 chromatic homotopy theory, see \cite{Yuan24Sphere} and \cite{CY23Groth}. We moreover remark that objects of the category $\Mack^\Glo_\Glo(\Ab)$ \big(which can be viewed as a decategorification of $\Mack^\Glo_\Glo(\Sp)$\big), or more generally $\Mack^\Glo_\Glo(\textup{Mod}_R)$ for an ordinary commutative ring $R$, have been well-studied in representation theory under the name \emph{biset functors}, see e.g.~\cite{bouc-biset-book}.
	\end{remark}

	\subsection{Mackey profunctors and quasi-finitely genuine \texorpdfstring{$\bm G$}{G}-spectra}
	As a new application of our results, we obtain universal characterizations for the category $\myMm(G,\Z)$ of \textit{$\Z$-valued $G$-Mackey profunctors} introduced by Kaledin \cite{Kaledin2022Mackey} and the category $\smash{\Sp_G^\textup{qfin}}$ of \textit{quasi-finitely genuine $G$-spectra} of Krause--McCandless--Nikolaus \cite{KMN2023Polygonic}.

	Let $G$ be an arbitrary group. Recall from \Cref{ex:Quasi-finite_G_Sets} the pre-inductible subcategory $\QFin_G \subseteq \PSh({\widehat{\Orb}}_G)$ of quasi-finite $G$-sets.

	\begin{definition}[Mackey profunctors, {cf.\ \cite{Kaledin2022Mackey}*{Definition~3.2}, \cite{KMN2023Polygonic}*{Definition~4.5}}]
		Let $G$ be a group and let $\Ee$ be a presentable category. A functor $M\colon \Span(\QFin_G) \to \Ee$ is called \textit{very additive} if for every quasi-finite $G$-set $S$ the canonical map
		\begin{equation}\label{eq:vsadd-comp-map}
		M(S) \to \prod_{\overline{s} \in S/G} M(\pi^{-1}(\overline{s}))
		\end{equation}
		is an equivalence in $\Ee$, where $\pi\colon S \to S/G$ denotes the quotient map. We write
		\[
			\Mack^{\mathrm{pro}}_G(\Ee) := \Fun^{\mathrm{vadd}}(\Span(\QFin_G),\Ee)
		\]
		for the full subcategory of the functor category spanned by the very additive functors, and refer to its objects as \textit{$\Ee$-valued $G$-Mackey profunctors}. The assignment $G/H \mapsto \Mack^{\mathrm{pro}}_H(\Ee)$ naturally defines a $G$-procategory $\ul{\Mack}^{\mathrm{pro}}_G(\Ee)\colon \smash{\widehat{\Orb}}_G\catop \to \Cat$.
	\end{definition}

	In order to apply our main results to this situation, we have to understand the Grothendieck topology $\tau$ on $\QFin_G$ provided by \Cref{ex:preinductible-topology}.

	\begin{lemma}
		A functor $M\colon \Span(\QFin_G) \to \Ee$ is a $\tau$-sheaf if and only if it is very additive.
	\end{lemma}
	\begin{proof}
		For the `only if'-direction, note that for every quasi-finite $G$-set $S$ the canonical map $\bigsqcup_{\overline{s} \in S/G} \pi^{-1}(\overline{s}) \twoheadrightarrow S$ is a surjection on $H$-fixed points for all finite-index $H \leqslant G$, and thus becomes an effective epimorphism in $\PSh(\smash{\widehat{\Orb}}_G)$. In particular, the inclusions $\{\pi^{-1}(\overline{s}) \hookrightarrow S\}_{\overline{s} \in S/G}$ define a $\tau$-cover. Note that for $\bar s\not=\bar t$ the pullback $\pi^{-1}(\overline{s})\times_S\pi^{-1}(\overline{t})$ is initial; unravelling the definitions, we therefore see that descent with respect to this cover precisely amounts to the map $(\ref{eq:vsadd-comp-map})$ being an equivalence. In particular, every $\tau$-sheaf is very additive.

		Conversely, assume that $M$ is very additive. We have to show that $M$ is a $\tau$-sheaf, or equivalently that its restriction $M' \coloneqq M\vert_{\QFin_G\catop}\colon \QFin_G\catop \to \Ee$ extends to a continuous functor $\PSh(\smash{\widehat{\Orb}}_G)^{\op}\to\Ee$. Define $N \colon \PSh(\smash{\widehat{\Orb}}_G)\catop \to \Ee$ as the limit-extension of the restriction $M'\vert_{\smash{\widehat{\Orb}}_G\catop}\colon \smash{\widehat{\Orb}}_G\catop \to \Ee$. Because $\QFin_G$ is a full subcategory of $\PSh(\widehat{\Orb}_G)$, the restriction of $N$ to $\QFin_G\catop$ is precisely the right Kan extension of $M'\vert_{\smash{\widehat{\Orb}}_G\catop}$ along the inclusion, so that there is a canonical map $M' \to N\vert_{\QFin_G\catop}$ extending the identity on $\widehat{\Orb}_G$. As both sides are very additive (for $N$ by the first paragraph), we see that this an equivalence, finishing the proof.
	\end{proof}

	\begin{corollary}
		For a presentable category $\Ee$ there is a canonical equivalence
		\begin{equation*}
			\ul{\CMon}_{\Bb}^{\QFin_G}(\ul{\Shv}(\Bb;\Ee))\simeq \ul{\Mack}^{\textup{pro}}_G(\Ee),
		\end{equation*}
		where $\Bb := \PSh(\smash{\widehat{\Orb}_G})$.
		\begin{proof}
			Both sides are canonically equivalent to $\ul{\Mack}^{\QFin_G}_{\Bb}(\Ee)$: for the left-hand side this is by \Cref{thm:span-BBQ} while for the right-hand side this is a combination of the previous lemma with \Cref{cor:Smaller_Spans}.
		\end{proof}
	\end{corollary}

	In the case $\Ee = \Sp$, the category $\Mack^{\mathrm{pro}}_G(\Sp)$ is precisely the category $\Sp_G^\textup{qfin}$ of \textit{quasi-finitely genuine $G$-spectra} of Krause--McCandless--Nikolaus \cite{KMN2023Polygonic}*{Definition~4.5}. Corollary~\ref{cor:free-stable-Mackey} therefore specializes to:

	\begin{theorem}\label{thm:qfin}
		The category $\Sp_G^\textup{qfin}$ is the underlying category of the free presentable very $G$-semiadditive stable $G$-procategory.\qed
	\end{theorem}

	On the other hand, for $\Ee = \Ab$ the category $\Mack^{\mathrm{pro}}_G(\Ab)$ is precisely the category $\myMm(G,\Z)$ of \textit{$\Z$-valued Mackey profunctors} introduced by Kaledin \cite{Kaledin2022Mackey}*{Definition~3.2}. Combining Remark~\ref{rk:mackey-mode-univ-prop} with Example~\ref{ex:Ab-mode} we therefore similarly get:

	\begin{theorem}\label{thm:Mackey-pro-fun}
		\label{thm:univ_prop_Kaledins_profunctors}
		The category $\myMm(G,\Z)$ of $G$-Mackey profunctors in abelian groups is the underlying category of the free presentable $1$-truncated very $G$-additive $G$-procategory.\qed
	\end{theorem}

	\appendix

	\section{A criterion for adjoints}
	In this short appendix we will recall a criterion from \cite{martiniwolf2021limits} for the existence of adjoints of parametrized functors and specialize it to a statement about parametrized colimits. We begin with the following characterization:

	\begin{proposition}[See \cite{martiniwolf2021limits}*{Proposition~3.2.9}]\label{prop:adj-criterion-MW}
		A $\Bb$-functor $G\colon\Cc\to\Dd$ admits a left adjoint if and only if the following conditions are satisfied:
		\begin{enumerate}
			\item For each $A\in\Bb$, the functor $G_A\colon\Cc(A)\to\Dd(A)$ admits a left adjoint $F_A$.
			\item For each $f\colon A\to B$ the Beck--Chevalley transformation $F_Af^*\to f^*F_B$ is an equivalence.
		\end{enumerate}
		In this case, the left adjoint $F$ is given at any object $A\in\Bb$ by the pointwise left adjoint $F_A$, and for any morphism $f\colon A\to B$ by the Beck--Chevalley square.\qed
	\end{proposition}

	\begin{remark}\label{rk:adjoint-via-limits}
		If the restriction functor $f^*$ has a right adjoint $f_*$, the second condition is equivalent to demanding that the Beck--Chevalley map $G_Bf_*\to f_*G_A$ be invertible. In particular, if $\Cc$ and $\Dd$ are $\Bb$-complete, then $G$ has a left adjoint if and only if it preserves $\Bb$-limits and each $G_A$ has a left adjoint.
	\end{remark}

	The following proposition allows us to significantly reduce the amount of conditions we have to check:

	\begin{proposition}\label{prop:adj-criterion-cover}
		Let $G\colon\Cc\to\Dd$ be a $\Bb$-functor. Assume there exists a covering sieve $\Sigma\subseteq\Bb$ of the terminal object $1\in\Bb$ such that for every $A\in\Sigma$ the functor $G_A$ admits a left adjoint $F_A$ and such that for every $f\colon A\to B$ in $\Sigma$ the Beck--Chevalley map $F_Af^*\to f^*F_B$ is invertible. Then $G$ admits a left adjoint.
	\end{proposition}

	\begin{proof}
		As $\Sigma$ is a sieve, the assumptions imply via the previous proposition that for every $A\in\Sigma$ the $\Bb_{/A}$-functor $\pi_A^*G\colon\pi_A^*\Cc\to\pi_A^*\Dd$ is a right adjoint. As the objects of $\Sigma$ cover $1 \in \Bb$, \cite{martiniwolf2021limits}*{Remark~3.3.6} then implies that also $G$ itself is a right adjoint as claimed.
	\end{proof}

	\begin{corollary}\label{cor:colimits-local}
		Let $\Qq\subseteq\Bb$ local and let $\Cc$ be any $\Bb$-category. Assume that for every $q\colon A\to B$ there exists a covering sieve $\Sigma\subseteq\Bb_{/B}$ such that for every $(f\colon B'\to B)\in\Sigma$ restriction functor $q^{\prime*}\colon \Cc(B') \to \Cc(A \times_B B')$ along $q' := q^*(f)$ admits a left adjoint $q'_!$, and such that these left adjoint satisfy base change along maps in $\Sigma$. Then $\Cc$ is $\Qq$-cocomplete.
		\begin{proof}
			We have to show that for each $q$ the $\Bb_{/B}$-functor $q^*\colon\pi_B^*\Cc\to\ul\Fun(\ul A,\pi_B^*\Cc)$ admits a left adjoint. This is however simply an instance of the previous proposition (with $\Bb_{/B}$ in place of $\Bb$).
		\end{proof}
	\end{corollary}

	We also note the following result complementing this corollary:

	\begin{lemma}\label{lemma:cocont-local}
		Let $\Qq\subseteq\Bb$ be local and let $F\colon\Cc\to\Dd$ be a functor of $\Qq$-cocomplete $\Bb$-categories.  Assume that for every $q\colon A\to B$ in $\Qq$ there exists a cover $(f_i\colon B_i\to B)_{i\in I}$ (not necessarily a sieve) such that for every $i\in I$ the Beck--Chevalley map $q'_!F_{A\times_BB'_i}\to F_{B'_i}q'_!$ is an equivalence, where $q'=f_i^*(q)$ denotes the pullback of $q$ along $f_i$.	Then $F$ is $\Qq$-cocontinuous.
		\begin{proof}
			Fix $q\colon A\to B$ together with such a covering; we have to show that the Beck--Chevalley map $\BC_!\colon q_!F_A\to F_Bq_!$ is an equivalence. As the $f_i$ form a cover, it will be enough to show that $f_i^*\BC_!$ is an equivalence for every $i\in I$, i.e.~that the pasting
			\begin{equation*}
				\begin{tikzcd}
					\arrow[d,"F"']\Cc(A)\arrow[r, "q_!"] & \Cc(B)\arrow[d, "F"] \arrow[r, "f_i^*"] & \Cc(B_i)\arrow[d, "F"]\\
					\Dd(A)\arrow[ur, Rightarrow, "\BC_!"{description}]\arrow[r, "q_!"'] &\Dd(B)\arrow[r, "f_i^*"'] & \Dd(B_i)
				\end{tikzcd}
			\end{equation*}
			is invertible. However, pasting with the equivalences $f_i^*q_!\to q'_!f_i^*$ coming from $\Qq$-cocompleteness and appealing to the compatibility of mates with pastings this is equivalent to saying that the pasting
			\begin{equation*}
				\begin{tikzcd}
					\Cc(A)\arrow[r, "(A\times_Bf_i)^*"]\arrow[d, "F"'] &[2em]\arrow[d, "F"] \Cc(A\times_BB_i)\arrow[r,"q_!'"] & \Cc(B_i)\arrow[d, "F"]\\
					\Dd(A)\arrow[r, "(A\times_Bf_i)^*"'] & \Dd(A\times_BB_i)\arrow[r, "q'_!"']\arrow[ur,Rightarrow,"\BC_!"{description}] & \Dd(B_i)
				\end{tikzcd}
			\end{equation*}
			is invertible, which holds by assumption on $f_i$.
		\end{proof}
	\end{lemma}

	\bibliography{reference}

\end{document}